\documentclass[10pt]{amsart}
\usepackage{amssymb,amsmath,exscale,latexsym,amsthm,graphics,overpic}
\usepackage{epsfig}    % eps-Bilder einlesen
\usepackage{epic}      % irgendsone Graphikerweiterung
\usepackage{eepic}     % nochne erweiterung
\usepackage[matrix,arrow,curve,cmtip]{xy} % xy-pic, package for commutative
\usepackage{enumerate} % more flexibility for numbered lists
\usepackage{paralist}
\usepackage{subfigure}
 
\theoremstyle{plain}
        \newtheorem{theorem}{Theorem}[section]
        \newtheorem*{theorem*}{Theorem}
        \newtheorem*{conj*}{Conjecture}
        \newtheorem{lemma}[theorem]{Lemma}
        \newtheorem{cor}[theorem]{Corollary}
        \newtheorem{prop}[theorem]{Proposition}
        \newtheorem*{poiriers_theorem*}{Poirier's Theorem}

\theoremstyle{definition}
        \newtheorem{definition}[theorem]{Definition}
        
        \newtheorem{open}{Open Problem}

\theoremstyle{remark}
        \newtheorem*{remark*}{Remark}
        \newtheorem{remark}[theorem]{Remark}

        \newtheorem*{question*}{Question}

        \newtheorem*{claim}{Claim}

        \newtheorem*{case}{Case}

\numberwithin{equation}{section}

\providecommand{\defn}[1]{\emph{#1}}

\newcounter{mylistnum}

\newcommand{\dist}{\operatorname{dist}}

\newcommand{\diam}  {\operatorname{diam}}

\newcommand{\clos}  {\operatorname{clos}}
\newcommand{\inte}  {\operatorname{int}}

\newcommand{\mesh} {\operatorname{mesh}}
\newcommand{\R}{\mathbb{R}}      % reelle Zahlen
\newcommand{\C}{\mathbb{C}}      % komplexe Zahlen
\newcommand{\N}{\mathbb{N}}      % natuerliche Zahlen
\newcommand{\Z}{\mathbb{Z}}      % ganze Zahlen
\newcommand{\Q}{\mathbb{Q}}      % rationale Zahlen
\newcommand{\Halb}{\mathbb{H}}   % obere Halbebene
\newcommand{\CDach}{\widehat{\mathbb{C}}}% Riemannsche Zahlenkugel
\newcommand{\D}{\mathbb{D}}      % Einheitskreis
\newcommand{\Dbar}{\overline{\mathbb{D}}}      % Einheitskreis

\providecommand{\abs}[1]{\lvert#1\rvert}

\providecommand{\norm}[1]{\lVert#1\rVert}

\newcommand{\crit}{\operatorname{crit}}
\newcommand{\post}{\operatorname{post}}

\newcommand{\CC}{\mathcal{C}}

\newcommand{\SC}{\mathcal{S}}

\newcommand{\X} {\mathbf{X}}

\newcommand{\E} {\mathbf{E}}

\newcommand{\V} {\mathbf{V}}

\newcommand{\A} {\mathbf{A}}

\newcommand{\G} {\mathbf{G}}

\newcommand{\simn}{\stackrel{n}{\sim}}
\newcommand{\siminfty}{\stackrel{\infty}{\sim}}
\newcommand{\simhat}{\widehat{\sim}}
\newcommand{\sima}{\stackrel{a}{\sim}}
\newcommand{\simb}{\stackrel{b}{\sim}}
\newcommand{\simalpha}{\stackrel{\alpha}{\sim}}
\newcommand{\Sim}[1]{\stackrel{#1}{\sim}}
\newcommand{\Approx}[1]{\stackrel{#1}{\approx}}

\newcommand{\LC}{\mathcal{L}}

\newcommand{\J}{\mathcal{J}}
\newcommand{\K}{\mathcal{K}}
\newcommand{\F}{\mathcal{F}}

\newcommand{\mate}{\perp \! \! \! \perp}

\newcommand{\wt}{\mathtt{w}}
\newcommand{\bt}{\mathtt{b}}

\begin{document}

\title{Expanding Thurston maps as quotients}

\date{\today} 

\author{Daniel Meyer}
\thanks{The author was partially supported by an NSF postdoctoral
  fellowship and the Academy of Finland (projects SA-11842 and SA-118634)} 
\address{Jacobs University Bremen, Campus Ring 1, 28759 Bremen, Germany}
\email{dmeyermail@gmail.com}  

\subjclass[2000]{Primary: 37F20, Secondary: 37F10}

\keywords{Expanding Thurston map, invariant Peano curve, mating of
  polynomials, critical portraits, fractal tilings}

\begin{abstract}
  A \emph{Thurston map} is a branched covering map
  $f\colon S^2\to S^2$ that is \emph{postcritically finite}. 
  \emph{Mating of polynomials}, introduced by Douady and Hubbard, is a
  method to \emph{geometrically} combine the Julia sets of two
  polynomials (and their dynamics) to form a rational map.
  We show
  that for every \emph{expanding} Thurston map $f$ every sufficiently
  high iterate $F=f^n$ is obtained as the mating of two polynomials. One
  obtains a concise description of $F$ via \emph{critical
    portraits}. The proof is based on the construction of the
  invariant Peano curve from \cite{peano}. 
  As another consequence we obtain a large number 
  of fractal tilings of the plane and the hyperbolic plane. 
\end{abstract}

\maketitle

\tableofcontents

\section{Introduction}
\label{sec:introduction}

Douady and Hubbard observed in the early 80's (see \cite{MR728980})
that the Julia set of certain rational maps ``contains'' the Julia
sets of some polynomials. This prompted them to introduce the notion
of \defn{mating of polynomials}.  This is a method to geometrically
combine the Julia sets of two polynomials and the dynamics defined on
them. Somewhat surprisingly this operation ``often'' (though certainly
not always) results in a rational map. The dynamics of such a rational
map may then be described in terms of the dynamics of polynomials
(which is much better understood). The main result of the present
paper is that for any postcritically finite rational map $f$ whose
Julia set is the whole Riemann sphere, every sufficiently high iterate
$F=f^n$ is obtained as a mating.

An excellent introduction to matings can be found in
\cite{MilnorMating}. 
There are several different ways to define matings (not all of them
equivalent). An overview is given in \cite{mating_defs}. We focus
here on the \defn{topological mating}, the definition of which is given
below. This is possibly the most commonly used notion of mating.

\subsection{The Carath\'{e}odory semi-conjugacy of a polynomial Julia set}
\label{sec:carath-semi-conjugacy-polyn}

Let $P$ be a monic polynomial (i.e., the coefficient of the
leading term is $1$) of degree $d\geq 2$ with connected and locally
connected filled Julia 
set $\K$. \emph{B\"{o}ttcher's theorem} (see  for example \cite[\S
9]{MR2193309} or \cite[II.4]{Carleson}) asserts that $P$ is
conformally conjugate to $z^n$ in a neighborhood of $\infty$, i.e.,
there is a conformal map $\phi\colon 
U\to V$, where $U$, $V$ are neighborhoods of $\infty$, such that
$\phi(\infty)= \infty$ and
\begin{equation}
  \label{eq:boettcher}
  \phi(z^d)= P(\phi(z)),
\end{equation}
for all $z\in U$. We may choose $\phi'(\infty):= \lim_{z\to \infty}
z/\phi(z) >0$ (in fact then $\phi'(\infty)=1$). This makes $\phi$
unique. The conjugacy $\phi$ may be extended conformally to the whole
domain of attraction of $\infty$ (i.e., to the complement of the
filled Julia set $\K$ of $P$). The extended map $\overline{\phi}\colon
\CDach \setminus \Dbar\to \CDach \setminus \K$ then is the Riemann map
of the simply connected domain $\CDach \setminus \K$.

% Let $\phi\colon \CDach\setminus \Dbar\to \CDach\setminus \K$ be
% the Riemann map normalized by $\phi(\infty)=\infty$ and
% $\phi'(\infty)=\lim_{z\to \infty} \phi(z)/z >0$ (in fact then
% $\phi'(\infty)=1$). 
Since the Julia set $\J=\partial \K$ of $P$ is assumed to be locally
connected it follows from Carath\'{e}odory's theorem (see for example
\cite[Theorem 17.14]{Milnor}) that $\overline{\phi}$    
extends continuously to
\begin{equation}
  \label{eq:defCaraSemi-Conjugacy}
  \sigma\colon S^1=\partial \Dbar\to\partial \K=\J, 
\end{equation}
where $\J$ is the \defn{Julia set} of $P$. Since $\sigma$ is the
extension of the conjugacy $\overline{\phi}$ it follows that
$\sigma(z^d)= P(\sigma(z))$ for all $z\in S^1$, i.e., the following
diagram commutes
\begin{equation}
\label{eq:cara_loop}
  \xymatrix{
    S^1 \ar[r]^{z^d} \ar[d]_\sigma & S^1 \ar[d]^\sigma
    \\
    \J \ar[r]^P & \J.
  }
\end{equation}
Note however that $\sigma$ will not be injective in general. 
We call the map $\sigma$ the \defn{Carath\'{e}odory semi-conjugacy} of
$\J$. 

We remind the reader that every postcritically finite polynomial
has connected and locally connected filled Julia set (see for example
\cite[Theorem 19.7]{Milnor}). Thus the description above holds in this
case. 

\subsection{Mating of polynomials}
\label{sec:mating-polynomials}

Consider two monic polynomials $P_{\wt},P_{\bt}$ (called the \emph{white} and
the \emph{black} polynomial) of the same degree $d\geq 2$ with
connected and locally connected Julia sets. 
Let $\sigma_{\wt},\sigma_{\bt}$ be the Carath\'{e}odory semi-conjugacies of
their Julia 
sets $\J_{\wt},\J_{\bt}$. 

Glue the filled Julia sets $\K_{\wt},\K_{\bt}$ (of $P_{\wt},P_{\bt}$) together by
identifying $\sigma_{\wt}(z)\in \partial \K_{\wt}$ with
$\sigma_{\bt}(\bar{z})\in \partial \K_{\bt}$. More precisely, we consider the
disjoint union of $\K_{\wt},\K_{\bt}$, 
and let $\K_{\wt}\mate \K_{\bt}$ be the quotient obtained from the
equivalence relation generated by $\sigma_{\wt}(z)\sim \sigma_{\bt}(\bar{z})$  for all $z\in S^1=\partial
\D$. The complex conjugation $\bar{z}$ is customary here, though not
essential: identifying $\sigma_{\wt}(z)$ with $\sigma_{\bt}({z})$ amounts to the
mating of $P_{\wt}$ with $\overline{P_{\bt}(\bar{z})}$.  The \defn{topological
  mating of $P_{\wt},P_{\bt}$} is the map
\begin{align*}
  &P_{\wt}\mate P_{\bt} \colon \K_{\wt}\mate \K_{\bt} \to \K_{\wt}\mate \K_{\bt},
  \intertext{given by}
  &P_\wt\mate P_\bt|_{\K_i}=P_i,
\end{align*}
for $i=\wt,\bt$. It follows from \eqref{eq:cara_loop} that $x_{\wt}\sim x_{\bt}
\Rightarrow P_{\wt}(x_{\wt})\sim P_{\bt}(x_{\bt})$ (for all $x_{\wt}\in \K_{\wt}, x_{\bt}\in
\K_{\bt}$). This shows that the map $P_{\wt}\mate P_{\bt}$ is well
defined. If a map is topologically conjugate to a map $P_{\wt}\mate P_{\bt}$
we say it is obtained as a (topological) mating. 

Most results obtained so far have focused on the question of when the
mating of two polynomials results in a map that is (topologically
conjugate to) a rational map. This has been completely answered in the
quadratic, postcritically finite case by Rees, Tan, and Shishikura
(see \cite{MR1149864}, \cite{MR1182664}, \cite{MR1765095}). See also
the result on matings of Siegel disk polynomials by Yampolsky-Zakeri
\cite{MR1800348}. Here we consider the question whether a given
rational map is obtained as a mating.

\subsection{Main results}
\label{sec:main-results}

The first main result
of this paper is the following.

\begin{theorem}
  \label{thm:main_rational}
  Let $f\colon \CDach \to \CDach$ be a postcritcally finite rational
  map such that its Julia set is the whole sphere, i.e., $\J(f)=
  \CDach$. Then every sufficiently high iterate $F=f^n$ arises as the
  topological mating of two polynomials. 
\end{theorem}

Note that mating of polynomials may result in maps that are branched
covering maps $f\colon S^2\to S^2$, but are not (topologically
conjugate to) rational maps. A branched covering map $f$ of the sphere
$S^2$ is called a \emph{Thurston map} if it is \defn{postcritically
  finite} (i.e., each critical point has finite orbit). We furthermore
assume that $f$ is \defn{expanding} in a suitable sense (see
Section~\ref{sec:thurston-maps-as-1} for definitions and more
background). Recall that a \defn{periodic critical point} (of a
Thurston map $f$) is a critical point $c$, such that $f^k(c)=c$ for
some $k\geq 1$. A postcritically finite rational map is expanding if
and only if it has no periodic critical points if and only if its
Julia set is the whole sphere. In general however, an expanding
Thurston maps may have periodic critical points.
Theorem~\ref{thm:main_rational} is a special case of the following
more general theorem.

% The following is the second main result of this
% paper. It covers the case when $f$ is a postcritically finite rational
% map, that has the whole sphere as its Julia set.
\begin{theorem}
  \label{thm:mating1}
  Let $f\colon S^2 \to S^2$ be an expanding Thurston map $f$ without
  periodic critical points. Then every sufficiently high iterate $F=f^n$
  is obtained as a topological mating of two polynomials
  $P_{\wt},P_{\bt}$.   
\end{theorem}
The polynomials $P_{\wt},P_{\bt}$ are postcritically finite, where each 
critical point is \emph{strictly} preperiodic (i.e., $P_{\wt},P_{\bt}$ have no
periodic critical points).

\smallskip 
We also prove a version of this theorem in the case when
$f$ is allowed to have periodic critical points. It is relatively easy
to show that in this case no iterate $F=f^n$ can be (i.e., is
topologically conjugate to) the mating of two polynomials. 

It is however possible to slightly alter the mating construction, so
that a result similar to Theorem~\ref{thm:mating1} holds. Namely we
collapse the closure of each bounded Fatou component of $P_{\wt}$ and
$P_{\bt}$ (i.e., each Fatou component distinct from the basins of
attraction of $\infty$). In addition we need to take the
\emph{closure} of the equivalence relation. An equivalence relation
$\sim$ on a compact metric space $S$ is called \defn{closed} if it is
closed as a subset of the product space $S\times S$.  If $\sim$ is not
closed the quotient $S/\!\sim$ fails to be Hausdorff.

\smallskip
Formally we consider
the equivalence relation on the disjoint union of
$\K_{\wt},\K_{\bt}$ generated by the following,
\begin{align*}
  &\sigma_{\wt}(z)\sim \sigma_{\bt}(\bar{z}), \quad \text{for all $z\in S^1$}
  \intertext{and for any bounded Fatou component  $A$ of $P_{\wt}$ or $P_{\bt}$,}
  &x\sim y, \quad \text{for all } x,y\in \clos A.
  % if there exists a bounded Fatou component } \F_{\wt}  x,y \in \clos \F_{\wt} \text{ or } x,y \in \clos
  % \F_{\bt},
\end{align*}
% for all $x,y \in \K_{\wt}$, or $x,y\in \K_{\bt}$. Here $\F_{\wt}$ is a bounded
% component of the Fatou set of $P_{\wt}$ and $\F_{\bt}$ a bounded component of 
% the Fatou set of $P_{\bt}$. 
Since $\sim$ may not be closed in general, we consider the closure
$\simhat$ of $\sim$ (see Lemma~\ref{lem:usc_closure}). Let $\K_{\wt}
\widehat{\mate} \,\K_{\bt}$ be the quotient of (the disjoint union of)
$\K_{\wt},\K_{\bt}$ with $\simhat$. We will show that the maps
$P_{\wt},P_{\bt}$ descend to this quotient, meaning that the following
is well defined.
\begin{align}
  \label{eq:PwPbsimhat}
  &P_{\wt} \,\widehat{\mate} \,P_{\bt} \colon \K_{\wt} \widehat{\mate} \,\K_{\bt} \to \K_{\wt}
  \widehat{\mate} \,\K_{\bt},
  \\
  \notag
  &P_{\wt}\widehat{\mate} P_{\bt}([x]) := 
  \begin{cases}
    [P_{\wt}(x)], &x\in \K_{\wt};
    \\
    [P_{\bt}(x)], &y\in K_{\bt}.
  \end{cases}
\end{align}

\begin{theorem}
  \label{thm:mating2}
  Let $f$ be an expanding Thurston map. Then every sufficiently high
  iterate $F=f^n$ is topologically conjugate to a map $P_{\wt}
  \,\widehat{\mate} \,P_{\bt}$ as above.  
\end{theorem}

The polynomials $P_{\wt},P_{\bt}$ are postcritically finite, furthermore their
Fatou sets are \emph{separated}. This means that two
distinct bounded Fatou
components of $P_{\wt}$ (of $P_{\bt}$) have disjoint closures.  

Note that if each critical point of $P_{\wt}, P_{\bt}$ is strictly
preperiodic there are no bounded Fatou components, i.e., 
$P_{\wt}\,\widehat{\mate} \,P_{\bt}= P_{\wt}\mate P_{\bt}$ in this case.  
Theorem~\ref{thm:mating1} may thus be viewed as a special case of
Theorem~\ref{thm:mating2}.

\smallskip
It is easy to find examples of Thurston maps $f\colon S^2 \to S^2$
(which may be rational) such that no iterate $F=f^n$ is obtained as a
mating of polynomials. For example this is the case when $f$ is a
hyperbolic, postcritically finite rational map which is not a
polynomial and has exactly three postcritical points; see
\cite{unmating}. 

\smallskip
In the theorems above we do not only show existence of the polynomials
$P_{\wt}, P_{\bt}$ into which $F$ ``unmates'', but the polynomials may be
explicitly computed via a simple algorithm. More precisely
their \defn{critical portraits} are computed explicitly. Roughly---and
somewhat incorrectly---speaking this is the set of external angles at
the critical points. Such a critical portrait determines a monic
centered polynomial uniquely. This was introduced in \cite{MR1149891}
and generalized in \cite{Poirier}. 

\smallskip
It has long been known that a rational map $F$ may arise as the mating
of polynomials in several distinct ways. This phenomenon of
\defn{shared mating} was first observed by Wittner in his thesis
\cite{MR2636558}. From our construction it follows that shared matings
are in fact abundant in our setting. The algorithm to find the
critical portraits together with several examples of shared matings
are presented in \cite{unmating}. 

\subsection{Invariant Peano curves}
\label{sec:invar-peano-curv}

This paper is a direct continuation of the paper \cite{peano}, where
the following theorem was proved.

\begin{theorem}
  \label{thm:main}
  Let $f\colon S^2\to S^2$ be an expanding Thurston map. Then for
  every sufficiently high iterate $F=f^n$ there is a \defn{Peano
    curve} $\gamma\colon S^1 \to S^2$ (onto) such that $F(\gamma(z))=
  \gamma(z^d)$. Here $d=\deg F$. This means that the following diagram
  commutes.
  \begin{equation*}
    \xymatrix{
      S^1 \ar[r]^{z^d} \ar[d]_{\gamma}
      &
      S^1 \ar[d]^{\gamma}
      \\
      S^2 \ar[r]_F & S^2
    }
  \end{equation*}
\end{theorem}

Thus $F$ admits a description akin to the description of polynomials
obtained from the Carath\'{e}odory semi-conjugacy \eqref{eq:cara_loop}.

It should be pointed out that this theorem follows from the respective
theorems in Section~\ref{sec:main-results}. Namely the
Carath\'{e}odory semi-conjugacy $\sigma_{\wt}\colon S^1 \to \J_{\wt}\subset
\K_{\wt}$ (similarly $\sigma_{\bt}\colon S^1 \to \J_{\bt}$) descends to (the
quotient) $\K_{\wt}\mate \K_{\bt}$. If we are in the situation of
Theorem~\ref{thm:mating1} there is a homeomorphism $h\colon \K_{\wt} \mate
\K_{\bt}\to S^2$. Then the Peano curve $\gamma\colon S^1\to S^2$ from
Theorem~\ref{thm:main} is given by the composition 
\begin{equation*}
  S^1\overset{\sigma_{\wt}}{\longrightarrow} \J_{\wt} \hookrightarrow 
  \K_{\wt} \mate \K_{\bt} \overset{h}{\longrightarrow} S^2. 
\end{equation*}
Similarly if we are in the case of Theorem~\ref{thm:mating2}.

The theorems from Section~\ref{sec:main-results} however are proved
using Theorem~\ref{thm:main}. The iterate $F=f^n$ needed
in Theorem~\ref{thm:mating1} and Theorem~\ref{thm:mating2} is the same
as the one from Theorem~\ref{thm:main}. Several sufficient conditions
for the existence of an invariant Peano curve $\gamma$ for an
expanding Thurston map $f$
are given in \cite{peano}. These conditions thus are sufficient to
show that $f$ (not some iterate) arises as a mating. An overview of
these conditions is given in \cite{unmating}.

\subsection{The map $F$  as a quotient}
\label{sec:f-described-via}

Using the invariant Peano curve $\gamma \colon S^1\to S^2$ from
Theorem \ref{thm:main}, an
\emph{equivalence relation} on $S^1$ is defined by
\begin{equation}
  \label{eq:eq_rel}
  z\sim w \Leftrightarrow \gamma(z)=\gamma(w),
\end{equation}
for all $z,w\in S^1$. Elementary topology yields that $S^1/\!\sim$ is
homeomorphic to $S^2$ and that $z^d/\!\sim\,\colon S^1/\!\sim \,\to S^1/\!\sim$
is topologically conjugate to the map $F$. 
\begin{theorem}
  \label{thm:S1simS2}
  Let $\gamma\colon S^1\to S^2$ be an invariant Peano curve for the
  expanding Thurston map $F=f^n$ as in Theorem~\ref{thm:main} and
  $\sim$ be the equivalence relation on $S^1$ induced by $\gamma$ as
  above. Then the following diagram commutes,
  \begin{equation*}
    \xymatrix{
      S^1/\!\sim \ar[r]^{z^d/\sim} \ar[d]_{h}
      &
      S^1/\!\sim \ar[d]^{h}
      \\
      S^2 \ar[r]_F & S^2.
    }
  \end{equation*}
  Here the homeomorphism $h\colon S^1/\!\sim \,\to S^2$ is given by
  $h\colon [s]\mapsto \gamma(s)$, for all $s\in S^1$. 
\end{theorem}

Thus the equivalence relation $\sim$ contains all the information to
recover $F$ up to topological conjugacy. 
It turns out that the equivalence relation $\sim$  may be
constructed from 
\emph{finite data}, more precisely from two finite families of finite
sets of rational numbers. Thus these families encode the map $F$ up to
topological conjugacy by the previous theorem. 

The proper setting is as follows. The Peano curve $\gamma$ from
Theorem \ref{thm:main} was constructed as the limit of
\defn{approximations} $\gamma^n$. The approximations have finitely
many points where they touch themselves, but they never cross
themselves. Thus $S^2\setminus \gamma^n$ has two components which are
colored  \emph{white} and \emph{black}. 
We define equivalence relations $\Sim{n,\wt}, \Sim{n,\bt}$ on
$S^1$ by
$s\Sim{n,\wt}t$ ($s\Sim{n,\bt} t$) whenever $\gamma^n(s)=\gamma^n(t)$ and
$\gamma^n$ touches itself at $s,t$ in the white component (black
component). It turns out that
\begin{itemize}
\item $\sim$ can be recovered from the equivalence relations
  $\Sim{n,\wt},\Sim{n,\bt}$ as a limit (defined in a suitable sense, see
  Theorem \ref{thm:sim_usc_clos_siminfty}).
\item The equivalence relations $\Sim{n,\wt},\Sim{n,\bt}$ can be
  \emph{inductively obtained} from the ``initial ones''
  $\Sim{1,\wt},\Sim{1,\bt}$ (Theorem \ref{thm:LnLn+1}).
\item The equivalence classes of $\Sim{1,\wt},\Sim{1,\bt}$ form a
  \defn{critical portrait} in the sense of \cite{Poirier}, see
  Definition \ref{def:critical_portrait}. Such a critical portrait is a
  (finite) family of finite sets of rational angles. 
\end{itemize}

Thus we obtain the following additional main result of this paper. The map
$F=f^n$ is the (same) iterate of the expanding Thurston map $f$ from
Theorem \ref{thm:main}. 
\begin{theorem}
  \label{thm:FcriticalPortraits}
  Let $f\colon S^2\to S^2$ be an expanding Thurston map. Then every
  sufficiently high iterate
  $F=f^n$ admits a description via two critical portraits.
\end{theorem}

This provides an effective way to describe $F$. 
The description of $F$ as in Theorem~\ref{thm:S1simS2} may be viewed
as a two-sided version of the viewpoint introduced by Douady-Hubbard
and Thurston 
(\cite{DHOrsayI}, \cite{DHOrsayII}, \cite{MR2508255}, see
also \cite{MR1149864} and 
\cite{MR1761576}), namely the combinatorial description of Julia sets
in terms of \emph{external rays}.  

\smallskip An outline of the proof Theorem~\ref{thm:mating1} is as
follows. Consider the polynomials $P_{\wt}, P_{\bt}$ given by the
critical portraits (i.e., by $\Sim{1,\wt}, \Sim{1,\bt}$) as outlined
above. The Carath\'{e}odory semi-conjugacy $\sigma_{\wt} \colon S^1 \to
\J_{\wt}$ induces an equivalence relation $\Approx{\wt}$ on $S^1$ as
above by
\begin{equation}
  \label{eq:approx_Cara}
  z\Approx{\wt} w \; :\Leftrightarrow \;\sigma_{\wt}(z)= \sigma_{\wt}(w),
\end{equation}
for all $z,w\in S^1$. The equivalence relation $\Approx{\bt}$ on $S^1$
is defined similarly via $\sigma_{\bt}$. The equivalence relations
$\Approx{\wt}$ and $\Approx{\bt}$ generate the equivalence relation
$\approx$ on $S^1$.  It is relatively easy to show that $z^d/\!\approx
\colon S^1/\! \approx\; \to S^1/\!\approx$ is topologically conjugate
to the mating $P_{\wt}\mate P_{\bt}$. The proof of
Theorem~\ref{thm:mating1} will be obtained by showing that $\approx$
equals the equivalence relation $\sim$ induced by the invariant Peano
curve $\gamma$ as in \eqref{eq:eq_rel}.

\subsection{Further results}
\label{sec:further-results}

We also prove the following theorem here. 
\begin{theorem}
  \label{thm:gammaL1L2}
  The invariant Peano curve $\gamma\colon S^1\to S^2$ from
  Theorem~\ref{thm:main} maps normalized Lebesgue measure
  on $S^1$ to the measure of maximal entropy on $S^2$ with respect to
  $F$.
\end{theorem}
By ``normalized'' it is meant that the total mass is $1$, i.e., that the
measure is a probability measure. The \defn{measure of maximal
  entropy} (also called the \defn{Lyubich} or \defn{Brolin} measure)
is the unique invariant probability measure that maximizes the
(measure theoretic) entropy. It can be defined as the weak limit of
$1/d^n \sum_{y\in F^{-n}(x_0)} \delta_y$, for any point $x_0\in
S^2$. Note that the measure of maximal entropy of $F=f^n$ equals the
measure of maximal entropy of $f$.

Finally we note that each invariant Peano curve $\gamma\colon S^1\to
S^2$ induces a \emph{fractal tiling}. Namely we divide the circle
$S^1$ into $d= \deg F$ arcs and consider their images by $\gamma$. If
$F$ is a rational map these fractal tilings may be lifted to the
\emph{orbifold covering}, which is either the plane $\C$ or the
hyperbolic plane $\Halb$.

\subsection{Outline of this paper}
\label{sec:outline}

In Section \ref{sec:thurston-maps-as-1} we recall the setup from
\cite{expThurMarkov}. Namely one picks a Jordan curve $\CC$ containing all
postcritical points. Then $F^{-n}(\CC)$ decomposes the sphere $S^2$
into \emph{$n$-tiles}. This in turn allows for a combinatorial
description of $F$. 

In Section \ref{sec:invar-peano-curve} the \emph{construction} of the
invariant Peano curve $\gamma$ from Theorem~\ref{thm:main} as given in
\cite{peano} is \emph{reviewed}. In particular $\gamma$ was constructed as a
limit of \emph{approximations} $\gamma^n$, whose properties we list.  

Some elementary facts about equivalence relations are provided in
Section \ref{sec:equivalence-rels}. 

In Section \ref{sec:laminations-2} we introduce \emph{equivalence
  relations} 
$\Sim{n,\wt}, \Sim{n,\bt}$. They describe the self-intersections of the
approximations $\gamma^n$. It is shown that these equivalence
relations are obtained inductively, i.e., from
$\Sim{1,\wt},\Sim{1,\bt}$. These ``initial equivalence relations''
$\Sim{1,\wt}, \Sim{1,\bt}$ form a \emph{critical portrait} in the sense of
\cite{Poirier}. Furthermore the map $F$ is completely determined from
them, up to topological conjugacy. 

In Section \ref{sec:equiv-class-are} we investigate the \emph{sizes} of the
\emph{equivalence classes} induced by $\gamma$. More precisely we show that
if $F$ does not have periodic critical points, the size of such equivalence
classes is bounded by some number $N< \infty$. If $F$ has periodic critical
points, we show that at least one equivalence class is
finite.

In Section \ref{sec:f-mating} we show that $F$ is \emph{obtained as a
  mating}, for the case when $F$ has no periodic critical 
points; i.e., Theorem \ref{thm:mating1} is proved.

The case when $F$ has periodic critical points, i.e.,
Theorem~\ref{thm:mating2}, is proved in Section
\ref{sec:proof-theorem-1}.

In Section \ref{sec:gamma-maps-lebesgue} we show that the invariant
Peano curve $\gamma$ maps \emph{Lebesgue measure} on the circle $S^1$
to the \emph{measure of maximal entropy} of $F$ (on $S^2$), i.e.,
prove Theorem~\ref{thm:gammaL1L2}.

In Section \ref{sec:fractal-tilings} we illustrate the fractal tilings
obtained from the construction. This also shows explicitly the
invariant Peano curve for some examples.

We conclude the paper in Section \ref{sec:open-questions} with some
open questions. 

\subsection{Notation}
\label{sec:notation}

%The Riemann sphere is denoted by $\CDach=\C\cup\{{\infty}\}$. We
%denote 
The circle is denoted by $S^1$, the $2$-sphere by $S^2$. 
By $\inte U$ we denote the interior, by $\clos U$ the closure of a
set $U$. The cardinality of a (finite) set $S$ is denoted by $\#S$.  

It will often be convenient to identify $S^1$ with $\R/\Z$; the map
$z^d \colon S^1\to S^1$ then is denoted by
\begin{equation*}
  \boxed{\phantom{x}\mu=\mu_d\colon \R/\Z \to \R/\Z, 
    \quad \mu(t) = dt (\bmod 1), \phantom{x}   }
\end{equation*}
the $n$-th iterate is $\mu^n(t)=d^nt(\bmod 1)$. 

For two non-negative expressions $A,B$ we write $A\lesssim
B$ if there is a constant $C>0$ such that $A\leq C B$. We refer to $C$
as $C(\lesssim)$. Similarly we write $A\asymp B$ if
$A/C \leq B\leq C A$ for a constant $C\geq 1$, we refer to $C$ as
$C(\asymp)$. 

\begin{asparaitem}[$\centerdot$]
\item The \emph{$n$-iterate} of a map $f$ is denoted by $f^n$.
\item $F=f^n$ is the iterate of the expanding Thurston map $f$ from
  Theorem \ref{thm:main}.
\item By $\crit=\crit(f)$, $\post=\post(f)$ we denote the \defn{set of
    critical} and \defn{postcritical points} (see Section
  \ref{sec:expand-thurst-maps}).
\item The \defn{degree} of $F$ is denoted by $d$, the \defn{number of
    postcritical points} by $k$.
\item The \defn{local degree} of $F$ at $v\in S^2$ is denoted by
  $\deg_F(v)$, see Definition \ref{def:Thurston_map}.
\item \emph{Upper indices} indicate the \defn{order} of an object.
  For example some preimage of some object $U^0$ by
  $F^n$ will be denoted by $U^n$. 
  Also maps and other objects associated with such objects
  $U^n$ will have an upper index ``$n$''.
\item $\CC$ is a Jordan curve containing all postcritical points. 
\item $X^0_{\wt},X^0_{\bt}$ are the white/black $0$-tiles, i.e., the closures
  of the two components of $S^2\setminus \CC$.
\item \emph{Lower indices} ``$\wt$'' or ``$\bt$'' indicate whether an
  object is colored ``white'' or ``black'', i.e., if it is mapped
  eventually to $X^0_{\wt}$ or $X^0_{\bt}$; or closely related to such objects. 
\item A visual metric is denoted by $\varrho$, see Section
  \ref{sec:visual-metric}. 

\item The \emph{sets of all $n$-tiles, -edges, -vertices} are denoted by
  $\X^n,\E^n,\V^n$ (Section~\ref{sec:tiles-edges}).
\item $\gamma^n$ is the \defn{$n$-th approximation of the Peano curve
    $\gamma$} (Section \ref{sec:approximations}).
\item A point $\alpha^n\in S^1$ such that $\gamma^n(\alpha^n)\in
  \V^n$ is called an \defn{$n$-angle}. The \defn{set of all
    $n$-angles} is denoted by $\A^n$ (Section
  \ref{sec:approximations}).
\item An $n$-arc $a^n$ is a closed interval in $\R/\Z=S^1$ that is
  mapped by $\gamma^n$ (homeomorphically) to an $n$-edge.
\item $H^n$ are the \defn{pseudo-isotopies} from which the
  approximations $\gamma^n$ were constructed, see Section
  \ref{sec:pseudo-isotopies}. 
\item $\pi_{\wt}\cup\pi_{\bt}$ is a cnc-partition (of a set $\{0,1, \dots,
  2n-1\}$). It describes the connection at a vertex, i.e., which
  white/black tiles are ``connected'' at $v$ (see Section
  \ref{sec:connetions}).
\item By $\sim$ we denote the equivalence relation on $S^1$ induced by
  the invariant Peano curve ($s\sim t \Leftrightarrow
  \gamma(s)=\gamma(t)$), by $\Sim{n}$ the equivalence relation induced
  by $\gamma^n$ (Section \ref{sec:sim-induced-gamma}). 
\item The \defn{join} of two equivalence relations $\Sim{a},\Sim{b}$ is
  denoted by $\Sim{a} \vee \Sim{b}$, their \defn{meet} by $\Sim{a}\wedge
  \Sim{b}$ (Section \ref{sec:latt-equiv-class}). 
\item The \defn{equivalence relations} $\Sim{n,\wt}, \Sim{n,\bt}$ describe where
  $\gamma^n$ ``touches itself on the white/black side'' (see
  Definition \ref{def:simnw_simnb}). Their equivalence classes are
  denoted by $[\alpha]_{n,\wt}$, $[\alpha]_{n,\bt}$. 
\item $S^2_{\wt}, S^2_{\bt}$ are the \defn{white/black hemispheres}, i.e., the
  components of $S^2\setminus S^1$. They are equipped with the
  hyperbolic metric. 
\item $\LC^n_{\wt},\LC^n_{\wt}$ are the \defn{laminations} associated to
  $\Sim{n,\wt}, \Sim{n,\bt}$. Namely for each equivalence class
  $[\alpha]_{n,\wt}$ there is a \emph{leaf} $L\in \LC^n_{\wt}$, given as the
  hyperbolically convex hull of $[\alpha]_{n,\wt}$. A white/black
  \defn{$n$-gap} is the closure of one component of $S^2_{\wt}\setminus
  \bigcup \LC^n_{\wt}$, or of $S^2_{\wt}\setminus \bigcup \LC^n_{\wt}$
  respectively. The \defn{set of 
    white/black $n$-gaps} is denoted by
  $\G^n_{\wt}, \G^n_{\bt}$ (Section \ref{sec:laminations-1}).
\item The \defn{Carath\'{e}odory semi-conjugacy} is denoted by
  $\sigma$ (see Section \ref{sec:carath-semi-conjugacy-polyn}). The
  equivalence relation induced by $\sigma$ is denoted by $\approx$
  (\ref{eq:approx_Cara}). The semi-conjugacies, equivalence relations
  of the white, black polynomials $P_{\wt},P_{\bt}$ are denoted by
  $\sigma_{\wt},\sigma_{\bt}$ and $\Approx{\wt}, \Approx{\bt}$. Identifying
  additionally 
  closures of bounded Fatou components yields the equivalence
  relations $\Approx{\F,\wt}, \Approx{\F,\bt}$ (see
  (\ref{eq:def_approxFw}), (\ref{eq:def_approxFb})).  
\end{asparaitem}

\section{Expanding Thurston maps}
\label{sec:thurston-maps-as-1}

Here some material from \cite{expThurMarkov} is reviewed. 

\subsection{Definition of expanding Thurston maps}
\label{sec:expand-thurst-maps}

\begin{definition}
  \label{def:Thurston_map}
  A \defn{Thurston map} is an orientation-preserving, postcritically
  finite, branched cover of
  the sphere $f\colon S^2\to S^2$. This means that locally $f$ can be
  written as 
  $z\mapsto z^q$, $q\geq 1$ (after suitable local, orientation
  preserving, 
  homeomorphic changes of coordinates in domain and range). 
  More precisely for each point
  $v\in S^2$ there exists a $q\in \N$, (open) neighborhoods $V, W$
  of $v, w=f(v)$ and orientation preserving homeomorphisms
  $\varphi\colon V\to \D$,  $\psi\colon W\to \D$ with $\varphi(v)=0$,
  $\psi(w)=0$ satisfying 
  \begin{equation*}
    \psi\circ f\circ \varphi^{-1}(z)= z^q,
  \end{equation*}
  for all $z\in \D$.  The integer $q=\deg_f(v)\geq 1$ is called the
  \defn{local degree} of the map at $v$. A point $c\in S^2$ at which
  the local degree $\deg_f(c)\geq 2$ is called a \defn{critical
    point}.  The set of all critical points is denoted by
  $\crit=\crit(f)$. \emph{Postcritically finiteness} means that the
  \defn{set of postcritical points}
  \begin{equation*}
    \post=\post(f):= \bigcup_{j\geq 1} f^j(\crit) 
  \end{equation*}
  is finite.

  Fix a Jordan curve $\CC\supset \post$. The Thurston map is called
  \defn{expanding} if
  \begin{equation*}
    \mesh f^{-n}(\CC)\to 0 \text{ as } n\to \infty.
  \end{equation*}
  Here $\mesh f^{-n}(\CC)$ is the maximal diameter of a component of
  $S^2\setminus f^{-n}(\CC)$. 
  This is independent of the chosen curve $\CC$ (see
  \cite[Lemma~6.1]{expThurMarkov}). It is equivalent to the notion of
  expansion by Ha\"{i}ssinsky-Pilgrim in \cite{HaiPilCoarse} (see
  \cite[Proposition~6.2]{expThurMarkov}).  
\end{definition}
If $f$ is a rational map it is expanding if and only if its Julia set is
the whole sphere if and only if $f$ has no periodic critical points,
see\cite[Proposition~2.3]{expThurMarkov}. 

We will however not assume that $f$ is a rational map; in fact it
should be emphasized that we make no assumption on the smoothness of
$f$.  

In particular we allow expanding Thurston maps that have
\emph{periodic critical points}. An example of  such a map is obtained
as follows. Consider a postcritically finite rational map $g$ whose Julia
set is a Sierpi\'{n}ski carpet (the closure of each Fatou component is
a Jordan domain, closures of distinct Fatou components are
disjoint). Identify the closure of each Fatou component, the map
descends to this quotient. The quotient map is an expanding Thurston
map with periodic critical points. 

\smallskip
The map
\begin{equation*}
  F=f^n
\end{equation*}
will always denote the iterate from Theorem \ref{thm:main}. Note that
$\post(F)=\post(f)$. 
Throughout this paper we denote by
\begin{equation*}
  \boxed{\phantom{x}d := \deg F = (\deg f)^n, \quad k:=\#\post.\phantom{x}   }
\end{equation*}

\subsection{Tiles and edges}
\label{sec:tiles-edges}

Fix a Jordan curve $\CC\supset \post$ (this was the first step to
construct the Peano curve $\gamma$) and give it an orientation.  In
\cite[Chapter~13]{expThurMarkov} it was shown that we can choose $\CC$
to be $F$-invariant ($F(\CC)\subset \CC$), but we do not assume this
here.  We \emph{color} the components of $S^2\setminus \CC$
\emph{white} and \emph{black}, such that $\CC$ is positively oriented
as the boundary of the white component. The closures of the
white/black components of $S^2\setminus \CC$ are called the
white/black $0$-tiles, denoted by $X^0_{\wt},X^0_{\bt}$. Similarly the
closure of each component of $S^2\setminus F^{-n}(\CC)$ is called an
\defn{$n$-tile}. It is colored white if it is mapped to $X^0_{\wt}$ by
$F^n$, black if mapped to $X^0_{\bt}$ by $F^n$. The set of all
$n$-tiles is denoted by $\X^n$.  The restricted map
\begin{equation}
  \label{eq:FXnX_homeo}
  F^n\colon X^n \to X_{j} \text{ is a homeomorphism}
\end{equation}
for each $X^n\in \X^n$, here $j\in\{\wt,\bt\}$ (see
\cite[Prop~5.17]{expThurMarkov}). % check !!!.
In particular each $n$-tile
is a closed Jordan domain. 

Any point $v\in F^{-n}(\post)$ is called an \emph{$n$-vertex}, the set
of all $n$-vertices is $\V^n:= F^{-n}(\post)$.  Each $n$-tile contains
$k$ $n$-vertices in its boundary. Sometimes a postcritical point is
also called a \defn{$0$-vertex} for convenience. The set of all
$n$-vertices is denoted by $\V^n$. Note that
\begin{equation*}
  \post =\V^0 \subset \V^1 \subset \dots .
\end{equation*}
Thus an $n$-vertex $v\in S^2$ is a $m$-vertex for all $m\geq
n$. Expansion implies that the union of the sets $\V^n$ is dense.

\smallskip
The postcritical points divide the curve $\CC$ into $k$ (closed) arcs,
called the \defn{$0$-edges}. Similarly the closure of a component of
$F^{-n}(\CC) \setminus \V^n$ is called an \defn{$n$-edge}. Each
$n$-edge is mapped by $F^n$ homeomorphically to some $0$-edge. 
The boundary of each $n$-tile consists of $k$
$n$-edges. The set of all $n$-edges is denoted by $\E^n$.

\smallskip
Recall that $n$-tiles are colored white and black. The $n$-tiles tile
the sphere $S^2$ in a \emph{checkerboard fashion}. This means that two
$n$-tiles which share an $n$-edge are colored differently. Put
differently, at each $n$-vertex $v$ an even number of $n$-tiles
intersect, their colors alternate around $v$. 

\subsection{The visual metric}
\label{sec:visual-metric}

In \cite[Chapter~7]{expThurMarkov} \defn{visual metrics} on $S^2$ for
an expanding Thurston map $F$ were considered. These are metrics
defined in combinatorial terms, i.e., in terms of the tiles defined
above.

Assume the Jordan curve $\CC\supset \post$ has been fixed, and tiles
are defined as above. Let
\begin{equation}
  \label{eq:def_mxy}
  m(x,y):= \min\{n\in \N \mid \text{ there exists \emph{disjoint}
    $n$-tiles } X^n\ni x, Y^n\ni y\},  
\end{equation}
for all distinct $x,y\in S^2$; we set $m(x,x)=\infty$ (for all $x\in
S^2$). A metric $\varrho$ on $S^2$ is a \defn{visual metric} for $F$
if there is a constant $\Lambda>1$ such that
\begin{equation}
  \label{eq:dS_comb}
  \varrho(x,y)\asymp \Lambda^{-m(x,y)},
\end{equation}
for all $x,y\in S^2$. Here we set $\Lambda^{-\infty}=0$. The constant
$C(\asymp)=C(\CC)$ is independent of $x,y$. The number $\Lambda$ is
called the \defn{expansion factor} of $\varrho$. Visual metrics are
not unique, but distinct visual metrics $\varrho, \widetilde{\varrho}$
are \defn{snowflake equivalent} (i.e., $\varrho^\alpha\asymp
\widetilde{\varrho}$ for some $\alpha>0$). 
 
\smallskip
We note the following. Let $A\subset S^2$ be an $n$-tile or an
$n$-edge. Then
\begin{equation}
  \label{eq:diam_visual}
  \diam A \asymp \Lambda^{-n},
\end{equation}
where $C=C(\asymp)$ is a constant independent of $A$, the diameter is
measured with respect to a visual metric $\varrho$ (with expansion
factor $\Lambda>1$). 

*** check: is this needed? ***
 it was shown that $S^2$ can be equipped with a
\defn{visual metric}, denoted by $\abs{x-y}_{\SC}$, with respect to
which $F$ is a \emph{local similarity}.
\begin{theorem}[\cite{expThurMarkov}]
  \label{thm:Smetric}
  There is a constant $\Lambda>1$ such that the following holds.
  For every $x\in S^2$ there is a neighborhood $U_x$ of $x$ such that
  \begin{equation*}
    \frac{\abs{F(x) -F(y)}_{\SC}}{\abs{x-y}_{\SC}} = \Lambda 
    \quad 
    \text{for all $y\in U_x$.}
  \end{equation*}
\end{theorem}

 Then (see \cite[Section 8]{expThurMarkov}) 

\section{The invariant Peano curve and its approximations}
\label{sec:invar-peano-curve}

The paper present is a direct continuation of \cite{peano}, as
mentioned in the introduction. In particular we do not only use the
main result (i.e., Theorem \ref{thm:main}), but in fact the whole
construction of the invariant Peano curve $\gamma$. We outline the
construction here. Results of \cite{peano} will be used freely.

\smallskip
A Jordan curve $\CC\subset S^2$ with $\post\subset \CC$ is fixed. A
visual metric $\varrho$ for $F$ with expansion factor $\Lambda>1$ is
defined via $\CC$ and fixed. Metrical terms below are defined
in terms of $\varrho$.

\smallskip
It will often be convenient to identify $S^1$ with $\R/\Z$. We write
$\mu=\mu_d \colon \R/\Z\to \R/\Z$ for the map $t\mapsto d t(\bmod 1)$
(which is conjugate to $z^d\colon \partial \D\to \partial \D$).  
We still write $\gamma\colon \R/\Z \to \R/\Z$ for the
invariant Peano curve, slightly abusing notation. The same abuse of
notation applies to the approximations $\gamma^n$ defined below. 

\subsection{The approximations $\gamma^n$}
\label{sec:approximations}

The curve $\gamma$ is constructed as the limit of curves
$\gamma^n\colon S^1\to S^2$, called the \defn{$n$-th approximation}
(of the Peano curve). The approximations satisfy the following.
\begin{itemize}
\item As a set $\gamma^0=\CC$, more precisely
  \begin{equation*}
    \gamma^0\colon S^1 \to \CC
  \end{equation*}
  is a homeomorphism.
\item The $\gamma^n$ cover all $n$-edges. This means when $\gamma^n$
  is viewed as a set it holds
  \begin{equation*}
    \gamma^n = \bigcup \E^n.
  \end{equation*}
\item A point $\alpha^n\in S^1$ such that $\gamma^n(\alpha^n)\in \V^n$ 
  is called an \defn{$n$-angle}. Each $n$-angle is \emph{rational}
  (here we 
  identify $S^1$ with $\R/\Z$).  The \defn{set of all $n$-angles} 
  \begin{equation*}
    \A^n:= (\gamma^n)^{-1}(\V^n)
  \end{equation*}
  is a finite set.
\item For $m\geq n$ it holds that $\gamma^m=\gamma^n$ on $\A^n$, 
  \begin{equation*}
    \gamma^n|_{\A^n}=\gamma^m|_{\A^n} = \gamma|_{\A^n}.
  \end{equation*}
  Thus (recall that $\V^0\subset \V^1\subset \dots$)  
  \begin{equation*}
    \A^0\subset \A^1\subset \dots .
  \end{equation*}
\item The $n$-angles divide the circle $S^1$ into (closed)
  \defn{$n$-arcs}. Each $n$-arc is mapped by $\gamma^n$
  homeomorphically to an 
  $n$-edge. Conversely for each $n$-edge $E^n$ there is a unique
  $n$-arc $a^n\subset S^1$, such that $\gamma^n(a^n)= E^n$.
\item Each $n$-arc is mapped homeomorphically to an $(n-1)$-arc by
  $z^d\colon S^1\to S^1$. More precisely, we have the following
  commutative diagram:
  \begin{equation}
    \label{eq:Angn_comm_dia}
    \xymatrix{
      \A^{n+1}\subset \R/\Z \ar[r]^{\mu} \ar[d]_{\gamma^{n+1}}
      &
      \A^n\subset \R/\Z \ar[d]^{\gamma^n}
      \\
      \V^{n+1}\subset S^2 \ar[r]_F & \V^n \subset S^2.
    }
  \end{equation}
  See \cite[Lemma 4.5 (2)]{peano}, as well as Remark
  \ref{rem:alpha_alpha_tilde}. Thus there is a constant
  $C(\lesssim)>0$ such that  
  \begin{equation*}
   \diam a^n\lesssim d^{-n} 
 \end{equation*}
 for each $n$-arc $a^n$. Recall that the diameter is measured with
 respect to the visual metric $\varrho$. 
\item The curve $\gamma^n$ touches itself, but does not intersect
  itself. This means for any $\epsilon>0$ there is a Jordan curve
  $\gamma^n_{\epsilon}\colon S^1\to S^2$ such that
  \begin{equation*}
    \norm{\gamma^n -\gamma^n_{\epsilon}}_{\infty} < \epsilon. 
  \end{equation*}
  Here $\norm{\gamma^n - \gamma^n_\epsilon}_\infty :=
  \max\{\varrho(\gamma^n(t), \gamma^n_\epsilon(t)) \mid t\in S^1\}$ is
  the supremums norm with respect to some fixed visual metric
  $\varrho$.  
\item The curves $\gamma^n$ converge uniformly to $\gamma$. More
  precisely 
  \begin{equation}
\label{eq:gntog}
    \norm{\gamma^n-\gamma}_{\infty}\lesssim \Lambda^{-n}.
  \end{equation}
  % Here the supremums norm (as well as the constant $\Lambda>1$) is
  % taken with respect to the metric from Theorem \ref{thm:Smetric}.
\item Each $n$-edge $E^n$ is contained in the boundary of a (unique)
  white $n$-tile $X^n_{\wt}$. We orient $\partial X^n_{\wt}$
  mathematically positively, this induces an orientation on
  $E^n$. Equivalently we may define the orientation on $E^n$ as
  follows. Each $0$-edge $E^0\subset\CC$ inherits an orientation from the
  orientation of $\CC$. Since $F^n$ maps $E^n$ to some $0$-edge $E^0$
  homeomorphically, we call pull back the orientation of $0$-edges
  to each $n$-edge. Note that $F$ maps
  positively oriented $n$-edges to positively oriented $(n-1)$-edges
  by definition.
  
  A distinct way to define the orientation of $n$-edges is as
  follows. Recall that for each $n$-edge $E^n$ there is exactly one
  $n$-arc $a^n\subset S^1$ which is mapped homeomorphically to $E^n$
  by $\gamma^n$. Thus $E^n= \gamma^n(a^n)$ inherits the orientation of
  $a^n\subset S^1$. The following holds for the approximations
  $\gamma^n$:
  \begin{equation*}
    (\star) \text{ The two orientations on $E^n$ described above agree.}
  \end{equation*}

\end{itemize}

\begin{remark}
  \label{rem:alpha_alpha_tilde}
  %We will need the approximations of the Peano curve as
  %\emph{parametrized curves}. 
  Note that the approximations
  $\tilde{\gamma}^n\colon \R/\Z\to S^2$ parametrized as in 
  \cite[Section 4.2]{peano}  %\ref{sec:parametrizing-gn} 
  converge to $\tilde{\gamma}\colon \R/\Z\to S^2$, which
  semi-conjugates $F$ to $\widetilde{\mu}(t)=dt + \theta_0 (\bmod 1$). 
  The approximations 
  \begin{align*}
    %&\gamma^n(z):=\tilde{\gamma}^n( e^{2\pi i \theta_0/(1-d)}z)
    &\gamma^n(t):=\tilde{\gamma}^n(t - \theta_0/(d-1))
    \intertext{converge to}
    &\gamma(t)=\tilde{\gamma}(t- \theta_0(d-1)),
  \end{align*}
  which is the desired Peano curve, i.e., semi-conjugates
  $F$ to $\mu(t) = dt (\bmod 1)$, see 
  \cite[Lemma 4.1]{peano}  %\ref{lem:theorem1reduction}.
  and the subsequent remark.
  Let
  $\{\tilde{\alpha}^n_j\}\subset S^1$ be the 
  points from 
  \cite[Section 4.2]{peano}  %\ref{sec:parametrizing-gn} 
  (that are mapped by
  $\tilde{\gamma}^n$ to $n$-vertices). Then the points 
  \begin{equation*}
    \alpha^n_j:= \tilde{\alpha}^n_j+\theta_0/(d-1)     
  \end{equation*}
  are the $n$-angles. 
\end{remark}

\subsection{Pseudo-isotopies}
\label{sec:pseudo-isotopies}

The approximations $\gamma^n$ described in the last section were
constructed as follows. A \defn{pseudo-isotopy} $H\colon S^2\times
[0,1]\to S^2$ is a homotopy, that is an isotopy an $[0,1)$ (i.e., it ceases
to be a homeomorphism only at $t=1$); if $H$ is constant on a set $A$
it is an \defn{isotopy rel.\ $A$}. In \cite{peano} a pseudo-isotopy $H^0$
rel.\ $\post=\V^0$ is constructed that deforms $\gamma^0$ to
$\gamma^1$, where $\gamma^0,\gamma^1$ are as in the last section. It
is possible to \emph{lift} $H^0$ by $F^n$ to $H^n$. This is a
pseudo-isotopy rel.\ $\V^n$ and deforms $\gamma^n$ to $\gamma^{n+1}$.

\subsection{Connections}
\label{sec:connetions}

The first approximation $\gamma^1$ (more precisely the pseudo-isotopy
$H^0$)  
was constructed as follows. At each $1$-vertex $v$ several
white $1$-tiles intersect. We declare which white $1$-tiles are
\defn{connected} at $v$; see Figure \ref{fig:connection} for an
illustration. Connections (of white $1$-tiles at $v$) are 
\defn{non-crossing}. Furthermore the resulting \emph{white tile graph}
forms a \emph{spanning tree}. ``Following the outline'' of this tree
yields the first approximation $\gamma^1$. There is one more
ingredient: the resulting curve $\gamma^1$ has to be ``in the right
homotopy class''. This means there has to be a pseudo-isotopy $H^0$ rel.\
$\post$ as
in Section \ref{sec:pseudo-isotopies} that deforms $\gamma^0$ to
$\gamma^1$. 

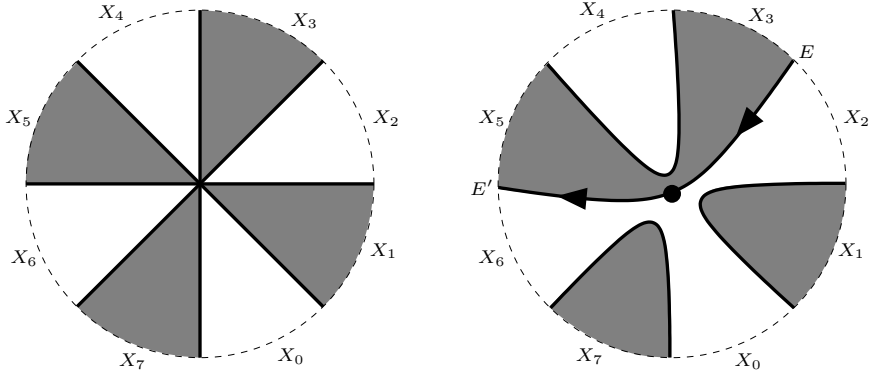
\begin{figure}
  \centering
  
  \includegraphics[width=11cm]{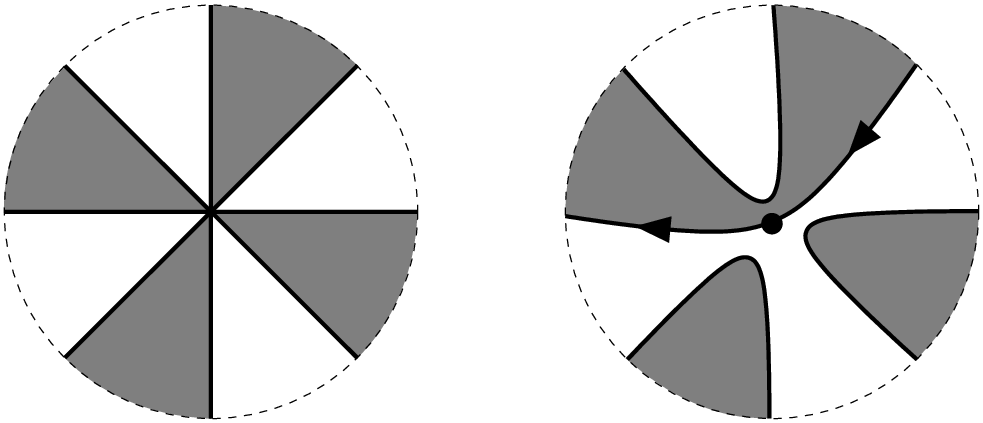}
  \begin{picture}(10,10)
    %
    % left figure
    \put(-220,0){$\scriptstyle{X_0}$}
    \put(-185,40){$\scriptstyle{X_1}$}
    \put(-184,90){$\scriptstyle{X_2}$}
    \put(-215,127){$\scriptstyle{X_3}$}
    \put(-288,130){$\scriptstyle{X_4}$}
    \put(-323,90){$\scriptstyle{X_5}$}
    \put(-321,37){$\scriptstyle{X_6}$}
    \put(-280,-2){$\scriptstyle{X_7}$}
    %
    % right figure
    \put(-47,-2){$\scriptstyle{X_0}$}
    \put(-8,40){$\scriptstyle{X_1}$}
    \put(-6,90){$\scriptstyle{X_2}$}
    \put(-42,129){$\scriptstyle{X_3}$}
    \put(-106,131){$\scriptstyle{X_4}$}
    \put(-144,90){$\scriptstyle{X_5}$}
    \put(-144,37){$\scriptstyle{X_6}$}
    \put(-107,0){$\scriptstyle{X_7}$}    
    \put(-23,115){$\scriptstyle{E}$}
    \put(-147,63){$\scriptstyle{E'}$}        
  \end{picture}
  \caption{Connection at a vertex.}
  \label{fig:connection}
\end{figure}

% \begin{figure}
%   \centering
  
%   \includegraphics[width=11cm]{connection.eps}
%   \begin{picture}(10,10)
%     \put(-220,0){$\scriptstyle{X_0}$}
%     \put(-175,40){$\scriptstyle{X_1}$}
%     \put(-190,90){$\scriptstyle{X_2}$}
%     \put(-215,133){$\scriptstyle{X_3}$}
%     \put(-273,130){$\scriptstyle{X_4}$}
%     \put(-330,90){$\scriptstyle{X_5}$}
%     \put(-310,37){$\scriptstyle{X_6}$}
%     \put(-280,-5){$\scriptstyle{X_7}$}
%     %
%     \put(-47,0){$\scriptstyle{X_0}$}
%     \put(-2,40){$\scriptstyle{X_1}$}
%     \put(-17,90){$\scriptstyle{X_2}$}
%     \put(-42,133){$\scriptstyle{X_3}$}
%     \put(-100,130){$\scriptstyle{X_4}$}
%     \put(-157,90){$\scriptstyle{X_5}$}
%     \put(-137,37){$\scriptstyle{X_6}$}
%     \put(-107,-5){$\scriptstyle{X_7}$}    
%   \end{picture}
%   \caption{Connection at a vertex.}
%   \label{fig:connection}
% \end{figure}

\smallskip
Formally let $X_0,\dots X_{2n-1}$ be the $1$-tiles intersecting in a
$1$-vertex $v$, ordered mathematically positively around $v$. The
white $1$-tiles have even index, the black ones odd index. We consider
a decomposition $\pi_{\wt}=\pi_{\wt}(v)$ of $\{0, 2, \dots , 2n-2\}$ (i.e., of
indices corresponding to white $1$-tiles around $v$); and a
decomposition $\pi_{\bt}=\pi_{\bt}(v)$ of $\{1,3,\dots, 2n-1\}$ (i.e., indices
corresponding to black $1$-tiles around $v$). They satisfy the
following:
\begin{itemize}
\item They are \emph{decompositions}. This means $\pi_{\wt}=\{b_1,\dots,
  b_N\}$, where each \defn{block} $b_i$ is a subset of $\{0,2,\dots,
  2n-2\}$, $b_i\cap b_j=\emptyset$ ($i\neq j$), and $\bigcup b_i =
  \{0,2, \dots, 2n-2\}$. Similarly for $\pi_{\bt}$. 
  %The sets $b_i$ are
  %called the \emph{blocks} of the decomposition.
\item The decompositions $\pi_{\wt},\pi_{\bt}$ are \defn{non-crossing}. This
  means the following. Two distinct blocks $b_i,b_j\in \pi_{\wt}$ are
  \emph{crossing} if there are numbers $a,c\in b_i$, $b,d\in b_j$ and
  \begin{equation*}
    a < b < c < d. 
  \end{equation*}
  Each partition $\pi_{\wt},\pi_{\bt}$ does not contain any (pair of) crossing
  blocks. 
\item The partitions $\pi_{\wt},\pi_{\bt}$ are \defn{complementary}. This
  means the following. Given $\pi_{\wt}$, the partition $\pi_{\bt}$ is the
  unique, biggest partition (of $\{1,3,\dots, 2n-1\}$) such that
  $\pi_{\wt}\cup \pi_{\bt}$ is a non-crossing partition of $\{0,1,\dots,
  2n-1\}$.  
\end{itemize}
A partition $\pi_{\wt}\cup \pi_{\bt}$ as above is called a \defn{complementary
  non-crossing partition}, or cnc-partition. A \defn{connection} (of
$1$-tiles) assigns to each $1$-vertex a cnc-partition as above. Two
$1$-tiles $X_i,X_j\ni v$ are said to be \defn{connected at $v$} if
the indices $i,j$ are contained in the \emph{same} block of $\pi_{\wt}\cup
\pi_{\bt}$. Note that tiles of different color are never connected. The
$1$-tile $X_i$ is \defn{incident} to the block $b\ni i$ of $\pi_{\wt}\cup
\pi_{\bt}$.  
In the
example illustrated in Figure \ref{fig:connection} we have
$\pi_{\wt}=\left\{\{0,2,6\}, \{4\}\right\}$,
$\pi_{\bt}=\left\{\{1\},\{3,5\},\{7\}\right\}$.   

\begin{remark}
  In the construction of the ``initial pseudo-isotopy'' $H^0$ we need
  additional data to the one described above. Namely in the case when
  $v$ is a postcritical point, we need to ``say where $v$ is in the
  connection''. This will be of no importance in the present paper. 
\end{remark}

% Given the Peano curve $\gamma\colon S^1\to S^2$ as in Theorem
% \ref{thm:main} we can define an equivalence relation on $S^1$ by
% \begin{equation}
%   \label{eq:defrel}
%   s\sim t \text{ if and only if } \gamma(s)=\gamma(t),
% \end{equation}
% for all $s,t\in S^1$. Then
% \begin{align}
%   \label{eq:S1rel}
%     &S^2\text{ is (homeomorphic to) } S^1/\sim \quad \text{and}
%     \\ 
%     \notag
%     &F \text{ is (conjugate to) } z^d/\sim.  
% \end{align}
% The equivalence relation may be constructed from \defn{finite
%   data}. 
% This data consists of the 
% self intersections of the first approximation
% $\gamma^1$, or equivalently the connection of the $1$-tiles. It may in fact
% be given as a finite list of rational numbers. This provides a way to
% classify expanding Thurston maps effectively. 

% \medskip
% We do \emph{not} assume that $F,\CC$ satisfy the assumptions from
% \cite{peano} Section 7. %\ref{sec:construction-h0}. 
% In particular we do \emph{not} assume that
% $\CC$ is $F$\emph{-invariant}. We \emph{do} however assume that
% $\gamma$ was constructed as in 
% \cite{peano}
% Section 3 %\ref{sec:appr-gn} 
% and Section 4. %\ref{sec:constr-g}.    

\subsection{Succeeding edges}
\label{sec:succeeding-edges}

The connection of $1$-tiles from Section \ref{sec:connetions} can be
used to define which $1$-edges are \defn{succeeding} at some
$1$-vertex $v$. Indeed this is the main purpose of the
connection. Figure \ref{fig:connection} serves again as an
illustration. 

Let the connection at a $1$-vertex $v$ be given by $\pi_{\wt}(v)\cup
\pi_{\bt}(v)$. Two
indices $i,j\in b\in \pi_{\wt}(v)\cup\pi_{\bt}(v)$ are called
\defn{succeeding} (in $b$), if $b$ does not contain any index $i+1,
i+2, \dots, j-1$. Indices are taken $\bmod \,2n$ here, where $2n$ is the
number of $1$-tiles containing $v$. 

Consider two positively oriented (as boundaries of white $1$-tiles)
$1$-edges $E,E'$; where $E$ has terminal point $v$, $E'$ has initial
point $v$. $E,E'$ are \defn{succeeding at $v$} (with respect to the
given connection) if $E\subset X_i$,
$E'\subset X_j$ and $i,j$ are succeeding indices of a block $b\in
\pi_{\wt}(v)$ (thus $X_i,X_j$ are \defn{white} $1$-tiles). The first
approximation $\gamma^1$ (viewed as an Eulerian circuit in $\bigcup
\E^1$) may be given by the connection,
\begin{align*}
  &\text{the $1$-edges $E,E'$ are \defn{succeeding} in } \gamma^1
  \text{ if and only if}
  \\
  &\text{they are succeeding with respect to the
    connection.}  
\end{align*}

\subsection{Connection of $n$-tiles}
\label{sec:connection-n-tiles}

The connection of $1$-tiles can be \emph{lifted} to a connection of
$n$-tiles (see \cite[Section 8]{peano}). Thus at each $n$-vertex $v$ a
cnc-partition $\pi^n_{\wt}(v)\cup\pi^n_{\bt}(v)$ (as in
Section~\ref{sec:connetions}) is defined. Succeeding $n$-edges are
defined as in Section~\ref{sec:succeeding-edges}. As before
\begin{align*}
  &\text{the $n$-edges $E,E'$ are \defn{succeeding} in } \gamma^n
  \text{ if and only if}
  \\
  &\text{they are succeeding with respect to the
    connection of $n$-tiles.}  
\end{align*}

\subsection{The connection graph}
\label{sec:connection-graph}

The connection of $n$-tiles can be used to defined the \defn{$n$-th white
  connection graph}. It is constructed as follows. For each
\emph{white} $n$-tile $X$ there is a vertex $c(X)$ (thought of as the
\emph{center} of the $n$-tile $X$); for each $n$-vertex $v$ and block
$b\in \pi^n_{\wt}(v)$ there is a vertex $c(v,b)$. The vertex $c(X)$ is
connected to $c(v,b)$ if (and only if) $X$ is incident to $b$ at
$v$. The connection of $n$-tiles satisfies the following.
\begin{equation*}
  \text{The $n$-th white connection graph is a tree.}
\end{equation*}

\section{Equivalence relations}
\label{sec:equivalence-rels}

After these preparations we begin with the proof of the main
theorems. As outlined in the introduction we consider the equivalence
relation $\sim$ induced by the invariant Peano curve $\gamma$. First
some elementary material is reviewed. The main result here is
Theorem~\ref{thm:sim_usc_clos_siminfty}, which says that $\sim$ may be
obtained from the equivalence relations $\Sim{n}$ induced by the
approximations $\gamma^n$.

\subsection{The lattice of equivalence relations}
\label{sec:latt-equiv-class}

Equivalence relations on a set $S$ can be \defn{partially ordered} in
a natural way, namely $\simb$ is \defn{bigger} than $\sima$ ($\simb
\;\geq\; \sima$) if
\begin{equation}
  \label{eq:eq_rel_po}
  s\sima t \Rightarrow s\simb t,
\end{equation}
for all $s,t\in S$. Equivalently, each equivalence class of $\sima$ is
a subset of some equivalence class of $\simb$; equivalently if we view
equivalence relations as subsets of $S\times S$ this means $\simb\;
\supset\; \sima$. 

The set of all equivalence relations on $S$ forms a \defn{lattice},
when equipped with this partial ordering. This means that
\defn{join} and \defn{meet} are well defined. Recall that
the join $\sima \vee \simb$ is the smallest equivalence relation bigger  
than $\sima$ and $\simb$. If
$\Sim{\vee}\; :=\; \sima\vee \simb$, then
\begin{align*}
  & s\Sim{\vee} t \text{ if and only if}
  \\
  & \text{there are } s_1,\dots, s_N \in S \text{ such that}
  \\
  & s=s_1\sima s_2 \simb s_3\dots s_{N-2} \sima s_{N-1} \simb s_N = t,
\end{align*}
for all $s,t\in S$.

The meet $\sima \wedge \simb$ is the
biggest equivalence relation smaller than $\sima$ and $\simb$. If
$\Sim{\wedge}\; := \;\Sim{a} \wedge \Sim{b}$, then
\begin{align}
  \label{eq:defmeet}
  & s\Sim{\wedge} t \text{ if and only if } s\Sim{a} t \text{ and }
  s\Sim{b} t.
%   \\
%   & \text{there are } s_1,\dots, s_N \in S \text{ such that}
%   \\
%   & s=s_1\sima s_2 \simb s_3\dots s_{N-2} \sima s_{N-1} \simb s_N = t,
\end{align}
for all $s,t\in S$. When $\sima, \simb$ are viewed as subsets of
$S\times S$ this means that $\Sim{\wedge}\; = \;{\sima \cap \simb}$.

\subsection{Closed equivalence relations}
\label{sec:upper-semic-decomp}

We consider an equivalence relation $\sim$ on a topological space
$S$. 
 
\begin{definition}[Closed equivalence relation]
  \label{deflem:usc}
  An equivalence relation $\sim$ on a compact metric space $S$ is called
  \defn{closed} if one of the following equivalent conditions 
  holds. 
  \begin{enumerate}[(CE 1)]
  \item 
    \label{item:simclosed}
    $\{(s,t) \mid s\sim t\}\subset S\times S$ is closed.
  \item 
    \label{item:usc_3}
    For any two convergent sequences in $S$; $s_n\to s_0, t_n\to t_0$
    it holds
%    Let $(s_n),(t_n)$ be two convergent sequences in $S$. Then
    \begin{equation*}
      s_n\sim t_n \text{ for all } n\geq 1 \Rightarrow s_0 \sim t_0. 
    \end{equation*}
%   \item 
%     \label{item:usc_1}
%     Each equivalence class $[s]\subset S$ is compact and
%     for each $[s]$ and open set $U\supset [s]$ there is an open set
%     $V\supset [s]$ such that 
%     \begin{equation*}
%       [t]\cap V\ne \emptyset \Rightarrow [t]\subset U,
%     \end{equation*}
%     for all $t\in S$.
  \item 
    \label{item:usc_2}
    The projection map
    \begin{equation*}
      \pi\colon S\to S/\!\sim \text{ given by } s\mapsto [s]
    \end{equation*}
    is closed.
    % 
    % save counter of enumeration
    \setcounter{mylistnum}{\value{enumi}}
  \end{enumerate}
\end{definition}

The proof that the above conditions are equivalent is straightforward
and left as an exercise.  
% \begin{proof}
%   The equivalence of the first two properties
%   is Proposition 1.1 in \cite{MR872468} (the
%   proof though is left as an exercise). It is straightforward to show
%   the equivalence with (\ref{item:usc_3}), which is left as an
%   exercise. 
% \end{proof}

\begin{remark}
  If $\sim$ is closed as above it follows that each equivalence class
  is compact. Indeed by (CE \ref{item:usc_2}) the set
  $[s]=\pi^{-1}(\pi(s))$ is closed, thus compact (for all $s\in S$). 
\end{remark}

\begin{remark}
  The set of equivalence classes $\{[s]\mid s\in S\}$ forms a
  \defn{decomposition} of $S$ (i.e., a set of disjoint subsets of $S$
  whose union is $S$). Conversely, each decomposition can be viewed as
  an equivalence relation. An equivalence relation is closed if and
  only if the induced decomposition is \defn{upper
    semicontinuous}. Property (CE \ref{item:usc_2}), together with the
  requirement that each equivalence class is compact, is the general
  definition of upper semicontinuity in any topological space. 
\end{remark}  

\begin{remark}
  The closedness/upper semicontinuity of $\sim$ should be viewed as
  the minimal requirement that the quotient space (or
  \emph{decomposition}   
  space) $S/\!\sim$ has a ``reasonable'' topology. For example ($S$ is
  a compact metric space)  $\sim$
  is closed if and only if
  \begin{enumerate}[(CE 1)]
    \setcounter{enumi}{\value{mylistnum}}
  \item 
    \label{item:usc4}
    $S/\!\sim$ is Hausdorff.
  \end{enumerate}
  The necessity follows immediately from (CE \ref{item:usc_3}), see 
  \cite[Proposition 2.1]{MR872468} for the sufficiency (this is the
  standard reference on decomposition spaces). Several other
  equivalent conditions for an equivalence relation to be closed may
  be found in \cite[Lemma~2.2]{mating_defs}. 
\end{remark}

\begin{lemma}[Closure of equivalence relation]
  \label{lem:usc_closure}
  Let $\sim$ be an equivalence relation on a compact metric space
  $S$. Then there is a unique smallest closed equivalence relation $\simhat$
  bigger than $\sim$. We call $\simhat$ the \defn{closure} 
  of $\sim$. 
\end{lemma}

% remark: lemma true for compact topological space

\begin{proof}
%   We show that $s\;\simhat\; t$ { if and only if}
%   \begin{align*}
%     &\text{ there are } (s_n), (t_n)\subset S, \text{ such that }
%     s_n\sim t_n \text{ for all } n 
%     \\
%     &\text{ and } s=\lim s_n , \quad
%     t=\lim t_n. 
%   \end{align*}
%   Define $\simhat$ by the above. Note that $\{(s,t)\mid s\;\simhat \; 
%   t\}$ is the closure of $\{(s,t)\mid s\sim t\}$. Clearly $\simhat$ is
%   an equivalence relation. Every closed equivalence relation bigger
%   than $\sim$ thus has to be bigger than $\simhat$.  
  Consider the family $\{\simalpha\}_{\alpha\in I}$ of all closed
  equivalence relations bigger than $\sim$. This family is
  non-empty. Let $\simhat$ be their meet
  \begin{equation*}
    s\;\simhat\; t := \textstyle{\bigwedge} \simalpha \;\;=\; \bigcap
    \simalpha. 
  \end{equation*}
  In the last expression each $\simalpha$ is viewed as a subset of
  $S\times S$. Clearly $\simhat$ is the unique smallest closed
  equivalence relation bigger than $\sim$.
\end{proof}
Note that $\{(s,t)\mid s\;\simhat\; t\}$ is generally \emph{not} the
closure of 
$\{(s,t)\mid s\sim t\}$, which may fail to be transitive. 

\subsection{Equivalence relation induced by $\gamma$}
\label{sec:sim-induced-gamma}

A surjection $h\colon S\to S'$ induces an equivalence relation in a
natural way. Under mild assumptions the quotient $S/\!\sim$ is
homeomorphic to $S'$.  
\begin{lemma}[Equivalence relation induced by a map]
  \label{lem:usc_from_map}
  Let $S,S'$ be compact Hausdorff spaces, and $h\colon S\to S'$ a
  continuous surjection. Define an equivalence relation on $S$ by
  $s\sim t$ if and only if $h(s)=h(t)$. Then
  \begin{itemize}
  \item $S/\!\sim$ is homeomorphic to $S'$. The homeomorphism is given by
    $\varphi\colon [s]\mapsto h(s)$;
  \end{itemize}
  Assume now furthermore that there are continuous maps $f\colon S\to
  S$, $g\colon S'\to S'$ such that $h\circ f= g\circ h$, i.e., the
  following diagram commutes
  \begin{equation*}
    \xymatrix{
      S \ar[r]^f \ar[d]_h & S \ar[d]^h
      \\
      S' \ar[r]^g & S'.
    }
  \end{equation*}
  Then
  \begin{itemize}
  \item the map $f/\!\sim\,\colon S/\!\sim \,\to S/\!\sim$ given by $f/\!\sim\,
    \colon [s] \mapsto [f(s)]$ is well defined;
  \item it holds $\varphi \circ f/\!\sim\, = g\circ \varphi$, i.e.,
    $\varphi$ is a topological conjugacy between $f/\!\sim$ and $g$.  
  \end{itemize}
\end{lemma}

\begin{proof}
  The first statement is \cite[Theorem 3-37]{MR1016814}, the second
  statement follows immediately from the commutative diagram. Finally
  let $[s]\in S/\!\sim$ be arbitrary, then
  \begin{align*}
    g\circ \varphi([s])=g\circ h(s)= h\circ f(s)=
    \varphi([f(s)])=\varphi\circ f/\!\sim ([s]),
  \end{align*}
  i.e., the third statement holds. 
\end{proof}

From now on the equivalence relation $\sim$ on $S^1$ is the one
induced by $\gamma$,
\begin{equation*}
  s\sim t :\Leftrightarrow \gamma(s)=\gamma(t),  
\end{equation*}
for all $s,t\in S^1$. The previous lemma together with Theorem
\ref{thm:main} now yields Theorem \ref{thm:S1simS2}.

% \begin{cor}
%   \label{cor:S1relS2}
%   $S^1/\!\!\sim$ is homeomorphic to the sphere $S^2$ and the map $z^d/\!\!\sim\;
%   \colon S^1/\!\!\sim\; \to S^1/\!\!\sim$ is conjugate to $F$. More precisely we
%   have the following commutative diagram
%   \begin{equation*}
%     \xymatrix{
%       S^1/\!\sim \ar[r]^{z^d/\sim} \ar[d]_{h}
%       &
%       S^1/\!\sim \ar[d]^{h}
%       \\
%       S^2 \ar[r]_F & S^2.
%     }
%   \end{equation*}
%   The homeomorphism $h$ is given by $h\colon [s]\mapsto \gamma(s)$. 
% \end{cor}

Consider now the equivalence relations $\simn$ induced by $\gamma^n$
and their join $\siminfty$, 
\begin{align}
  \label{eq:def_simn}
  & s\simn t \;\text{ if and only if }\; \gamma^n(s)=\gamma^n(t);
  \\
  \label{eq:def_siminfty}
  & \siminfty\; := \textstyle{\bigvee} \Sim{n}\,, 
  \text{ meaning }  s\Sim{\infty} t \;
  \text{ if and only if }\; s\simn t \;\text{ for
    some } n;
\end{align}
for all $s,t\in S^1$. 

\begin{theorem}
  \label{thm:sim_usc_clos_siminfty}
  The equivalence relation $\sim$ on $S^1$ induced by $\gamma$
  (\ref{eq:eq_rel}) is the closure of $\siminfty$.  
\end{theorem}

To prove this theorem we will need some preparations first. Consider 
two points $s,t\in S^1$ such that $\gamma(s)=\gamma(t)$. We want to
show that $s,t$ are equivalent with respect to the closure of
$\siminfty$. Recall 
that $\gamma^n(s),\gamma^n(t)$ are contained in some $n$-edges by
construction. 
% For each $n$ consider $n$-arcs on $S^1$ containing $s,t$,
% $s\in a^n_i\subset S^1$, $t\in a^n_j\subset S^1$ (see Section
% \ref{sec:parametrizing-gn}). Consider the $n$-edges
% \begin{equation*}
%   E^n_s:= \gamma^n(a^n_i), \quad E^n_t:= \gamma^n(a^n_j).
% \end{equation*}
% In general there may be two $n$-edges $E^n_s$ and/or two $n$-edges
% $E^n_t$, it does not matter which ones we pick.
% Since $\gamma(a^n_s)\cap \gamma(a^n_t)\ni \gamma(s)=\gamma(t)$ we
% obtain from Lemma \ref{lem:paragn} (\ref{item:paragn_3})
% \begin{equation}
%   \label{eq:distEsEt}
%   \dist(E^n_s,E^n_t)\leq C \Lambda^n,
% \end{equation}
% here $C$ is independent of $s,t$, and $n$. 

\begin{lemma}
  \label{lem:NEsEt}
  There is a constant $N$ (independent of $s,t$, and $n$) such that 
  \begin{equation*}
    \gamma^n(s),\gamma^n(t) \text{ can be joined by at most $N$ $n$-edges,}
  \end{equation*}
  for all $s,t\in S^1$ with $\gamma(s)=\gamma(t)$. 
\end{lemma}

\begin{proof}
  Since $\varrho$ is a visual metric for $F$ (see
  Section~\ref{sec:visual-metric}) it holds $\varrho(x,y) \asymp
  \Lambda^{-m}$ for all $x,y\in S^2$. Here $\Lambda>1$ is the
  expansion factor of $\varrho$ and $m=m(x,y)$ is the smallest number
  for which there exist disjoint $m$-tiles $X^m\ni x, Y^m\ni y$.

%   Distances (with respect to the metric $\abs{x-y}_{\SC}$) may be
%   expressed in purely combinatorial terms, namely let $m=m(x,y)$ be the
%   smallest number such that there exist disjoint $m$-tiles $X^m\ni x,
%   Y^m\ni y$. Then (see \cite{expThurMarkov}) % Lemma 10.9 (2)) % Theorem 8.10 (2)
%   where $\abs{x-y}_{\SC}$ is the visual metric from Section
%   \ref{sec:visual-metric}; the constant $C(\asymp)$ is independent of
%   $x,y$. 
  Fix $s,t\in S^1$ with $\gamma(s)=\gamma(t)$. 
  %By \cite{peano} Lemma 4.5 (2) %\ref{lem:paragn}
  %(\ref{item:paragn_3}) 
  Since $\gamma^n$ converges uniformly to $\gamma$ as in
  \eqref{eq:gntog} it follows that
  \begin{equation*}
    \varrho(\gamma^n(s),\gamma^n(t)) \lesssim \Lambda^{-n},
  \end{equation*}
  with a constant $C(\lesssim)$ independent of $s,t$, and $n$. Thus,
  there is a constant $n_0$ such that the $(n-n_0)$-tiles
  \begin{equation*}
    X^{n-n_0}\ni \gamma^n(s),\; Y^{n-n_0}\ni \gamma^n(t) \text{
      are not disjoint.}
  \end{equation*}
  We now want to cover $X^{n-n_0},Y^{n-n_0}$ by $n$-tiles. The number
  required may be unbounded (since we do not assume that $\CC$
  is $F$-invariant). 
  Given an $n$-vertex $v$, let $W^n(v)$ be the union of all $n$-tiles
  containing $v$ (this is the closure of an \defn{$n$-flower} as
  defined in \cite[Chapter~5.4]{expThurMarkov}). Then
  \begin{equation*}
    \text{ every $(n-n_0)$-tile can be covered by $M$ sets $W^n(v)$,}
  \end{equation*}
  where the number $M$ is independent of $n,n_0$ (see
  \cite[Lemma~5.29]{expThurMarkov}). % check reference
  Clearly in any set $W^n(v)$ we can connect any two $n$-edges
  $E_1,E_2\subset W^n(v)$ by at most $2k$ $n$-edges in $W^n(v)$. Thus
  $\gamma^n(s),\gamma^n(t)$ can be connected by at most $4kM$
  $n$-edges. 
\end{proof}

\begin{proof}[Proof of Theorem \ref{thm:sim_usc_clos_siminfty}]
  Fix $s,t\in S^1$ with $\gamma(s)=\gamma(t)$. 
  According to the last lemma let
  \begin{equation*}
    E^n_1,\dots, E^n_N, \text{ with } \gamma^n(s)\in E^n_1, \gamma^n(t)\in
    E^n_N, 
  \end{equation*}
  be a chain of $n$-edges for each $n$. We can assume that
  $E^n_1,E^n_N$ are the images of $n$-arcs containing $s,t$ by
  $\gamma^n$. By taking a subsequence we can assume that $N$, the number of
  $n$-edges in this chain, is the same for all $n$. 

  Let $[u^n_j,v^n_j]\subset \R/\Z$ be the $n$-arc that is mapped to
  $E^n_j$, 
  $\gamma^n([u^n_j,v^n_j])=E^n_j$. Then
  \begin{equation*}
    v_j^n\simn u^{n}_{j+1} \text{ for } j=1,\dots N-1. 
  \end{equation*}
  Taking subsequences, we can assume that all the sequences
  $(v^n_j)_{n\in\N}$, $(u^n_j)_{n\in \N}$ converge. Thus it follows
  from \eqref{eq:diam_visual} (for all
  $j=1,\dots, N$) that 
  \begin{align*}
    &\lim_n u^n_j=\lim_n v^n_j=: v_j \text{ and}
    \\
    &\lim_n v^n_1=s, \quad \lim_n v^n_N=t.
  \end{align*}

  \medskip
  Now let $\simhat$ be the closure of $\siminfty$. Then
  \begin{equation*}
    s=v_1 \;\simhat\; v_2\;\simhat\; \dots\;\simhat\; v_N=t. 
  \end{equation*}
  Hence $s\;\simhat\; t$, meaning that $\simhat$ is bigger than $\sim$.  
\end{proof}

\section{The white and black equivalence relations}
\label{sec:laminations-2}

In view of Theorem \ref{thm:S1simS2} and Theorem
\ref{thm:sim_usc_clos_siminfty}, it is possible to recover the map $F$
(up to topological conjugacy) from the equivalence relations
$\Sim{n}$, i.e., the self intersections of the approximations
$\gamma^n$. Since our ultimate goal is to ``unmate'' the map $F$ into
two polynomials, we will decompose each equivalence relation $\Sim{n}$
into two equivalence relations $\Sim{n,\wt}, \Sim{n,\bt}$ (on
$S^1$). The equivalence relations $\Sim{n,\wt},\Sim{n,\bt}$ can be
obtained inductively from $\Sim{1,\wt},\Sim{1,\bt}$. Thus it is
possible to recover $F$ up to topological conjugacy from
$\Sim{1,\wt},\Sim{1,\bt}$, i.e., from finite data.

\smallskip Roughly speaking $\Sim{n,\wt}, \Sim{n,\bt}$ describe where
$\gamma^n$ touches itself ``in the white component'', or ``in the
black component''. We have
\begin{equation*}
  \Sim{n} \;= \;\Sim{n,\wt} \vee \Sim{n,\bt}.
\end{equation*}

\smallskip These two sequences of equivalence relations are closely
related to the two polynomials $P_\wt, P_\bt$ into which $F$
unmates. In fact $\Sim{1,\wt}, \Sim{1,\bt}$ yield the ``critical
portraits'' of polynomials $P_\wt, P_\bt$.  From the sequence
$\Sim{n,\wt}$ we construct an equivalence relation $\Sim{\wt}$. More
precisely $\Sim{\wt}$ will be the closure of the join of the $\Sim{n,
  \wt}$ (recall that $\sim$ was obtained from the relations $\Sim{n}$
in exactly the same way). We will then show that $\Sim{\wt}$ is the
equivalence relation induced by the Carath\'{e}odory semi-conjugacy
$\sigma_\wt\colon S^1\to \J_\wt$ of $P_\wt$. The corresponding
statements are obtained in exactly the same way for the black
polynomial $P_\bt$.

\smallskip The equivalence relations $\Sim{n,\wt},\Sim{n,\bt}$ will be
\defn{non-crossing}. This means the following. Given two distinct
equivalence classes $[s]_{n,\wt}, [t]_{n,\wt}\subset \R/\Z$ of
$\Sim{n,\wt}$ there are no points $a,c\in [s]_{n,\wt}$, $b,d\in
[t]_{n,\wt}$ such that
\begin{equation}
  \label{eq:def_non-crossing}
  0\leq a<b<c<d<1.   
\end{equation}
Note that
sometimes the term ``unlinked'' has been used for what we call
non-crossing. 

\smallskip 
It will be convenient to represent the equivalence relations
geometrically in two different ways. View $S^1\subset S^2$ as the
equator. The components (hemispheres) of $S^2\setminus S^1$ are
denoted by $S^2_{\wt},S^2_{\bt}$. The circle $S^1$ is positively oriented as
the boundary of the \emph{white} hemisphere $S^2_{\wt}$ (negatively
oriented as boundary of the \emph{black} hemisphere $S^2_{\bt}$). We equip
$S^2_{\wt},S^2_{\wt}$ with the hyperbolic metric (solely to be able to talk
about \emph{hyperbolic geodesics}).

For each
equivalence class $[s]_{n,\wt} \subset S^1\subset S^2$ of $\Sim{n,\wt}$
there is a \defn{leaf}, which is the hyperbolically convex hull of
$[s]_{n,\wt}$ in $S^2_{\wt}$. The \defn{set of all leaves} is the
\defn{lamination} $\LC^n_{\wt}$. Similarly the lamination $\LC^n_{\bt}$ is
defined in terms of $\Sim{n,\bt}$ (with leaves in $S^2_{\bt}$). 
That $\Sim{n,\wt}$ is non-crossing is equivalent with the fact that
distinct leaves in $\LC^n_{\wt}$ are disjoint. 

The second way to geometrically represent the equivalence relations
$\Sim{n,\wt},\Sim{n,\bt}$ is via \defn{gaps}. A \emph{white $n$-gap} is
the closure of one component of $S^2_{\wt}\setminus \bigcup \LC^n_{\wt}$. 
Gaps will correspond to tiles, leaves to vertices, and arcs to edges.  

\subsection{An example}
\label{sec:laminations}

% From Corollary \ref{cor:S1relS2} and Theorem
% \ref{thm:sim_usc_clos_siminfty} it follows that it is possible to
% recover $F$ from the equivalence
% relations $\simn$ (up to topological conjugacy). Here we show that
% $\simn$ can be obtained iteratively from $\gamma^1$. 

The construction
will be illustrated first for the example $g$ from 
\cite[Section 1.4]{peano}.
%\ref{sec:example}. 
It is a Latt\`{e}s map and may be obtained as follows. Consider the
equivalence relation $z\simeq \pm z + m + ni$, for all $m,n\in \Z$, on
$\C$. The map $z\mapsto 2z$ descends to the quotient
$\C/\!\simeq$. Then $g$ is topologically conjugate to $2z/\!\simeq
\;\colon \C/\!\simeq\, \to \C/\!\simeq$. The degree of the map is
$4$. This map is best represented as follows. Glue two unit squares
together along their boundaries to form a \emph{pillow}, which is a
topological sphere. Divide each side of this pillow (i.e., each unit
square) into four squares. Map each of these small squares to one of
the sides as indicated in Figure~\ref{fig:mapg} of the pillow yields
the map $g$ (up to topological conjugacy). Note that the vertices as
the pillow are the postcritical points of the map. We choose the
boundary of the pillow as the curve $\CC$. The two sides of the pillow
then are the two $0$-tiles. The $8$ small squares are the $1$-tiles. 

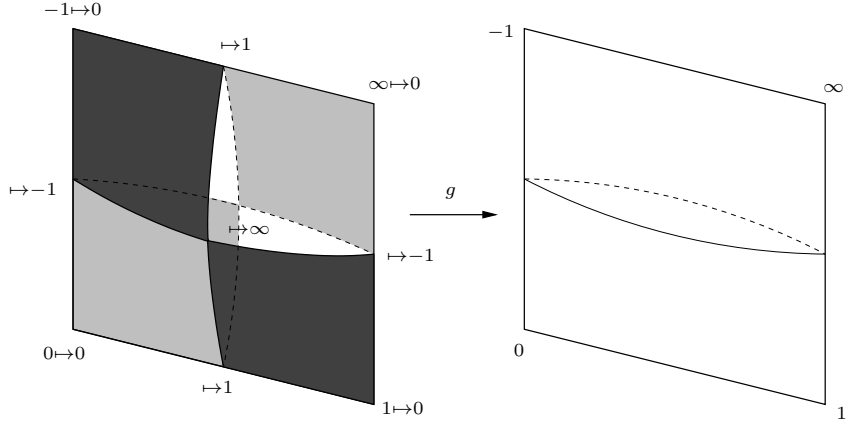
\begin{figure}
  \centering
  \includegraphics[scale=0.5]{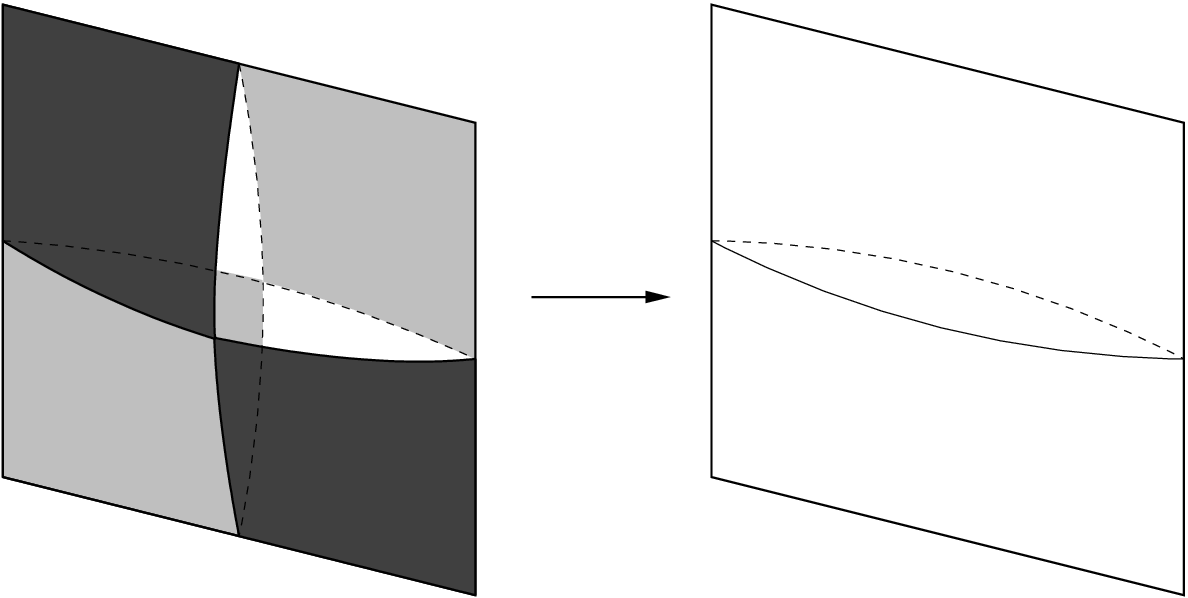}
  \begin{picture}(10,10)
    \put(-122,19){$\scriptstyle 0$}
    \put(0,-5){$\scriptstyle 1$}
    \put(-132,140){$\scriptstyle -1$}
    \put(-5,119){$\scriptstyle \infty$}
    \put(-172,-3){$\scriptstyle 1\mapsto 0$}
    \put(-240,5){$\scriptstyle \mapsto 1$}
    \put(-300,18){$\scriptstyle 0\mapsto 0$}
    \put(-313,80){$\scriptstyle \mapsto -1$}
    \put(-300,148){$\scriptstyle -1\mapsto 0$}
    \put(-233,135){$\scriptstyle \mapsto 1$}
    \put(-177,120){$\scriptstyle \infty \mapsto 0$}
    \put(-170,55){$\scriptstyle \mapsto -1$}
    \put(-148,80){$\scriptstyle g$}
    \put(-230, 65){$\scriptstyle \mapsto \infty$}
  \end{picture}
  \caption{The Latt\`{e}s map $g$.}
  \label{fig:mapg}
\end{figure}

\begin{figure}
  \centering
  \begin{overpic}
    [width=12cm, tics=20,
     %grid
    ]
    {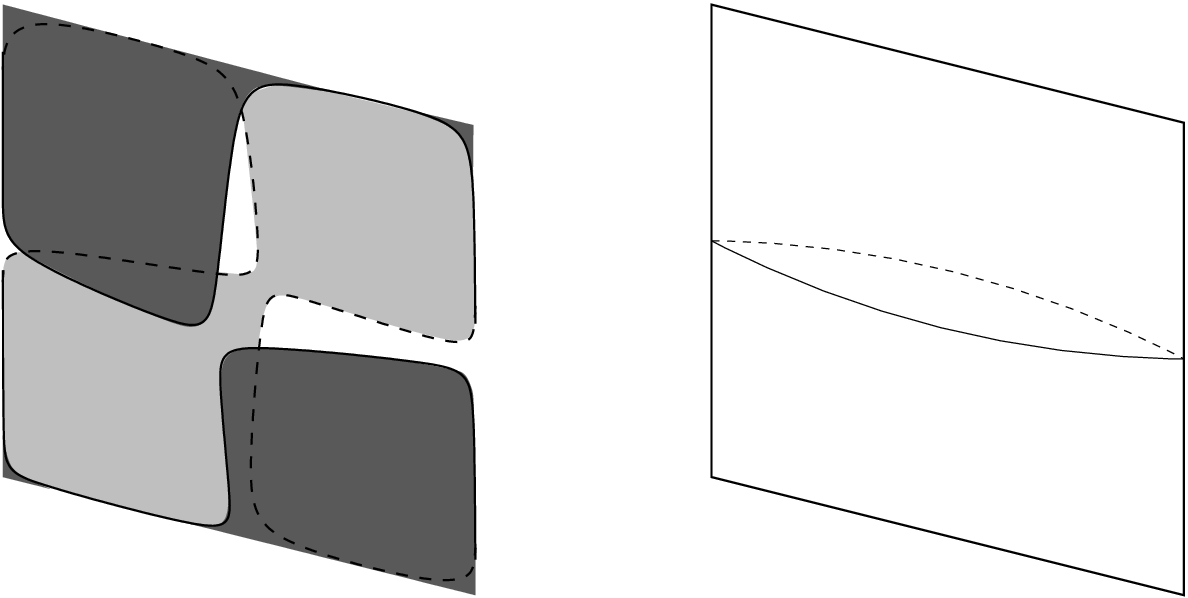}
    \put(60,5){$\gamma^0=\CC$}
    \put(5,5){$\gamma^1$}
  \end{overpic}
  \caption{First approximation of invariant Peano curve.}
  \label{fig:g1_g}
  \includegraphics[width=12cm]{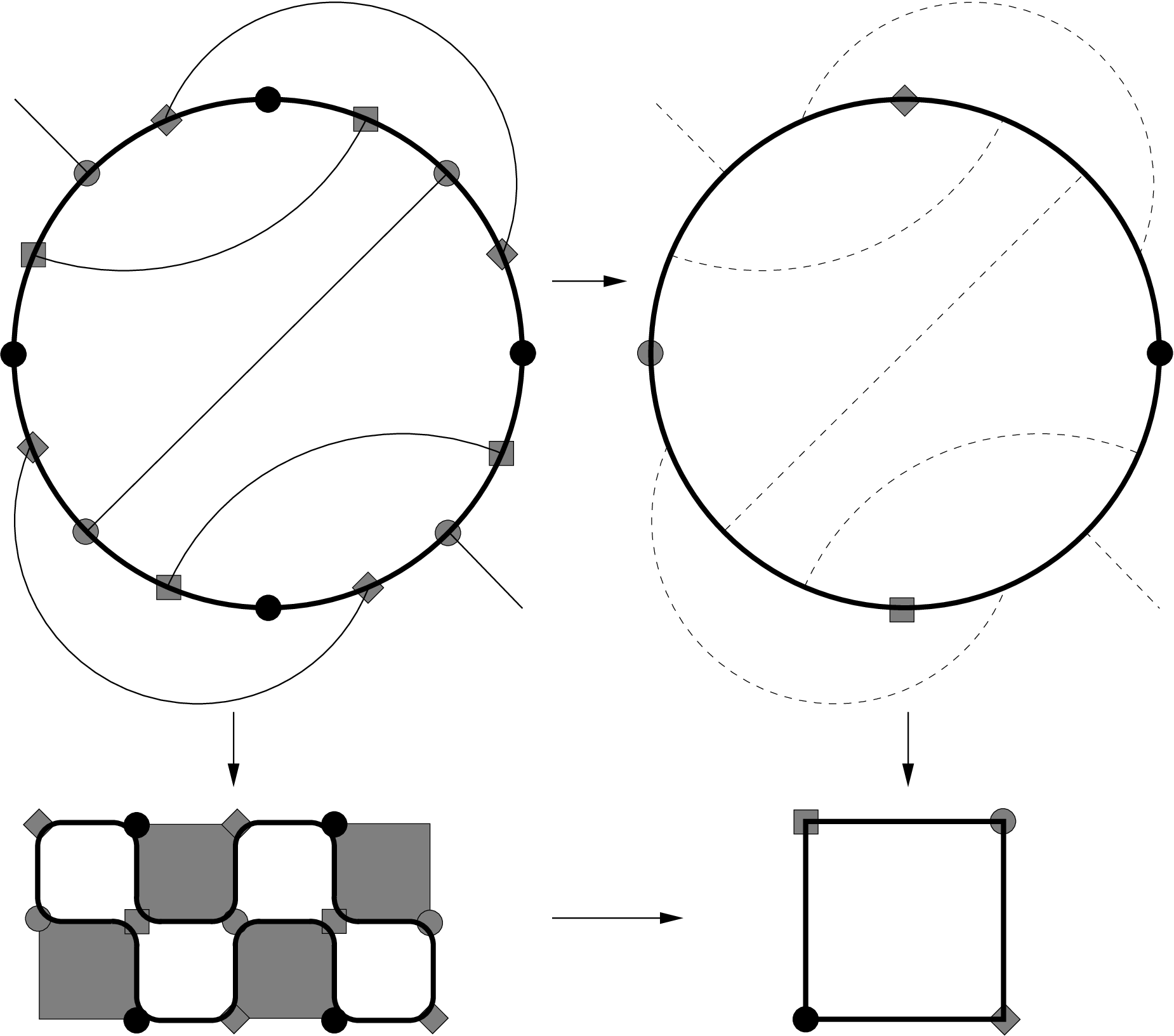}
  \begin{picture}(10,10)
    \put(-185,193){$\scriptstyle{0}$}
    \put(-0,193){$\scriptstyle{0}$}
    \put(-317,-7){$\scriptstyle{p_0=0}$}
    \put(-122,-7){$\scriptstyle{p_0=0}$}
    \put(-179,228){$z^4$}
    \put(-270,80){$\gamma^1$}
    \put(-74,80){$\gamma^0$}
    \put(-170,40){$g$}
  \end{picture}
  \caption{The laminations $\LC^1_{w,g},\LC^1_{b,g}$.}
  \label{fig:circleeq}
\end{figure}

The $0$-th approximation of the invariant Peano curve is
$\gamma^0=\CC$. The first approximation $\gamma^1$ is shown in the
left of Figure~\ref{fig:g1_g}. It should be emphasized that we could
have used several other choices for $\gamma^1$. Each of these would
have resulted in a different invariant Peano curve and different pairs
of polynomials $P_\wt, P_\bt$ into which $g$ unmates.

For illustrative purposes it is best to lift objects
to $\C$, technically speaking we work in the \emph{orbifold
  covering}. The (lifts of the) approximations $\gamma^1,\gamma^0$ are
illustrated in Figure~\ref{fig:circleeq}. Practically speaking this
means we cut the pillow from Figure~\ref{fig:g1_g} so that the the two
halves of the ``back of the pillow'' are folded to the left and right,
to obtain the picture in the lower left of
Figure~\ref{fig:circleeq}. 

The parametrization of $\gamma^1$ and $\gamma^0$ is shown as well. 
Additionally the connections of the $1$-tiles is shown (see Section
\ref{sec:connection-n-tiles}). Namely the $4$ white $1$-tiles are connected at the $3$
$1$-vertices in the middle. The $1$-angles $2/16$ and $10/16$ are both
mapped by $\gamma^1$ to the point in the middle. The set
$\{2/16,10/16\}$ forms one equivalence class of $\Sim{1,\wt}$. The
non-trivial equivalence classes (i.e., the ones which are not
singletons) are
\begin{equation}
  \label{eq:sim1w_exampleg}
  \text{equivalence classes of $\Sim{1,\wt}$: }
  \left\{\frac{2}{16},\frac{10}{16}\right\}, 
  \left\{\frac{3}{16},\frac{7}{16}\right\},
  \left\{\frac{11}{16}, \frac{15}{16}\right\}. 
\end{equation}
The curve $\gamma^1$ also ``touches itself in the black component''. This
however is more easily seen in Figure~\ref{fig:g1_g} than in the
orbifold covering, i.e., in Figure~\ref{fig:circleeq}. The
corresponding non-trivial equivalence classes of $\Sim{1,\bt}$ are
\begin{equation}
  \label{eq:sim1b_exampleg}
  \text{equivalence classes of $\Sim{1,\bt}$: }
  \left\{\frac{1}{16},\frac{5}{16}\right\}, 
  \left\{\frac{6}{16},\frac{14}{16}\right\},
  \left\{\frac{9}{16}, \frac{13}{16}\right\}. 
\end{equation}

We now view $S^1=\R/\Z$ (i.e., the domain of $\gamma^1, \gamma^0$) as
the boundary of a disk, which we identify with the hemisphere
$S^2_\wt$. 
The hyperbolic geodesic (in $S^2_\wt$) connecting the two points of one 
equivalence class of $\Sim{1,\wt}\subset S^1$ is a \defn{leaf}.  
The lamination $\LC^1_{\wt}$ is the set of these three leaves. 

A \emph{white $1$-gap} $G$ is the closure of one component of $S^2_\wt
\setminus \bigcup \LC^1_\wt$. The set of all white $1$-gaps is denoted
by $\G^1_\wt$. 

\smallskip Similarly we connect the points of one equivalence class of
$\Sim{1,\bt}$ by a hyperbolic geodesic in the outside of the circle
(which we identify with the hemisphere $S^2_{\bt}$). The lamination
$\LC^1_{\bt}$ is the set of the three leaves thus obtained, the
closure of each component of $S^2_\bt$ is a black $1$-gaps. This is
illustrated in the top left of Figure~\ref{fig:circleeq}.

\smallskip Recall that each $1$-arc was mapped by $\gamma^1$ to a
$1$-edge. Each white/black $1$-gaps correspond to a white/black
$1$-tile. Note that each $1$-gap $G$ has $4$ $1$-arcs in its boundary,
one of each type. Thus $G\cap S^1$ is mapped by
$\mu(t)=4t$. Furthermore these four $1$-arcs are mapped by $\gamma^1$
to $1$-edges contained in the same $1$-tile.  Each $1$-leave
corresponds to a $1$-vertex, i.e., each equivalence class (of
$\Sim{1,\wt}$ or of $\Sim{1,\bt}$) is mapped by $\gamma^1$ to a
$1$-vertex. Note however that several distinct $1$-leaves may be
mapped to the same $1$-vertex (this corresponds to distinct critical
points being identified in the mating). Thus the combinatorial
description via tiles, edges, and vertices is given via gaps, arcs,
and leaves.

\smallskip The equivalence relations $\Sim{n,\wt},\Sim{n,\bt}$ are
constructed in the same fashion. They can however be obtained
inductively from $\Sim{1,\wt},\Sim{1,\bt}$ as follows. Consider a
\defn{white $1$-gap} $G$. The following holds for all $s,t \in
S^1=\R/\Z$:
\begin{align*}
  s\Sim{2,\wt} t &\text{ if and only if}
  \\
  &\text{there is a white $1$-gap } G\ni s,t \text{ and } 
  \mu(s)\Sim{1,\wt} \mu(t);
\end{align*}
thus the second white equivalence relation $\Sim{2,\wt}$ is obtained
from $\Sim{1,\wt}$. 
In the same fashion all equivalence relations $\Sim{n,\wt}$ can be
inductively constructed from $\Sim{1,\wt}$; similarly all $\Sim{n,\bt}$
can be constructed from $\Sim{1,\bt}$. Thus the lists
(\ref{eq:sim1w_exampleg}), (\ref{eq:sim1b_exampleg}) contain all
information 
to recover the map $g$ up to topological conjugacy (by Theorem
\ref{thm:sim_usc_clos_siminfty} and \ref{thm:S1simS2}). They are
called the white and black \defn{critical portraits} of $g$.

\subsection{Definition of $\Sim{n,\wt}, \Sim{n,\bt}$}
\label{sec:laminations-1}

We now define the equivalence relations $\Sim{n,\wt}, \Sim{n,\bt}$  in
general. They will be 
defined in terms of which white/black $n$-tiles are connected at some
$n$-vertex, i.e., in terms of the \emph{connection of $n$-tiles} (see
Section \ref{sec:connection-n-tiles}). 

\medskip
Let $v$ be an $n$-vertex. Consider one block $b\in \pi^n_{\wt}(v)\cup
\pi^n_{\bt}(v)$ (from the cnc-partition defining the connection at
$v$). Let $X_0,\dots,X_{2m-1}$ be the $n$-tiles containing $v$. We
call an $n$-edge $E\ni v$ \defn{incident to $b$ at $v$} if
\begin{equation*}
  E\subset X_i, \quad \text{ where } i\in b.
\end{equation*}
Each $n$-edge is incident to exactly one block $b\in \pi^n_{\wt}(v)$ and
one block $c\in \pi^n_{\bt}(v)$. 
Succeeding $n$-edges $E,E'$ (at $v$) are incident to the same
white/black block 
(see \cite[Lemma 6.13]{peano}). %\ref{lem:arcs_successors}).  

Consider 
an $n$-angle $\alpha^n_j\in S^1$ such that $\gamma^n(\alpha^n_j)= v$. It is
called \defn{incident to the block $b\in
  \pi^n_{\wt}(v)\cup \pi^n_{\bt}(v)$} at $v$ if the $n$-arc
$a^n_j=[\alpha^n_j,\alpha^n_{j+1}]$ (as well as
$a^n_{j-1}=[\alpha^n_{j-1},\alpha^n_j]$) is mapped by $\gamma^n$ to
an $n$-edge incident to $b$ at $v$. In this case we also say that the
$n$-edges $E= \gamma^n(a^n_{j-1})$ and $E'=\gamma^n(a^n_j)$ are
incident to the $n$-angle $\alpha^n_j$ (at $v$) and vice versa. 
Recall that the set of $n$-angles
$\A^n=\{\alpha^n_j\}\subset S^1$ is the set of all points that are
mapped by $\gamma^n$ to some $n$-vertex.

%Points $s\notin \{\alpha^n_j\}$ are never
%incident to any block $b\in \pi^n_{\wt}(v)\cup \pi^n_{\bt}(v)$.  

% \begin{definition}[Adjacent successors]\
%   \label{def:adjacent_edges}
%   Consider two pairs of succeeding $n$-edges $D,D'$ and $E,E'$ at
%   $v$. These pairs are called \defn{adjacent at the white
%     component} (at $v$), if
%   \begin{equation*}
%     D,D' \text{ and } E,E' \text{ are both incident to the same block
%     } b\in\pi^n_{\wt}(v);
%   \end{equation*}
%   \defn{adjacent at the black component} if
%   \begin{equation*}
%     D,D' \text{ and } E,E' \text{ are both incident to the same block
%     } c\in\pi^n_{\bt}(v).
%   \end{equation*}
% \end{definition}

% Recall how the points $\{\alpha^n_j\}\subset S^1$ divided the circle
% into $n$-arcs $\{a^n_j\}$, where
% $a^n_j=[\alpha^n_j,\alpha^n_{j+1}]\subset S^1$
% (see Section \ref{sec:parametrizing-gn}). Succeeding $n$-arcs
% are mapped by $\gamma^n$ to succeeding $n$-edges.

% Consider pairs of succeeding $n$-arcs $a^n_{i-1}, a^n_{i}\subset
% S^1$ and $a^n_{j-1},a^n_{j}\subset S^1$; and the (succeding)
% $n$-edges $D:=\gamma^n(a^n_{i-1}), D':=\gamma^n(a^n_i)$ and
% $E:=\gamma^n(a^n_{j-1}),E':=\gamma^n(a^n_j)$. 

\begin{definition}[Equivalence
  relations $\Sim{n,\wt}, \Sim{n,\bt}\,$]
  \label{def:simnw_simnb}
  Let the connection at any $n$-vertex $v$ be given by the
  cnc-partition $\pi^n_{\wt}(v)\cup \pi^n_{\bt}(v)$.  
  
  Define the equivalence relations $\Sim{n,\wt},\Sim{n,\bt}$ on $\A^n$ by
  \begin{align*}
    \alpha \Sim{n,\wt} \alpha'  & :\Leftrightarrow %\text{ if and only if } 
    \alpha,\alpha' \text{ are incident to
      the same block } b\in \pi^n_{\wt}(v),
    \\
    &\phantom{XXXXXXXXXXXXXXXX}\text{ at some $v\in \V^n$};
    \\
    \alpha \Sim{n,\bt} \alpha' & :\Leftrightarrow %\text{ if and only if } 
    \alpha,\alpha' \text{ are incident to
      the same block } b\in \pi^n_{\bt}(v),
    \\
    &\phantom{XXXXXXXXXXXXXXXX}\text{ at some $v\in \V^n$};    
  \end{align*}
  for all $\alpha,\alpha'\in \A^n$. Thus there is a one-to-one
  correspondence 
  between blocks $b\in \pi^n_{\wt}(v)$ and equivalence classes with
  respect to $\Sim{n,\wt}$; analogously between blocks $b\in
  \pi^n_{\bt}(v)$ and equivalence classes with respect to
  $\Sim{n,\bt}$. Note that $n$-angles $\alpha,\alpha'$ such that
  $\gamma^n(\alpha)\neq \gamma^n(\alpha')$ are never equivalent with
  respect to $\Sim{n,\wt}$ or $\Sim{n,\bt}$.  

  It will be convenient to consider the equivalence relations
  $\Sim{0,w},\Sim{0,b}$ on $\A^0$ as well. They are defined to be
  trivial, meaning that each equivalence class is a singleton. 
\end{definition}

We can view $\Sim{n,\wt},\Sim{n,\bt}$ as equivalence relations on
$S^1$. Namely each $s\in S^1\setminus \A^n$ is equivalent to itself
and only to itself with respect to both $\Sim{n,\wt}$ and $\Sim{n,\bt}$. 
Let us record that $\Sim{n,\wt}, \Sim{n,\bt}$ \emph{generate} the
equivalence relation $\Sim{n}$, which is an immediate consequence of
\cite[Lemma 6.2]{peano},
\begin{equation}
  \label{eq:simn-simnw-simnb}
  \Sim{n}\;=\; \Sim{n,\wt} \vee \Sim{n,\bt}.
\end{equation}

%In the next lemma we give an alternative description of $\Sim{n,\wt}$. 

The angles of each equivalence class $[\alpha^n]_{n,\wt}\subset S^1$
inherit the cyclical ordering from $S^1$.   
\begin{definition}[Succeeding]
  \label{def:succeeding}
  Two $n$-angles $\alpha^n,\tilde{\alpha}^n$ are called
  \defn{succeeding} (with respect to $\Sim{n,\wt}$) if
  \begin{align*}
    &\alpha^n \Sim{n,\wt} \tilde{\alpha}^n \quad \text{and}
    \\
    &\tilde{\alpha}^n \text{ is the successor to $\alpha^n$ in }
    [\alpha^n]_{n,\wt}. 
  \end{align*}
  This means that the open arc $(\alpha^n,\tilde{\alpha}^n)\subset
  S^1$ does not contain any $n$-angle from $[\alpha^n]_{n,\wt}$. The
  $n$-angle 
  $\tilde{\alpha}^n$ is the \defn{predecessor} to $\alpha^n$. 

  A finite sequence $\alpha^n_1,\dots,\alpha^n_k$ is called
  \defn{succeeding} (with respect to $\Sim{n,\wt}$) if
  $\alpha^n_j,\alpha^n_{j+1}$ are succeeding with respect to
  $\Sim{n,\wt}$ (for all $1\leq j\leq k-1$). Clearly
  $\alpha^n\Sim{n,\wt}\tilde{\alpha}^n$ if and only if there is a
  sequence of succeeding $n$-angles from $\alpha^n$ to
  $\tilde{\alpha}^n$ ($\A^n$ is a finite set). 
\end{definition}
Two succeeding angles correspond to a hyperbolic geodesic which forms a
boundary arc of some leaf of the lamination. 

% \begin{definition}[Succeeding]
%   \label{def:succeeding}
%   Two $n$-angles $\alpha^n_i,\alpha^n_j\in L\in \LC^n_{\wt}$ are called
%   \defn{succeeding} vertices of the leaf $L$, if the open arc between
%   them 
%   $(\alpha^n_i,\alpha^n_j)\subset S^1$  
%   does not contain any other vertex of $L(v,b)$ (another $n$-angle in
%   $[\alpha^n_j]_{n,\wt}$).  

% \end{definition}

The nect lemma shows that the cyclical ordering may 
The following lemma gives an alternative way to construct
$\Sim{n,\wt}$. 
\begin{lemma}
  \label{lem:lamination_succeeding}
  Consider two $n$-angles $\alpha^n_i,\alpha^n_j\in S^1$ that are mapped by
  $\gamma^n$ to the same $n$-vertex $v$. Then 
  \begin{align*}
    &\alpha^n_i,\alpha^n_j \text{ are \emph{succeeding}} 
    \intertext{ if and only if}
    &\text{The $n$-arcs $a^n_i=[\alpha^n_i,\alpha^n_{i+1}],
      a^n_{j-1}=[\alpha^n_{j-1},\alpha^n_j]$ are mapped by $\gamma^n$}
    \\
    &\text{to $n$-edges in the \emph{same white %(black) 
        $n$-tile $X$}.} 
  \end{align*}
  In this case the $n$-edge $E'=\gamma^n(a^n_i)$ succeeds
  $E=\gamma^n(a^n_{j-1})$ in $\partial X$.   
\end{lemma}

Here $\partial X$ is oriented as positively as boundary of $X$ as
usual. The situation is illustrated in
Figure~\ref{fig:lem_succed_angles}. 

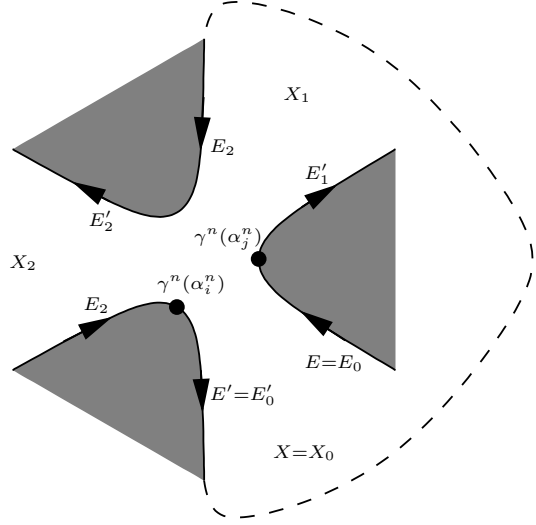
\begin{figure}
  \centering
  \begin{overpic}
    [width=7cm, tics=20,
    %grid
    ]
    {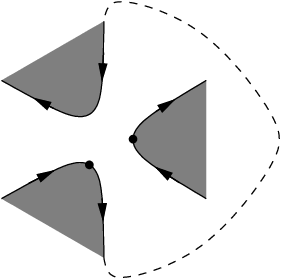}
    \put(50,12){$\scriptstyle{X=X_0}$}
    \put(52,80){$\scriptstyle{X_1}$}
    \put(0,48){$\scriptstyle{X_2}$}
    \put(28,44){${\scriptstyle \gamma^n(\alpha^n_i)}$}
    \put(35,53){${\scriptstyle \gamma^n(\alpha^n_j)}$}
    \put(38,23){$\scriptstyle{E'=E'_0}$}
    \put(56,29){$\scriptstyle{E=E_0}$}
    \put(56,65){$\scriptstyle{E'_1}$}
    \put(38,70){$\scriptstyle{E_2}$}
    \put(15,56){$\scriptstyle{E'_2}$}
    \put(14,40){$\scriptstyle{E_2}$}
  \end{overpic}
  \caption{Illustration to Lemma~\ref{lem:lamination_succeeding}}
  \label{fig:lem_succed_angles}
\end{figure}

\begin{proof}
  Recall from Section~\ref{sec:approximations} that the curve
  $\gamma^n$ traverses each $n$-edge $E$ positively as boundary of the
  white $n$-tile that $E$ is contained in.

  Note that both conditions in the statement of the lemma imply that
  $\alpha^n_i,\alpha^n_j$ are incident to the same block
  $b\in\pi^n_{\wt}(v)$.
  %$b\in\pi^n_{\wt}(v)\cup
  %\pi^n_{\bt}(v)$. W.l.og. $b\in \pi^n_{\wt}(v)$.

%   Assume $\alpha^n_i,\alpha^n_j$ are succeeding vertices of a leaf
%   $L(v,b) \in \LC^n_{\wt}$.  
  Let $X_0,\dots, X_{m-1}$ be the white $n$-tiles incident to $b$ (at
  $v$), cyclically ordered around $v$. Let $E_j,E'_j\subset X_j$ be
  the $n$-edges with terminal/initial point $v$. Then the cyclical
  order of the $n$-edges incident to $b$ around $v$ is
  $E_0',E_0,E_1',E_1,\dots, E_{m-1}',E_{m-1}$. Recall from Section
  \ref{sec:connection-n-tiles} that $n$-edges are
  succeeding in $\gamma^n$ (at $v$) if and only if they are succeeding
  with respect to $\pi^n_{\bt}(v)\cup\pi^n_{\wt}(v)$. Thus $E_l$ is succeeded
  by $E_{l+1}'$ in $\gamma^n$ (index $l$ is taken $\bmod\, m$), by
  definition (see Section \ref{sec:succeeding-edges}). This means that
  $E_l$ and $E'_{l+1}$ are both incident to the same $n$-angle
  $\alpha$ incident to $b$. The proof thus will be finished if we can
  show that in $\gamma^n$ each $n$-edge $E_l'$ is followed by $E_l$,
  i.e., in $\gamma^n$ between $E'_l$ and $E_l$ there is no other
  $n$-edge incident to $b$. 

  \smallskip
  Recall that $\gamma^n$ is a non-crossing Eulerian circuit. This
  means we can change $\gamma^n$ slightly in a neighborhood of each
  $n$-vertex to obtain a Jordan curve $\gamma^n_\epsilon$. If $E'_l$
  would be followed in 
  $\gamma^n$ by $E_m$ (where $m\neq l$) (as shown in
  Figure~\ref{fig:lem_succed_angles}) this would not be
  possible. Indeed let $\alpha$ be the $n$-angle incident to $E'_l$
  and $\beta$ be the $n$-angle incident to $E_m$ at
  $v$. We can connect $\gamma^n_\epsilon(\alpha)$ and
  $\gamma^n_\epsilon(\beta)$ by a an arc that does not cross
  $\gamma^n_\epsilon$ to form a Jordan curve $\Gamma$. Both components
  of $\Gamma$ then contain parts of the curve $\gamma^n_\epsilon$,
  which is a contradiction.  
  
  % The $n$-edges $E_l,E'_{l+1}$ as above divide the Eulerian circuit
  % $\gamma^n$ into $m$ chains between $E'_l$ and $E_l$. Recall that
  % $\gamma^n$ is non-crossing 
  % (we can change $\gamma^n$ slightly in a neighborhood of each
  % $n$-vertex to obtain a Jordan curve). Thus the chain from $E'_{l}$
  % to $E_{l}$ on $\gamma^n$ does not contain any other $n$-edge
  % incident to $b$ (i.e., any $E_k, E_k'$ for $k\neq l$). 

  % Then $\alpha^n_i,\alpha^n_j$ are succeeding with respect to
  % $\Sim{n,\wt}$ if and only if for $a^n_i=[\alpha^n_i,\alpha^n_{i+1}]$,
  % $a^n_{j-1}=[\alpha^n_{j-1},\alpha^n_j]$ it holds
  % $\gamma^n(a^n_i)=E'_l$, $\gamma^n(a^n_{j-1})=E_l$ (for some
  % $l=0,\dots, m-1$). Note that
  % $E'_l,E_l$ are contained in the same white $n$-tile by definition. 
  %one such chain
  %connects $E'_i$ to $E_j$. 
  % Consider a geometric representation of
%   $\gamma^n$ (the boundary of $K^n_{w,\epsilon}$, the cluster of white
%   $n$-tiles).  
  %Since a geometric representation of
  %$\gamma^n$ is a Jordan curve this is the case if and only if
  %$E_i',E_j$ are contained in the same white $n$-tile $X$. 
\end{proof}

Note that the previous lemma remains valid if $[\alpha^n_i]_{n,\wt}$
consists of a single point, then of course $\alpha^n_i=\alpha^n_j$. 

\subsection{Laminations and gaps}
\label{sec:laminations-gaps}

We now define the laminations and gaps. These are different geometric
representations of the equivalence relations $\Sim{n,\wt},
\Sim{n,\bt}$. 

We view $S^1\subset S^2$ as the equator. Recall that the two
hemispheres $S^2_\wt, S^2_\bt$ (i.e., complementary components of
$S^2\setminus S^1$) are equipped with the hyperbolic metric. 

\begin{definition}[Laminations $\LC^n_{\wt},\LC^n_{\bt}$]
  \label{def:Ln} 
  Let $\alpha\in \A^n$ and $[\alpha]_{n,\wt}\subset S^1$ be its equivalence
  class with respect to $\Sim{n,\wt}$. 
  A \defn{leaf} $L=L([\alpha]_{n,\wt})$ is the hyperbolically convex hull of
  $[\alpha]_{n,\wt}$ in $S^2_{\wt}$. If $[\alpha]_{n,\wt}$ corresponds to the 
  block $b\in \pi^n_{\wt}(v)$ we sometimes write $L=L(v,b)$ to indicate
  ``where the leaf comes from''. The set of all such leaves is the
  \defn{white lamination} $\LC^n_{\wt}$.  
%   Let $[s]$
%   Fix a block $b\in
%   \pi^n_{\wt}(v)$. The \defn{leaf} $L(v,b)$ is the \emph{ideal $m$-gon} in
%   $S^2_{\wt}$ bounded by points $\alpha^n_{j_1},\dots,\alpha^n_{j_m}\in
%   S^1$ incident to $b$. 

  If $\#[\alpha]_{n,\wt}=m$ the leaf $L=L([\alpha]_{n,\wt})$ is an
  \defn{ideal $m$-gon}.  
  If $[\alpha]_{n,\wt}$ consists of two points the corresponding leaf $L$ is the
  \emph{hyperbolic geodesic} (in $S^2_{\wt}$) between these points. 
  If $[\alpha]_{n,\wt}$ consists of a single point the corresponding
  leaf is $L=\{\alpha\}$. 

  \smallskip
  The \defn{black lamination} $\LC^n_{\bt}$ (containing $m$-gons in $S^2_{\bt}$) is
  constructed similarly from the equivalence classes of $\Sim{n,\bt}$. 
\end{definition}

\begin{remark}
  A \emph{lamination} is usually defined to be a closed set of
  disjoint \defn{geodesics}. One obtains a lamination in this standard
  sense 
  by taking the boundaries of all our leaves. 
  %A leaf $L(v,b)$ as
  %above which is not a geodesic is
  %usually called a \defn{gap} of the lamination. 
  We use this slightly
  non-standard terminology, since we do not want to make a
  distinction between hyperbolic geodesics and ideal $m$-gons.  
\end{remark}

\begin{remark}
  \label{rem:LnwLnb}
  We will formulate the following lemmas only for the \emph{white}
  equivalence relations $\Sim{n,\wt}$, and the corresponding white
  laminations $\LC^n_{\wt}$. In each
  case there is an obvious analog for \emph{black} equivalence
  relations $\Sim{n,\bt}$, and the corresponding black laminations
  $\LC^n_{\bt}$. More precisely one 
  has to replace ``white'' by ``black'', and each index ``$\wt$'' by
  ``$\bt$''. Furthermore in Lemma \ref{lem:prop_Ds} ``cyclically'' has
  to be replaced by ``anti-cyclically''. 
\end{remark}

\smallskip
Recall that $\mu\colon \R/\Z\to \R/\Z$, $\mu(t)=dt (\bmod 1)$. 
The \defn{image} of
an ideal $m$-gon $L$ (in $S^2_{\wt}$ or $S^2_{\bt}$) with vertices
$\alpha^n_{j_1},\dots, \alpha^n_{j_m}$ by $\mu$ is the 
ideal $\widetilde{m}$-gon $\tilde{L}$ with vertices
$\mu(\alpha^n_{j_1}),\dots,\mu(\alpha^n_{j_m})$. The ideal
$\widetilde{m}$-gon $\tilde{L}$ lies in the same hemisphere (in
$S^2_{\wt}$ or $S^2_{\bt}$) as $L$. We write $\mu(L)=\tilde{L}$.

\begin{lemma}[Properties of $\Sim{n,\wt}$ and $\LC^n_{\wt}$]
  \label{lem:prop_L}
  The equivalence relations $\Sim{n,\wt}$ and the laminations $\LC^n_{\wt}$
  satisfy the following. 
  \begin{enumerate}[\upshape{($\LC^n$} 1\upshape{)}]
  \item 
    \label{item:LinQ}
    The non-trivial equivalence classes are in $\Q$ and are mapped by
    the iterate $\mu^n$ to a single point,
    \begin{equation*}
      [\alpha^n]_{n,\wt} \subset \Q, \quad
      \mu^n([\alpha^n]_{n,\wt})=\{\alpha^0\}\in \A^0,
    \end{equation*}
    for all $\alpha^n\in \A^n$.
%   \item 
%     \label{item:propL1}
%     They \emph{generate} the equivalence relation $\simn$; meaning
%     $\simn$ is the smallest equivalence relation such that 
%     \begin{equation*}
%       s\simn t \text{ whenever } s,t \text{ are connected by a
%         leaf in $\LC^n_{\wt}\cup \LC^n_{\bt}$.}
%     \end{equation*}
  \item 
    \label{item:propL2}
    The equivalence relations $\Sim{n,\wt}$ %, $\Sim{n,\bt}$
    are \defn{non-crossing} (see (\ref{eq:def_non-crossing})). This
    means that the leaves in $\LC^n_{\wt}$ %(and $\LC^n_{\bt}$)  
    are \emph{disjoint}.
%   \item 
%     \label{item:propL3}
%     There is a bijection of components $D$ of $S^2_{\wt}\setminus \bigcup
%     \LC^n_{\wt}$ (of $S^2_{\bt}\setminus \bigcup \LC^n_{\bt}$) to white (black)
%     $n$-tiles $X$, such that
%     \begin{equation*}
%       \text{each $n$-arc $a^n\subset \partial D\cap S^1$ is mapped by
%         $\gamma^n$ to an $n$-edge $E\subset \partial X$.}
%     \end{equation*}
%     \bigskip
%     Let $D$ be one component of $S^2_{\wt}\setminus \bigcup\LC^n_{\wt}$ (of
%     $S^2_{\bt}\setminus \bigcup\LC^n_{\bt}$), and $a^n_{j_0},\dots, a^n_{j_{k-1}}$ be the
%     $n$-arcs in $\partial D\cap S^1$. Then the $n$-edges 
%     \begin{equation*}
%       E^n_j:= \gamma^n(a^n_{j_l}),\; l=0,\dots, k-1,
%     \end{equation*}
%     all lie in the \emph{same} white (black) $n$-tile $X^n$. 
  \item
    \label{item:propLnum_L}
    Let $\{L_j\}=\LC^n_{\wt}$, % (or $\{L_j\}=\LC^n_{\bt}$), 
    and each $L_j$ an ideal
    $m_j$-gon. Then
    \begin{align*}
      \sum_j m_j= k d^n,
      \quad
      \sum_j (m_j-1) = d^n -1.
    \end{align*}
  \item 
    \label{item:propL6}
    The lamination $\LC_{\wt}^{n+1}$ 
    is mapped by $\mu$
    to $\LC^n_{\wt}$. 
    This means 
    \begin{align*}
      &L^{n+1}\in \LC^{n+1}_{\wt} 
      \Rightarrow \mu(L^{n+1})
      \in \LC^n_{\wt}, 
      \intertext{equivalently}
      &\mu([\alpha^{n+1}]_{n+1,\wt})
      = [\mu(\alpha^{n+1})]_{n,\wt}
      = [\alpha^{n}]_{n,\wt},
    \end{align*}
    for all $\alpha^{n+1}\in \A^{n+1}$, where
    $\alpha^n=\mu(\alpha^{n+1})\in \A^n$.  
  \item 
    \label{item:propL7}
    For each leaf $L^n\in \LC^n_{\wt}$ %, $\LC^n_{\bt}$ 
    there is a leaf $L^{n+1}\in
    \LC^{n+1}_{\wt}$, %$\LC^{n+1}_{\bt}$, 
    such that $\mu(L^{n+1})=L^n$. More
    precisely let $L^{n+1}_1,\dots, L^{n+1}_N$ be the leaves that are
    mapped by $\mu$ to $L^n$. Let $L^n=L^n([\alpha^n]_{n,\wt})$ be an
    ideal $m$-gon  ($\#[\alpha^n]_{n,\wt}=m$) and each
    $L^{n+1}_j=L^{n+1}([\alpha^{n+1}_j]_{n,\wt})$ an ideal $m_j$-gon
    ($\#[\alpha^{n+1}_j]=m_j$). Each $m_j$ is a multiple of $m$ and 
    \begin{equation*}
      \sum_j m_j = d\, m.
    \end{equation*}
  \item
    \label{item:Ln5}
    Restricted to $\A^n$ it holds 
    \begin{align*}
      % \text{Restricted to } \A^n \text{ it holds }
      &\Sim{n+1,\wt}\; =\;\Sim{n,\wt}, \text{ equivalently}
      \\
      &[\alpha^n]_{n,\wt} =[\tilde{\alpha}^n]_{n,\wt} 
      \Leftrightarrow
      [\alpha^n]_{n+1,\wt} = [\tilde{\alpha}^n]_{n+1,\wt},
    \end{align*}
    for all $\alpha^n,\tilde{\alpha}^n\in \A^n$, $n\geq 0$.  
    In particular 
    \begin{align*}
      &[\alpha^n]_{n,\wt}\subset [\alpha^{n}]_{n+1,\wt}, 
      \quad \text{equivalently }\;
      L^n\subset L^{n+1},
      \intertext{where $L^n=L^n([\alpha^n]_{n,\wt})\in \LC^n_{\wt}$,
        $L^{n+1}=L^{n+1}([\alpha^{n}]_{n+1,\wt})\in \LC^{n+1}_{\wt}$. 
        This means that }
      &\Sim{1,\wt} \; \leq \; \Sim{2,\wt} \; \leq \dots \;.
    \end{align*}
    Note also that this means that distinct $0$-angles
    $\alpha^0,\tilde{\alpha}^0\in \A^0$ are not equivalent with
    respect to any $\Sim{n,\wt}$. 
    % 
    % save counter of enumeration
    \setcounter{mylistnum}{\value{enumi}}
  \end{enumerate}
\end{lemma}

We now shift our attention to the \emph{complements} of
the laminations.

\begin{definition}[Gaps]
  \label{def:gaps}
  The closure of one component of $S^2_{\wt}\setminus \bigcup \LC^n_{\wt}$
  ($S^2_{\wt}\setminus \bigcup \LC^n_{\wt}$) is
  called a white (black) \defn{gap} or $n$-gap. The set of all white 
  $n$-gaps is denoted by $\G^n_{\wt}$. 
\end{definition}

% Note that leaves of the
% lamination $\LC^n_{\wt}$ %,$\LC_{\bt}$ 
% correspond to $n$-vertices, and 
% $n$-arcs (i.e., the closures of components of $S^1\setminus \bigcup
% \A^n$) correspond to $n$-edges. Each white $n$-gap $G$ 
% corresponds to a white %(black) 
% $n$-tile. 

Let $E_0,\dots,E_{k-1}$ be the $0$-edges ordered cyclically on $\CC$,
meaning 
mathematically positively as boundary of the white $0$-tile
$X^0_{\wt}$. Each $n$-edge $E^n$ is said to be of
\defn{type} $j$ if $F^n(E^n)=E_j$. Similarly each $n$-arc $a^n\subset
S^1$ is of type $j$ if $\gamma^n(\alpha^n)$ is, i.e., if
$F^n(\gamma^n(a^n))=E_j$. 

In the same fashion let $p_0,\dots,p_{k-1}$ be the postcritical points
labeled cyclically on $\CC$. Each $n$-vertex $v$ is of \defn{type} $j$
if $F^n(v)=p_j$. Note that $v$ is also an $(n+m)$-vertex (for each
$m\geq 0$), and might be of different type as such. A leaf $L=L(v,b)\in
\LC^n_{\wt}$ is of \defn{type} $j$ if $v$ is. 

\begin{lemma}[Properties of gaps]
  \label{lem:prop_Ds}
  We have the following properties.
  \begin{enumerate}[\upshape{(}i\upshape{)}]
    \renewcommand{\theenumi}{$\G$ \arabic{enumi}}
  \item 
    \label{item:propL4}
    There is one white $n$-gap for each white $n$-tile,
    \begin{equation*}
      \#\G^n_{\wt}= d^n.      
    \end{equation*}
  \item 
    \label{item:propL5}
    Each gap $G\in \G^n_{\wt}$ 
    has \emph{$k$ $n$-arcs} $\subset S^1$
    in its boundary, \emph{one of each type}. Their types are
    \emph{cyclically ordered} as boundary of $G$. Equivalently
    $G$ intersects \emph{$k$ leaves} $L\in\LC^n_{\wt}$, \emph{one
      of each type}, 
    \emph{cyclically ordered} on $\partial G$.  
  \item
    \label{item:propLDtoX}
    Consider two $n$-arcs $a^n,b^n\subset S^1$. Then
    \begin{align*}
      &a^n,b^n \text{ are in the boundary of the \emph{same} gap $G\in
        \G^n_{\wt}$}
      \intertext{if and only if}
      &\gamma^n(a^n),\gamma^n(b^n) \text{ are contained in the \emph{same}
        white %(black) 
        $n$-tile.} 
    \end{align*}
  \item 
    \label{item:propGmapG}
    The $(n+1)$-gaps are mapped to $n$-gaps by $\mu$. That means
    for each gap $G^{n+1}\in \G^{n+1}_{\wt}$ there is a gap $G^n\in
    \G^n_{\wt}$ such that
    \begin{equation*}
      \mu(G^{n+1}\cap S^1)= G^n\cap S^1. 
    \end{equation*}
    Furthermore $\mu$ is injective on the interior of $G^{n+1}\cap
    S^1$.
  \item 
    \label{item:propGnGn+1}
    Every gap $G^{n+1}\in \G^{n+1}_{\wt}$ is contained in a (unique) gap
    $G^n\in \G^n_{\wt}$,
    \begin{equation*}
      G^n\supset G^{n+1}.
    \end{equation*}
  \item 
    \label{item:propDcap}
    There is a constant $n_0$ such that the following holds. Let
    $\alpha^n,\tilde{\alpha}^n\in \A^n$ be \emph{not equivalent} with
    respect to $\Sim{n,\wt}$. Then for $m\geq n+n_0$ 
    no gap $G^{m}\in \G^m_{\wt}$ contains points from both sets
    $[\alpha^n]_{m,\wt}, [\tilde{\alpha}^n]_{m,\wt}$.
  \end{enumerate}
\end{lemma}

Properties \eqref{item:propL4}, \eqref{item:propL5}, and
\eqref{item:propL5} show that each gap $G\in \G^n_\wt$ corresponds to
a white $n$-tile. 

% \begin{remark}
%   The previous lemma holds analogously for black tiles. Namely one
%   has to replace ``white'' by ``black'', the index ``$w$'' by ``$b$'',
%   and ``cyclically'' by ``anti-cyclically'' in the above. 
% \end{remark}

\begin{proof}[Proof of Lemma \ref{lem:prop_L} and Lemma \ref{lem:prop_Ds}]
%  ($\LC^n$ \ref{item:propL1}) follows immediately from 
%  \cite{peano} Lemma 6.2. %\ref{lem:prop_comp}. 

  \medskip
  ($\LC^n$ \ref{item:propL2}) 
  Consider an equivalence class $[\alpha]_{n,\wt}$, where $\alpha\in
  \A^n$. Let $b\in \pi^n_{\wt}(v)$ be the block corresponding to
  $[\alpha]_{n,\wt}$ ($\tilde{\alpha}\in [\alpha]_{n,\wt} \Leftrightarrow
  \tilde{\alpha}$ incident to $b$ at $v$). 

  Recall from Section \ref{sec:connection-graph} that the $n$-th white
  \emph{connection graph} $\Gamma^n$ is a \emph{tree}. 

  Consider an $n$-arc $[\beta,\beta']\subset \R/\Z=S^1$ (between
  two consecutive $n$-angles $\beta,\beta'\in \A^n$) in one component of
  $S^1\setminus [\alpha]_{n,\wt}$. Let $d\in \pi^n_{\wt}(w), d'\in
  \pi^n_{\wt}(w')$ be the blocks associated to $[\beta]_{n,\wt},
  [\beta']_{n,\wt}$. Then $c(w,d), c(w',d')$ (the vertices in the
  connection graph $\Gamma^n$ associated to $d,d'$) are in the \emph{same
    component} of $\Gamma^n\setminus c(v,b)$. Indeed
  $\gamma^n([\beta,\beta'])$ is contained in a white $n$-tile incident
  to both $d,d'$.  
 
  Assume there is an equivalence class $[\alpha']_{n,\wt}$ ($\alpha'\in
  \A^n$) containing $n$-angles in \emph{distinct} components of
  $S^1\setminus [\alpha]_{n,\wt}$ (i.e., $[\alpha]_{n,\wt},
  [\alpha']_{n,\wt}$ are \emph{crossing}). Let $b'\in \pi^n_{\wt}(v')$ be
  the associated block. Then in $\Gamma^n$ the vertex $c(v',b')$
  connects distinct components of $\Gamma^n\setminus c(v,b)$. Thus
  $\Gamma^n$ is not a tree, which is a contradiction.

%   Fix a leaf $L=L(v,b)\in \LC^n_{\wt}$ and two
%   consecutive vertices $\alpha^n_i,\alpha^n_j\in S^1$ of the ideal $m$-gon
%   $L$. Let $X$ be the $n$-tile as in Lemma
%   \ref{lem:lamination_succeeding}. 

%   Consider now the ``cluster connected at $X$'', i.e., the set
%   of 
%   white $n$-tiles $Y\ne X$ that $\gamma^n |_{(\alpha^n_i,\alpha^n_j)}$
%   intersects. Since $K^n_{\wt}$ (the cluster of white $n$-tiles)  is a tree
%   it follows that such a $Y$ is not 
%   connected to any white $n$-tile outside this
%   cluster, i.e., not connected to any white $n$-tile intersecting
%   $\gamma^n|_{(\alpha^n_j,\alpha^n_i)}$.  

%   This means that a leaf $L\in \LC^n_{\wt}$ with one vertex in
%   $(\alpha^n_i,\alpha^n_j)$ has all its vertices in this arc, proving
%   ($\LC^n$ \ref{item:propL2}). 

  \medskip
  (\ref{item:propL5}) 
  Fix a white $n$-tile $X$. Let
  $E_0,\dots,E_{k-1}$ be the $n$-edges in the boundary of
  $X$ (ordered mathematically positively in $\partial X$). Consider
  the $n$-arc 
  $a_{j-1}^n=[\alpha^n_{j-1},\alpha^n_j]\subset S^1$ that is mapped by 
  $\gamma^n$ to $E_{0}$. 
  Let $G\in \G^n_{\wt}$ be the gap %(i.e., closure of component of
  %$S^2_{\wt}\setminus \bigcup \LC^n_{\wt}$) 
  having 
  $a_{j-1}^n$ in its boundary. Consider the $n$-arc
  $a^n_i=[\alpha^n_i,\alpha^n_{i+1}]$ that is the cyclical successor
  to $a^n_{j-1}$ in $\partial G$. Note that $\alpha^n_i$ is the
  $n$-angle \emph{preceding} $\alpha^n_j$ with respect to
  $\Sim{n,\wt}$. From Lemma
  \ref{lem:lamination_succeeding} it follows that $\gamma^n(a^n_i)=
  E_1$. Continuing in this fashion yields the claim (note $\gamma^n$
  maps \emph{exactly one} $n$-arc to each $n$-edge).   
%   Consider the leaf $L\in \LC^n_{\wt}$
%   containing $\alpha^n_j$. Let $\alpha^n_i\in S^1$ be the
%   \emph{preceding} vertex of $L$. The hyperbolic geodesic from
%   $\alpha^n_i$ to $\alpha^n_j$ is a boundary arc of $G$ by ($\LC^n$
%   \ref{item:propL2}). 

%   The arc $a^n_i=[\alpha^n_i,\alpha^n_{i+1}]$ is the one that is
%   mapped by $\gamma^n$ to $E_{k-2}$ by Lemma
%   \ref{lem:lamination_succeeding}. Furthermore this is a boundary arc 
%   of $G$. Proceeding in this manner finishes the claim.  

  \medskip
  (\ref{item:propLDtoX})
  This follows directly from the previous argument. 

  \medskip
  (\ref{item:propL4}) From the above it is clear that there is
  a \emph{bijection} between white $n$-tiles and white $n$-gaps.
  Thus there are exactly $d^n$ such components.

  \medskip
  ($\LC^n$ \ref{item:propLnum_L}) The first equality follows from the
  fact that $\sum m_j=\#\A^n$, the number of $n$-angles, which in turn
  equals $kd^n$. 

  Recall from Section \ref{sec:connection-graph} the definition of the
  \emph{connection graph}. Given a white $n$-tile $X$ let $G=G(X)\in
  \G^n_{\wt}$ be the corresponding gap according to
  (\ref{item:propLDtoX}), meaning that $\gamma^n(G\cap S^1) =\partial
  X$. From Definition  
  \ref{def:simnw_simnb} it follows that the vertices $c(X)$, $c(v,b)$
  (of the $n$-th white connection graph) are connected by an edge
  if and only if the gap $G=G(X)$ has non-empty intersection with the leaf
  $L(v,b)\in \LC^n_{\wt}$.  

  By (\ref{item:propL5}) and \eqref{item:propL4} it thus follows that
  the ($n$-th white) 
  connection graph has $kd^n$ edges. On the other hand it has
  $\#\LC^n_{\wt} + d^n$ vertices and is a tree. Thus it has $\#\LC^n_{\wt} +
  d^n -1$ edges, which implies that $\#\LC^n_{\wt}=(k-1)d^n+1$. Hence
  \begin{equation*}
    \sum_j(m_j-1)=\sum_j m_j - \#\LC^n_{\wt}= kd^n - ((k-1)d^n +1)=d^n-1.
  \end{equation*}

  \medskip
  ($\LC^n$ \ref{item:propL6}) This follows from 
  \cite[Lemma 6.15]{peano}, %\ref{lem:prop_successor}, 
  Lemma \ref{lem:lamination_succeeding}, as
  well as the fact that $F$ maps succeeding $(n+1)$-edges to
  succeeding $n$-edges and $(n+1)$-tiles to $n$-tiles (of the same
  color). 

  \medskip
  (\ref{item:propGmapG}) Each $(n+1)$-arc is mapped by $\mu$ to an
  $n$-arc. The statement follows from (\ref{item:propL5}) and ($\LC^n$
  \ref{item:propL6}).  
  
  \medskip
  ($\LC^n$ \ref{item:propL7}) Consider a leaf $L^n=L^n(v,b)\in
  \LC^n_{\wt}$, and an $n$-edge $E^n\ni v$ incident to $b\in
  \pi^n_{\wt}(v)$ at $v$. Let 
  $E^{n+1}\ni v'$ be an 
  $(n+1)$-edge that is mapped by $F$ to $E^n$, where $F(v')=v$. 
  Then (by 
  \cite[Lemma 6.15]{peano}  % \ref{lem:prop_successor} 
  and Lemma \ref{lem:lamination_succeeding})
  the leaf $L_j^{n+1}=L^{n+1}(v',b')\in \LC^n_{\wt}$ is mapped by
  $\mu$ to $L^n$, where $E^{n+1}$ is incident to $b'\in
  \pi^{n+1}_{\wt}(v')$ at $v'$. If 
  $L^{n+1}_j$ is an ideal $m_j$-gon the number $d_j:=m_j/m$ is the
  number of preimages ($(n+1)$-edges) of $E^n$ incident to $b'$ at
  $v'$. Thus $\sum d_j=d$.   

  \medskip ($\LC^n$ \ref{item:LinQ}) Each $n$-angle $\alpha^n\in \A^n$
  is rational by construction (see \cite[Sections~4.1 and
  Section~4.2]{peano}). The second claim follows from ($\LC^n$
  \ref{item:propL6}) and the fact that each equivalence class
  $[\alpha^0]_{0,w}$ ($\alpha^n\in\A^0$) is a singleton (recall that
  $\gamma^0=\CC$ is a Jordan curve).

  \medskip
  The proof of (\ref{item:propGnGn+1}) will be postponed to Section
  \ref{sec:constr-lcn-induct}, 
  the proofs of ($\LC^n$ \ref{item:Ln5}) and (\ref{item:propDcap})
  to Section \ref{sec:lcn--d}. 
\end{proof}

\subsection{Inductive construction of $\Sim{n,\wt}$}
\label{sec:constr-lcn-induct}

We come to the main result in this section, namely that $\Sim{n,\wt}$
(or equivalently $\LC^n_{\wt}$) can be constructed inductively. This means
that $\Sim{1,\wt},\Sim{1,\bt}$ together allows to recover the map $F$ up
to topological conjugacy.

The inductive construction of $\LC^{n+1}_{\wt}$ may be paraphrased as
follows. 
The boundary $G\cap S^1$ of each gap $G\in \G^n_{\wt}$ is mapped
by $\mu^n$ onto $S^1$. Pulling the lamination 
$\LC^1_{\wt}$ back constructs $\LC^{n+1}_{\wt}$. 

Recall that $\mu^n(z)= d^nt (\bmod 1)$ denotes the $n$-th
iterate of $\mu\colon \R/\Z\to \R/\Z$.  
For any gap $G\in \G^n_{\wt}$, 
this map is surjective on $G\cap S^1$ by (\ref{item:propL5}), a
homeomorphism 
on each $n$-arc, and may fail to be injective only at
endpoints of $n$-arcs. 

Given $\Sim{n,\wt}$ we know all white $n$-gaps. To construct
$\Sim{n+1,\wt}$ it is enough to construct all $(n+1)$-angles that are
succeeding with respect to $\Sim{n+1,\wt}$ (Definition
\ref{def:succeeding}).  

\begin{theorem}[Inductive construction of $\Sim{n,\wt}$]
  \label{thm:LnLn+1}
  For all $\alpha^{n+1},\tilde{\alpha}^{n+1}\in \A^{n+1}$ it holds
  \begin{align*}
    &\alpha^{n+1},\tilde{\alpha}^{n+1}\text{ are succeeding
      with respect to $\Sim{n+1,\wt}$}
    \intertext{if and only if}
    &\alpha^{n+1},\tilde{\alpha}^{n+1} \text{ are contained in the
      same white $n$-gap $G^n$ and}   
    \\
    &\mu^n(\alpha^{n+1}),\mu^n(\tilde{\alpha}^{n+1})  
    \text{ are succeeding with respect to $\Sim{1,\wt}$}.
  \end{align*}
\end{theorem}

Let us emphasize that $(n+1)$-angles in disjoint $n$-gaps are never
succeeding, though 
the generated equivalence relation may identify angles from different
gaps. 

% Before proving the theorem we note the following additional property
% of $\LC^n,\Sim{n,\wt}$. 
% \begin{enumerate}[($\LC^n$ 1)]
%   % 
%   % restore counter of enumeration
%   \setcounter{enumi}{\value{mylistnum}}
% \item 
%   \label{item:Ln5}
%   Restricted to $\A^n$ it holds 
%   \begin{equation*}
% %    \text{Restricted to } \A^n \text{ it holds }
%     \Sim{n+1,\wt}\; =\;\Sim{n,\wt}.
%   \end{equation*}
%   In particular 
%   each leaf $L^n=L^n([\alpha^n]_{n,\wt})\in \LC^n_{\wt}$ is contained in
%   one 
%   leaf $L^{n+1}=L^{n+1}([\alpha^{n+1}]_{n+1,\wt})\in \LC^{n+1}_{\wt}$,
%   \begin{align*}
%     &L^n\subset L^{n+1}, \quad \text{equivalently }\;
%     [\alpha^n]_{n,\wt}\subset [\alpha^{n+1}]_{n+1,\wt}.
%     \intertext{This means that }
%     &\Sim{1,\wt} \; \leq \; \Sim{2,\wt} \; \leq \dots \;.
%   \end{align*}
%   % 
%   % save counter of enumeration
%   \setcounter{mylistnum}{\value{enumi}}
% \end{enumerate}

% \medskip
% Consider the lamination $\LC^1_{\wt}$. We replace each leaf $L\in \LC^n_{\wt}$
% by a \emph{star} as follows. Pick a \emph{center} in $L$ different
% from the vertices $\alpha^n_{j_1},\dots,\alpha^n_{j_m}$. Connect this
% center to each $\alpha^n_{j_l}$ by a hyperbolic geodesic.  

We first need some preparation to prove Theorem \ref{thm:LnLn+1}.
In the following lemma we consider $(n+1)$-arcs
\begin{align*}
  &a^{n+1},b^{n+1}\subset S^1,
  \intertext{the corresponding $(n+1)$-edges}
  &D^{n+1}=\gamma^{n+1}(a^{n+1}),E^{n+1}=\gamma^{n+1}(b^{n+1}),
  \intertext{and arcs
  $A^{n+1}, B^{n+1}\subset \bigcup \E^n$, satisfying}
  &H^n_1(A^{n+1})=D^{n+1}, H^n_1(B^{n+1})=E^{n+1}.
\end{align*}
Here $H^n$ is the pseudo-isotopy from which $\gamma^{n+1}$ was
constructed (see Section \ref{sec:pseudo-isotopies}). 
%Furthermore $G^n$ will always denote a white $n$-gap.

\begin{lemma}
  \label{lem:anan+1}
  In the setting as above
  \begin{enumerate}
  \item 
    \label{item:anan_prop1}
    \begin{align*}
      &\text{$a^{n+1},b^{n+1}$ are contained in the boundary of }
      \\
      &\text{the
        \emph{same} white $n$-gap }G^n 
      \intertext{if and only if}
      &A^{n+1},B^{n+1} \text{ are contained in the boundary of}
      \\
      &\text{the \emph{same} white $n$-tile $X^n$.}
    \end{align*}
  \item 
    \label{item:anan_prop2}
    \begin{align*}
      &D^{n+1},E^{n+1} \text{ are contained in the \emph{same} white
        $(n+1)$-tile,}
      \intertext{implies}
      &A^{n+1},B^{n+1} \text{ are contained in the \emph{same} white
         $n$-tile. }
    \end{align*}
  \end{enumerate} 
\end{lemma}

The situation is illustrated in Figure~\ref{fig:lemma5_11}. 

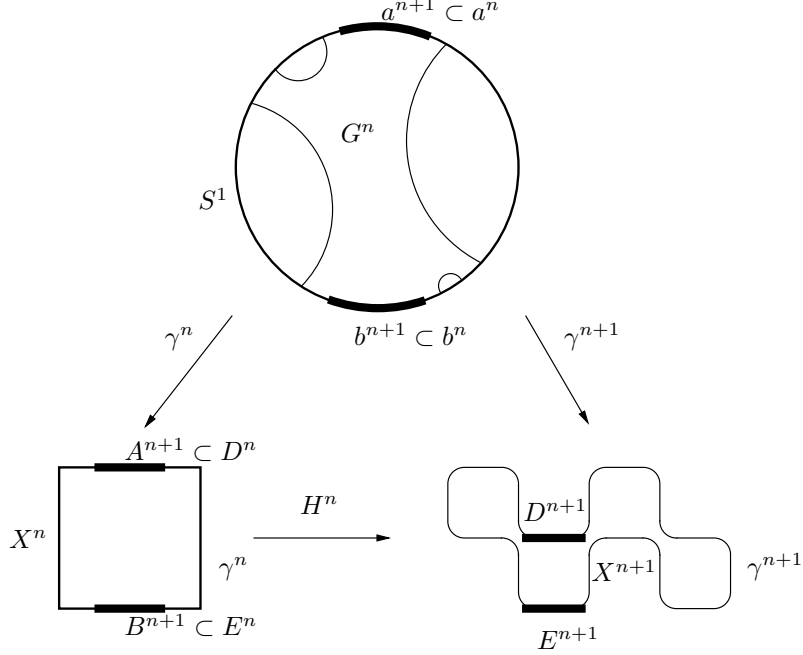
\begin{figure}
  \centering
  \begin{overpic}
    [width=9cm, tics=20,
    %grid
    ]  
    {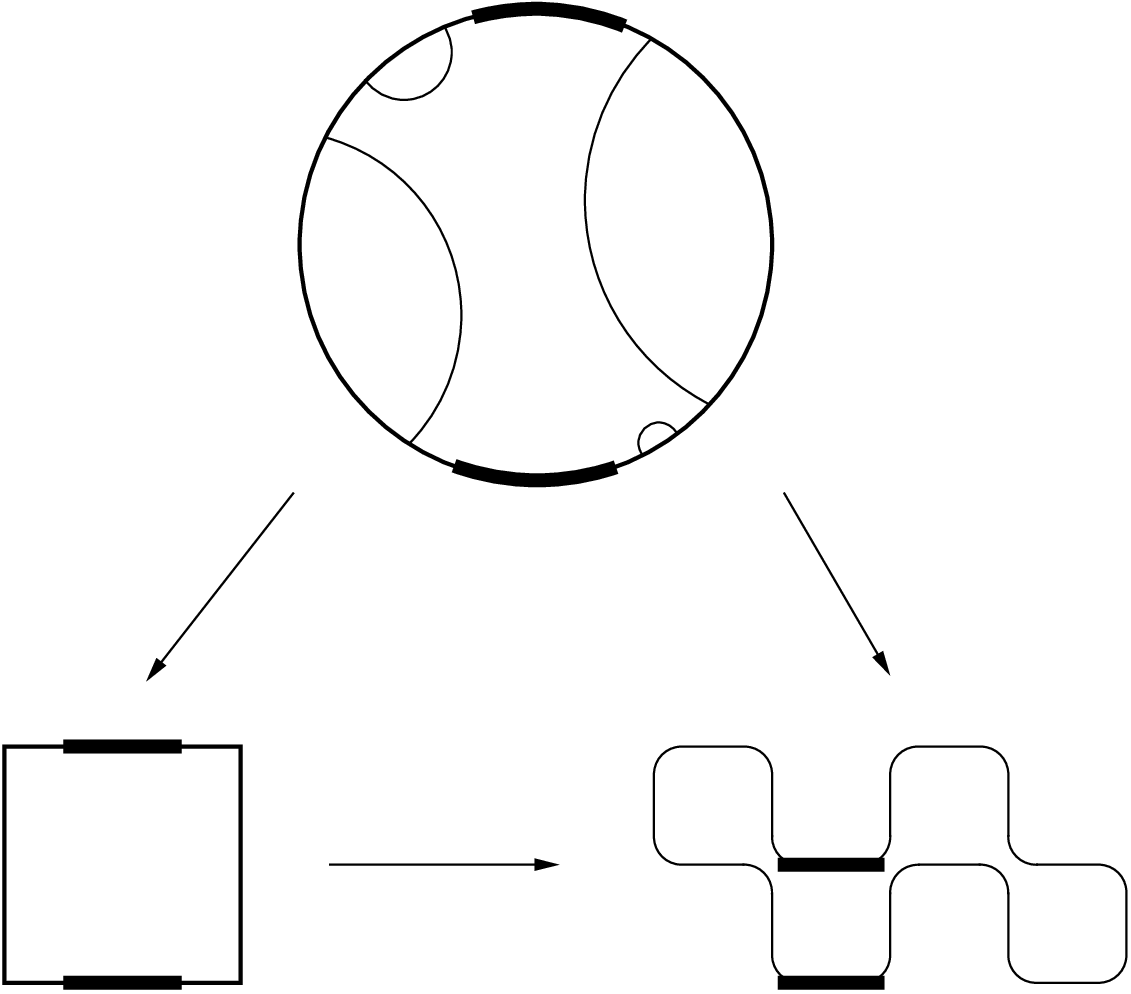}
    \put(21,60){$S^1$}
    \put(24,6){$\gamma^n$}
    \put(102,6){$\gamma^{n+1}$}
    \put(16,40){$\gamma^n$}
    \put(75,40){$\gamma^{n+1}$}
    \put(36,15){$H^n$}
    \put(-7,10){$X^n$}
    \put(48,88){$a^{n+1}\subset a^n$}
    \put(44,40){$b^{n+1}\subset b^n$}
    \put(42,70){$G^n$}
    \put(10,-3){$B^{n+1}\subset E^n$}
    \put(10,23){$A^{n+1}\subset D^n$}
    \put(71,-5){$E^{n+1}$}
    \put(79,5){$X^{n+1}$}
    \put(69,14){$D^{n+1}$}
  \end{overpic}
  \caption{Illustration for Lemma~\ref{lem:anan+1}.}
  \label{fig:lemma5_11}
\end{figure}

\begin{proof}
  (\ref{item:anan_prop1})
  Consider the $n$-arcs $a^n\supset a^{n+1}, b^n\supset b^{n+1}$, and
  the $n$-edges $D^n:=\gamma^n(a^n),E^n:=\gamma^n(b^n)$. 

  By construction of $\gamma^{n+1}$
  (see \cite[Definition 3.8]{peano}) %\ref{def:gamma_n}) 
  it holds $D^{n+1}\subset H^n_1(D^n)$,
  $E^{n+1}\subset 
  H^n_1(E^n)$. Thus $A^{n+1}\subset D^n,B^{n+1}\subset E^n$.  

  From (\ref{item:propLDtoX}) it follows that
  $a^{n+1},b^{n+1}$ are contained in the same white $n$-gap $G^n$ if
  and only if $D^n\supset A^{n+1},E^n\supset B^{n+1}$ are contained in
  the same white $n$-tile $X^n$.

  \medskip
  (\ref{item:anan_prop2})
  Consider the white $(n+1)$-tile $X^{n+1}\supset D^{n+1},E^{n+1}$. 
  Consider interiors, $U^{n+1}:=\inte X^{n+1}$ and $U^n:=(H^n_1)^{-1}
  (U^{n+1})\subset S^2\setminus\bigcup \E^n$. Clearly
  $A^{n+1},B^{n+1}\subset \partial U^n$. 

  The map $H^n_1$ is a homeomorphism on $U^n$, this follows from
  \cite[Lemma~3.4 ($H^n$ 3)]{peano}. 
  %($H^n$ \ref{item:H0_3}) 
  Thus $U^n$ is connected, hence in the interior of a single $n$-tile
  $X^n$. From \cite[Lemma~3.5 ($H^n$ 5)]{peano}   
  %($H^0$ \ref{item:H0_5}) 
  it follows that $X^n$ has the same color as $X^{n+1}$.
\end{proof}

We show that Property (\ref{item:propGnGn+1}) follows as a corollary.  
% Recall from 
% (\ref{item:propL5}) that for each gap $G^{n+1}\in \G^{n+1}_{\wt}$
% \begin{equation}
%   \label{eq:GnS1}
%   G^{n+1}\cap S^1 = [s^{n+1}_0, t^{n+1}_0] \cup \dots \cup
%   [s^{n+1}_{k-1}, t^{n+1}_{k-1}],
% \end{equation}
% where each $[s^{n+1}_j, t^{n+1}_j]\subset S^1$ is a $(n+1)$-arc. We
% say $t^{n+1}_j, s^{n+1}_{j+1}\in \A^{n+1}$ are \defn{succeeding with
%   respect to $G^{n+1}$}. Clearly
% \begin{align}
%   \label{eq:succeeding2}
%   &\alpha^{n+1},\tilde{\alpha}^{n+1} \text{ are succeeding with respect
%     to $\Sim{n+1}$}
%   \intertext{if and only if}
%   \notag
%   &\alpha^{n+1},\tilde{\alpha}^{n+1} \text{ are succeeding with respect
%   to some gap $G^{n+1}\in\G^{n+1}_{\wt}$.}
% \end{align}

\begin{proof}
  [Proof of \upshape{(\ref{item:propGnGn+1})}]
  We want to prove that every white $(n+1)$-gap $G^{n+1}$ is contained 
  in a white $n$-gap $G^n$. 

  Consider two $(n+1)$-arcs $a^{n+1},b^{n+1}\subset G^{n+1}\cap
  S^1$. From (\ref{item:propLDtoX}) it follows that the $(n+1)$-edges
  $D^{n+1}:= \gamma^{n+1}(a^{n+1}), E^{n+1}:= \gamma^{n+1}(b^{n+1})$
  are contained in the same white $(n+1)$-tile $X^{n+1}$. From Lemma
  \ref{lem:anan+1} (\ref{item:anan_prop2}) it follows that the arcs
  $A^{n+1},B^{n+1}\subset \bigcup \E^n$ satisfying $H^n_1(A^{n+1})=
  D^{n+1}, H^n_1(B^{n+1})=E^{n+1}$ are contained in the same white
  $n$-tile $X^n$. Thus, by Lemma \ref{lem:anan+1}
  (\ref{item:anan_prop1}), $a^{n+1},b^{n+1}$ are contained in the same
  gap $G\in \G^n_{\wt}$.
%   Consider the arcs $A^{n+1}, B^{n+1}\subset \bigcup \E^{n+1}$, such
%   that $H^{n+1}_1(A^{n+1})= D^{n+1}$, $H^{n+1}_1(B^{n+1})= E^{n+1}$.
%   Using the notation from (\ref{eq:GnS1}) consider the $(n+1)$-arcs
%   $a^{n+1}:= 
%   [s^{n+1}_j,t^{n+1}_j]$,
%   $b^{n+1}:=[s^{n+1}_{j+1},t^{n+1}_{j+1}]\subset G^{n+1}\cap S^1$. 

%   We keep the notation from Lemma \ref{lem:anan+1}, namely
%   $D^{n+1}:=\gamma^{n+1}(a^{n+1})=H^{n}_1(A^{n+1})$,
%   $E^{n+1}:=\gamma^{n+1}(b^{n+1})=H^n_1(B^{n+1})$.

%   \smallskip
%   The $(n+1)$-edges $D^{n+1},E^{n+1}$ are contained in the same white
%   $(n+1)$-tile $X^{n+1}$ by Lemma
%   \ref{lem:lamination_succeeding}. Therefore  
\end{proof}

Consider two $n$-arcs $a^n=[\alpha^n,\tilde{\alpha}^n]$,
$b^n=[\beta^n,\tilde{\beta}^n]$ that are cyclically consecutive in
$\partial G^n$, for a  
$G^n\in \G^n_{\wt}$. Then we call $\beta,\tilde{\alpha}^n$
\defn{succeeding in $G^n$}. Note that two $n$-angles are succeeding
with respect to $\Sim{n,\wt}$ if and only if they are succeeding with
respect to some white $n$-gap.

\begin{proof}
  [Proof of Theorem \ref{thm:LnLn+1}] ($\Rightarrow$) Let
  $\alpha^{n+1},\tilde{\alpha}^{n+1}\in \A^{n+1}$ be succeeding with
  respect to $\Sim{n+1,\wt}$. Then there is a white $(n+1)$-gap
  $G^{n+1}$ containing both. From (\ref{item:propGnGn+1}) it follows that
  $\alpha^{n+1},\tilde{\alpha}^{n+1}$ are contained in the same white
  $n$-gap $G^n$.

  \smallskip
  By (\ref{item:propGmapG}) and (\ref{item:propL5}) it follows that
  the $1$-angles $\mu^n(\alpha^{n+1}), \mu(\tilde{\alpha}^{n+1})$ are
  succeeding with respect to $\Sim{1,\wt}$. 

  \smallskip
  ($\Leftarrow$)
  Let $\alpha^{n+1}, \tilde{\alpha}^{n+1}\in \A^{n+1}$ be contained in
  the same gap $G^n\in \G^n_{\wt}$, such that
  $\mu(\alpha^{n+1}),\mu(\tilde{\alpha}^{n+1})$ are succeeding
  with respect to $\Sim{1,\wt}$. Thus they are succeeding
  with respect to some $1$-gap $G^1$. 

  From (\ref{item:propGmapG}) and
  (\ref{item:propGnGn+1}) it follows that there is an $(n+1)$-gap
  $G^{n+1}\ni \alpha^{n+1}, \tilde{\alpha}^{n+1}$ such that
  $\mu^n(G^{n+1}\cap S^1)= G^1\cap S^1$. Combined with
  (\ref{item:propL5}) it follows that $\alpha^{n+1},
  \tilde{\alpha}^{n+1}$ are succeeding with respect to $G^{n+1}$, thus
  they are succeeding with respect to $\Sim{n+1,\wt}$. 
%   each $(n+1)$-gap
%   $G^{n+1}\subset G^n$ is mapped to a $1$-gap $G^1$, and
%   $\alpha^{n+1},\tilde{\alpha}^{n+1}$ are succeeding with respect to
%   $G^{n+1}$ if and only if they are succeeding with respect to
%   $G^1$. The claim follows from (\ref{eq:succeeding2}).  
\end{proof}

\subsection{($\LC^n$ \ref{item:Ln5}) and (\ref{item:propDcap})}
\label{sec:lcn--d}

Using Theorem \ref{thm:LnLn+1} we can finish the proofs of Lemma
\ref{lem:prop_L} and Lemma \ref{lem:prop_Ds}. 

\begin{proof}[Proof of \upshape{($\LC^n$ \ref{item:Ln5})}]
  We need to show that
  \begin{equation*}
    \alpha^n\Sim{n,\wt} \tilde{\alpha}^n 
    \Leftrightarrow
    \alpha^n\Sim{n+1,\wt} \tilde{\alpha}^n,
  \end{equation*}
  for all $\alpha^n,\tilde{\alpha}^n\in \A^n$ (recall that
  $\A^{n}\subset\ \A^{n+1}$). 
  
  \smallskip
  We first show the statement for $n=0$, which is the following.
  \begin{claim}
    Let $\alpha^0,\tilde{\alpha}^0\in \A^0$ be distinct. Then
    $\alpha^0,\tilde{\alpha}^0$ are not equivalent with respect to
    $\Sim{1,\wt}$. 
  \end{claim}
  To prove this claim
  consider distinct angles $\alpha^0,\tilde{\alpha}^0\in \A^0$. 
  By definition of $\A^0$,
  $\gamma^0(\alpha^0),\gamma^0(\tilde{\alpha}^0)$ are postcritical
  points; in fact \emph{different postcritical points}, since
  $\gamma^0$ is a Jordan curve. By construction
  \begin{equation*}
    \gamma^1(\alpha^0)=\gamma^0(\alpha^0)\neq
    \gamma^0(\tilde{\alpha}^0)=\gamma^1(\tilde{\alpha}^0). 
  \end{equation*}
  Thus $\alpha^0,\tilde{\alpha}^0$ are \emph{not} equivalent with
  respect to $\Sim{1,\wt}$, proving the claim.

  \smallskip
  After this preparation we are ready to prove the above equivalence
  in general.
  
  \smallskip ($\Rightarrow$) Let $\alpha^n,\tilde{\alpha}^n\in \A^n$
  be succeeding with respect to $\Sim{n,\wt}$. Then there is a gap
  $G^n\in \G^n_{\wt}$ containing both. Furthermore they are contained
  in the same $n$-leaf. Thus they are mapped by $\mu^n$ to the same
  point $\alpha^0=\mu^n(\alpha^n)=\mu^n(\tilde{\alpha}^n)\in \A^0$
  according to ($\LC^n$~\ref{item:propL6}). Consider the elements of
  $[\alpha^0]_{1,w}$,
  \begin{equation*}
    \alpha^0=\alpha^1_0, \alpha^1_1,\dots, \alpha^1_N= \alpha^0.
  \end{equation*}
  Here $\alpha^1_j,\alpha^1_{j+1}$ are succeeding with respect to
  $\Sim{1,\wt}$. Note that by the claim above each point $\alpha^1_j$,
  $1\leq j \leq N-1$ is \emph{not} a $0$-angle, thus in the
  \emph{interior} of some $0$-arc. Hence there is exactly \emph{one}
  $(n+1)$-angle $\alpha^{n+1}_j\in G^n$, such that
  $\mu^n(\alpha^{n+1}_j)= \alpha^1_j$ (for $1\leq j\leq N-1$)by
  \eqref{item:propGmapG}. It 
  follows from Theorem \ref{thm:LnLn+1} that in the list
  \begin{equation*}
    \alpha^n=: \alpha^{n+1}_0, \alpha^{n+1}_1, \dots,
    \alpha^{n+1}_{N-1}, \alpha^{n+1}_N:= \tilde{\alpha}^n
  \end{equation*}
  the $(n+1)$-angles $\alpha^{n+1}_j, \alpha^{n+1}_{j+1}$ are
  succeeding with respect to $\Sim{n+1, w}$. Thus
  $\alpha^n,\tilde{\alpha}^n$ are equivalent with respect to
  $\Sim{n+1,\wt}$.

  \smallskip ($\Leftarrow$) Assume for $\alpha^n, \tilde{\alpha}^n$ it
  holds $\alpha^n\Sim{n+1,\wt} \tilde{\alpha}^n$. We want to show that
  $\alpha^n\Sim{n,\wt}\tilde{\alpha}^n$. Consider the sequence of
  succeeding $(n+1)$-angles
  \begin{equation*}
    \alpha^n=: \alpha^{n+1}_0, \alpha^{n+1}_1, \dots, \alpha^{n+1}_N:=
    \tilde{\alpha}^n. 
  \end{equation*}
  Let $\alpha^{n+1}_m, m\geq 1$ be the first element (after
  $\alpha^n$) in this sequence
  that is an $n$-angle. Since each angle $\alpha^{n+1}\in
  \A^{n+1}\setminus \A^n$ is contained in a single gap $G^n\in \G^n_{\wt}$
  it follows that $\alpha^{n+1}_0,\dots, \alpha^{n+1}_m$ is contained
  in a single gap $G^n\in \G^n_{\wt}$. Note that
  \begin{equation*}
    \alpha^1_0:= \mu^n(\alpha^{n+1}_0), \dots, \alpha^1_m :=
    \mu^n(\alpha^{n+1}_m) 
  \end{equation*}
  is a sequence of succeeding $1$-angles, i.e., $\alpha^n\Sim{1,\wt}
  \alpha^1_m$. Furthermore $\alpha^1_0,\alpha^1_k\in \A^0$, since
  $\alpha^n, \alpha^{n+1}_m$ were $n$-angles. 
 
  From the claim above it follows that $\alpha^1_0= \alpha^1_k \in
  \A^0$, this means that the $n$-angles $\alpha^n, \alpha^{n+1}_k$ are
  contained in $n$-leaves of the same type. Since they are contained
  in the same $n$-gap $G^n$ this means that they are contained in the
  same $n$-leaf (i.e., boundary component of $G^n$) by
  \eqref{item:propL5}, i.e., $\alpha^n\Sim{n,\wt} \alpha^{n+1}_m$.
  Continuing in this fashion, i.e., considering sequences in
  $(\alpha^{n+1}_j)$ between consecutive $n$-angles we conclude that
  $\alpha^n\Sim{n,\wt}\tilde{\alpha}^n$ as desired.

\end{proof}

\begin{proof}[Proof of \eqref{item:propDcap}]
  We prove the claim first for $n=0$. 
  Consider first distinct $\alpha^0,\tilde{\alpha}^0\in \A^0$. They are
  mapped by $\gamma^0$, hence by $\gamma$ to distinct postcritical
  points, $p:= \gamma^0(\alpha^0),
  \tilde{p}:=\gamma^0(\tilde{\alpha}^0)$. Since $F$ is expanding there
  is $n_0$, such that for $m\geq n_0$ 
  no $m$-tile contains both
  $p,\tilde{p}$. Thus $\alpha^0,\tilde{\alpha}^0$ are not contained in
  the same white $m$-gap
  by (\ref{item:propLDtoX}). %This proves the claim for $n=0$. 
  
  \smallskip
  Consider now non-equivalent (with respect to $\Sim{n,\wt}$)
  $\alpha^n,\tilde{\alpha}^n\in \A^n$. 
  Assume $\alpha^n\tilde{\alpha}^n\in G^{m+n}\in \G^{m+n}_{\wt}$ for some
  $m\geq n_0$. 
  Then the distinct $0$-angles
  $\alpha^0:= \mu^n(\alpha^n), \tilde{\alpha}^0:=
  \mu^n(\tilde{\alpha}^n)$ are contained in the same $m$-gap
  $G^m$, satisfying $\mu^n(G^{m+n}\cap S^1)= G^m\cap S^1$ by
  (\ref{item:propGmapG}). This contradicts the case $n=0$ above.   
  
%   Assume 
%   they are contained in the boundary of the same component $D^{n+n_0}$
%   of $S^2_{\wt}\setminus \bigcup \LC^{n+n_0}_{\wt}$. Then by
%   (\ref{item:propLDtoX}) 
%   $v^n:=\gamma^n(\alpha^n)=\gamma^{n+n_0}(\alpha^n), \tilde{v}^n:= 
%   \gamma^n(\tilde{\alpha}^n)=\gamma^{n+n_0}(\tilde{\alpha}^n)$ are
%   distinct $n$-vertices contained in the same (white) $(n+n_0)$-tile
%   $X^{n+n_0}$. Thus the distinct postcritical points $p:=F^n(v^n),
%   \tilde{p}:=F^n(\tilde{v}^n)$ (recall that $F^n$ is a homeomorphism
%   on $X^{n+n_0}$) are contained in the same $n_0$-tile $X^n_0:=
%   F^{n_0}(X^{n+n_0})$. This contradicts the above.     
\end{proof}

\subsection{The critical portrait}
\label{sec:critical-portrait}

According to Theorem \ref{thm:LnLn+1} we can recover all equivalence
relations $\Sim{n,\wt}$ from $\Sim{1,\wt}$
(and all $\Sim{n,\bt}$ from $\Sim{1,\bt}$). Note that the non-trivial
equivalence classes of 
$\Sim{1,\wt}, \Sim{1,\bt}$ (i.e., the ones containing at least two points)
are mapped by $\gamma$ to critical points of $F$. 

\begin{definition}
  \label{def:critical_portrait}
  The sets $[\alpha^1_j]_{1,\wt}\subset \Q/\Z\subset \R/\Z=S^1$,
  $j=1,\dots, m$ satisfy the following, meaning they form a 
  \defn{critical portrait}. 
  \begin{itemize}
  \item $\mu$ maps all points of $[\alpha^1_j]_{1,\wt}$ to a single
  point,
  \begin{equation*}
    \mu([\alpha^1_j]_{1,\wt})=\{\alpha^0_j\}. 
  \end{equation*}
\item 
  $\sum_j (\#[\alpha^1_j]_{1,\wt}-1)= d-1$.
\item The sets $[\alpha^1_1]_{1,\wt},\dots, [\alpha^1_m]_{1,\wt}$ are
  non-crossing. 
\end{itemize}
The orbit
\begin{equation*}
  \A^0:= \bigcup \{\mu^n(\alpha^1_j)\mid j=1,\dots, m, \; n\geq 1\}
\end{equation*}
is a finite set. 
\begin{itemize}
\item No set $[\alpha^1_j]_{1,\wt}$ contains more than one point from
  $\A^0$. 
\end{itemize}
The equivalence relations $\Sim{n,\wt}$, as well as the laminations
$\LC^n_{\wt}$ and the gaps $\G^n_{\wt}$, are defined inductively
as in Theorem \ref{thm:LnLn+1}. 
\begin{itemize}
\item There is a constant $n_0$ such that the following holds. Let
  $\alpha^0,\tilde{\alpha}^0\in \A^0$ be distinct. Then for $m\geq
  n_0$ no gap $G^m\in \G^m_{\wt}$ contains
  points from both sets $[\alpha^0]_{m,\wt}, [\tilde{\alpha}^0]_{m,\wt}$.
\end{itemize}
\end{definition}

The first property follows from ($\LC^n$~\ref{item:LinQ}), the second
from ($\LC^n$~\ref{item:propLnum_L}), the third from
($\LC^n$~\ref{item:propL2}), the fourth from ($\LC^n$~\ref{item:Ln5}),
and the last from \eqref{item:propDcap}.

\smallskip It can be shown given any critical portrait the properties
($\LC^n$ \ref{item:LinQ})--($\LC^n$ \ref{item:Ln5}) as well as
(\ref{item:propL4}), (\ref{item:propL5}),
(\ref{item:propGmapG})--(\ref{item:propDcap}) are then satisfied,
which we do not prove here.  

The equivalence classes of $\Sim{1,\bt}$ form a critical portrait as
well. This proves Theorem \ref{thm:FcriticalPortraits}.

\smallskip
A.\ Poirier \cite{Poirier} (extending work of Bielefeld-Fisher-Hubbard
\cite{MR1149891}) has shown that for every critical portrait there is
a polynomial ``realizing it'', see Section~\ref{sec:white-polynomial}.  
The last condition above is actually much stronger than the
corresponding condition for the general case in \cite{Poirier}. Indeed
we will show in Section \ref{sec:fatou-set-pw} that it implies that
the corresponding polynomial is of a special 
form. Namely closures of distinct bounded Fatou components are
disjoint and do not contain any critical point.

\smallskip
We remark that the sets $[\alpha^n_j]_{n,\wt}$ as well as
$[\alpha^n]_{n,\bt}$ form critical portraits as well. They will be
realized by the $n$-th iterate of the polynomials $P_\wt, P_\bt$
realizing the critical portraits $[\alpha^1_j]_{1,\wt},
[\alpha^1_j]_{1,\bt}$. 

\subsection{The equivalence relations $\Sim{\wt},\Sim{\bt}$}
\label{sec:equiv-relat-simw}

We consider the \emph{closure of the join} of the equivalence
relations $\Sim{n,\wt},\Sim{n,\bt}$. We will show in Section
\ref{sec:f-mating} that there are polynomials $P_{\wt},P_{\bt}$, such that
$\Sim{\wt},\Sim{\bt}$ are the equivalence relations induced by their
Carath\'{e}odory semi-conjugacies. These will be the polynomials into
which $F$ unmates. 

\begin{definition}[$\Sim{\wt}, \Sim{\bt}$]
  \label{def:simw_simb}
  The equivalence relation $\Sim{\wt}$ on $S^1$ is defined as follows,
  \begin{align*}
    &\Sim{\infty,\wt}:= \textstyle{\bigvee} \Sim{n,\wt}, \text{ meaning
      $s\Sim{\infty,\wt} t$ if and only if $s\Sim{n,\wt} t$ for some $n$};
    \\
    &\Sim{\wt} \text{ is defined to be the closure of } \Sim{\infty,\wt}.
  \end{align*}
  The equivalence relation $\Sim{\bt}$ is similarly defined as the
  closure of $\Sim{\infty,\bt}:= \bigvee \Sim{n,\bt}$. 
\end{definition}

The equivalence relation $\sim$ induced by the invariant Peano curve
$\gamma$ \eqref{eq:eq_rel} may be recovered from $\Sim{\wt},
\Sim{\bt}$ (hence $F$ from $\Sim{1,\wt}, \Sim{1,\bt}$ up to
topological conjugacy by Theorem~\ref{thm:S1simS2} and
Theorem~\ref{thm:LnLn+1}).

\begin{lemma}
  \label{lem:Fnocrit_simwsimb}
  Let $\Sim{\wt}, \Sim{\bt}$ be defined as above, $\sim$ the
  equivalence relation induced by $\gamma$. Then 
  \begin{enumerate}[\upshape(i)]
  \item 
    \label{item:sim_cl_simwsimb}
    $\sim$ is the \emph{closure} of $\Sim{\wt}\vee \Sim{\bt}$;
  \item 
    \label{item:sim_simwsimb}
    if $F$ has no periodic critical points, then
  \begin{equation*}
    \Sim{\wt} \vee \Sim{\bt} \;= \;\sim.
  \end{equation*}
  \end{enumerate}
\end{lemma}

\begin{proof}
  \eqref{item:sim_cl_simwsimb}
  Denote by $\widehat{\Sim{\wt} \vee \Sim{\bt}}$ the closure of
  $\Sim{\wt} \vee \Sim{\bt}$. The equivalence relations $\Sim{n}$ and
  $\Sim{\infty}$ were defined in \eqref{eq:def_simn} and
  \eqref{eq:def_siminfty}. 

  Since $\Sim{\wt} \,\geq \,\Sim{n,\wt}$,
  $\Sim{\bt}\,\geq\, \Sim{n,\bt}$ for all $n\in \N$ it follows that
  \begin{equation*}
    \Sim{\wt} \vee \Sim{\bt} 
    \;\;\geq \;\;
    \Sim{n,\wt} \vee \Sim{n,\bt}
    \;\;=\;\;
    \Sim{n},
  \end{equation*}
  for all $n\in \N$ by \eqref{eq:simn-simnw-simnb}. Hence
  $\Sim{\wt}\vee \Sim{\bt} \;\;\geq \;\; \Sim{\infty}$ and the closure
  of $\Sim{\wt} \vee \Sim{\bt}$ is bigger than the closure of
  $\Sim{\infty}$, which is $\sim$ by
  Theorem~\ref{thm:sim_usc_clos_siminfty}, i.e., 
  \begin{equation*}
    \widehat{\Sim{\wt}\vee \Sim{\bt}} \;\;\geq\;\; \sim. 
  \end{equation*}
  On the other hand we note that $\Sim{n,\wt} \;\leq \;\Sim{n}$
  for all $n\in\N$, thus $\Sim{\infty,\wt} \;\leq\;
  \Sim{\infty}$. Taking closures yields $\Sim{\wt}\; \leq \;\sim$ (using
  Theorem~\ref{thm:sim_usc_clos_siminfty} again), similarly $\Sim{\bt}
  \;\leq \; \sim$. Thus $\Sim{\wt}\vee \Sim{\bt} \;\leq \; \sim$,
  taking closures yields 
  \begin{equation*}
    \widehat{\Sim{\wt}\vee \Sim{\bt}} \;\;\leq\;\; \sim,
  \end{equation*}
  i.e., the statement. 
\end{proof}

The proof of part \eqref{item:sim_simwsimb} of the previous lemma will
be given in the next section after some preparation.

\section{Equivalence classes are finite}
\label{sec:equiv-class-are}

Here we show the following.

\begin{theorem}
  \label{thm:sim_finite}
  Let $\sim$ be the equivalence relation induced by $\gamma$. 
  \begin{itemize}
  \item If $F$ has \emph{no critical periodic cycles}, there is a
    number $N< \infty$ such that
    \begin{equation*}
      \#[s] \leq N \text{ for all } s\in S^1.
    \end{equation*}
  \item If $F$ \emph{has critical periodic cycles}, there is (at
    least) one finite equivalence class $[s]$, where 
    $\gamma(s)\notin \post$. 
  \end{itemize}
\end{theorem}

\begin{proof}
  The second statement could be proved using \emph{Douady's lemma},
  which says that given any set $A\subset S^1$, such that $\mu(A)=A$
  and $\mu$ is injective on $A$, then $A$ is finite. A proof of a
  stronger statement may be found in \cite[Lemma~4.2]{MR1765095}.

  We give a slightly different proof, which will follow from the 
  setup in the proof of the first claim.

  \smallskip
  Recall that $\mu(x)= \mu_d(x)= d x\,(\bmod 1)$ maps any $(n+1)$-arc to a
  $n$-arc. It follows that for any $n$-arc $a^n\subset S^1$ it holds
  $\diam a^n\leq d^{-n}$. Assume that for a given point $x\in S^2$ there is
  an $N\in \N$ such that $\gamma^{-1}(x)\subset S^1$ is contained in
  the union of at most $N$ $n$-arcs for any $n\in \N$. Then it follows
  that $\gamma^{-1}(x)$ contains at most $N$ points. 
  
  \smallskip
  Let us fix a visual metric $\varrho$ for $F$ with expansion factor
  $\Lambda>1$. All metrical properties (such as $\diam$ and the
  supremums norm
  $\norm{\cdot}_\infty$) on $S^2$ will from now on be given in terms
  of $\varrho$. 
  % Consider points $x\in S^2$ and $n$-arcs $a^n\subset S^1$, such
  % that $\gamma(a^n)\ni x$. We show that the number of such $n$-arcs is
  % independent of $n$ for some points $x$.
  Recall from Section \ref{sec:approximations} that there is a
  constant $C>0$ such that 
  \begin{equation}
    \label{eq:distgng}
    \norm{\gamma-\gamma^n}_\infty\leq C \Lambda^{-n}. 
  \end{equation}
  Thus only $n$-arcs $a^n$ such that the $n$-edges $\gamma^n(a^n)$
  have ``small combinatorial distance'' from $x$ can satisfy
  $\gamma(a^n)\ni x$. 
  
  \smallskip 
  Let $x\in S^2$ be arbitrary, and $X^n$ be an $n$-tile
  containing $x$ (for any $n\in \N$). Consider the set of $n$-tiles
  that ``can be reached in $j$ steps from $X^n$'', i.e., 
  \begin{align*}
    &U^n_0=U^n_0(X^n):= X^n, 
    \\
    &U^n_j=U^n_j(X^n):=\bigcup\{Z^n\in \X^n\mid Z^n\cap
    U^n_{j-1}\ne \emptyset\},
  \end{align*}
  for all $j\geq 1$. We will show that points outside of $U^n_j$ are
  ``far away'' from $x$. More precisely we show the following. 
  \begin{claim}
    There are $\alpha>0$, $c>0$ (independent of $n$ and $X$) such that 
    \begin{equation*}
      \dist(X^n, S^2\setminus U^n_j)\geq c j^\alpha \Lambda^{-n},
    \end{equation*}
    for all $n$-tiles $X^n$ and $j\in \N$.
  \end{claim}

  To prove the claim let $y\in S^2\setminus U^n_j$. Recall that
  $\varrho(x,y)\asymp \Lambda^{-m}$, where $m=m(x,y)$ is the smallest
  number such that there are disjoint $m$-tiles containing $x$, $y$
  (see Section~\ref{sec:visual-metric}). Thus (any) $(m-1)$-tiles
  $X^{m-1}, Y^{m-1}$ containing $x,y$ are not disjoint. 

  From \cite[Lemma~5.29]{expThurMarkov} it follows that that there is
  a constant $M\in \N$, such that any two points in  $X^{m-1}\cup
  Y^{m-1}$ (in particular $x,y$) may be joined by a chain
  of $m$-tiles of length at most $M$. This means there are $m$-tiles
  $Y_1, \dots, Y_M$ with $x\in Y_1, y\in Y_M$ and $Y_l \cap Y_{l+1}
  \neq \emptyset$ for all $1\leq l\leq M-1$. Inductively it follows
  that $x,y$ may be connected by a chain of $n$-tiles of length at
  most $M^{n-m+1}$. Since $y\notin U^n_j$ it follows that
  \begin{equation*}
    j \leq M^{n-m+1} \;\text{ or }\; -m \geq \frac{\log j}{\log M} -n -1.
  \end{equation*}
  Thus 
  \begin{align*}
    \varrho(x,y) \asymp \Lambda^{-m} \geq \tilde{c}
    j^{\alpha}\Lambda^{-n}, 
  \end{align*}
  for constants $\tilde{c}= \Lambda^{-1}> 0$, $\alpha= \log \Lambda/
  \log M> 0$ independent of $j,n$. The claim follows.

  \smallskip From the claim together with \eqref{eq:distgng} it
  follows that there is a constant $j_0\in \N$ such that $x=
  \gamma(s)\in X^n$ and $y= \gamma^n(t) \notin U^n_{j_0}$ implies that
  $\gamma(t)\neq x$.  Consider now the $n$-edges in $U^n_{j_0}$ and
  the $n$-arcs in $S^1$ that are mapped by $\gamma^n$ into
  $U_{j_0}^n$,
  \begin{align*}
    \label{eq:def_Enk0}
    &E^n=E^n_{j_0}(X^n) := U^n_{j_0}\cap \bigcup \E^n
    \\
    \notag
    &A^n=A^n_{j_0}(X^n):=(\gamma^n)^{-1}(E^n). 
  \end{align*}
  From the above it follows that for all $t\in S^1$ it holds
  \begin{equation*}
    \label{eq:gammaAnky}
    s\notin A^n \Rightarrow \gamma(t)\ne x,
  \end{equation*}
  for all $n\in \N$. 

  \smallskip
  Assume now that $F$ has no critical periodic cycles. Then the local
  degree is uniformly bounded, i.e., 
  \begin{equation*}
    \deg_{F^n}(v) \leq 2 d,
  \end{equation*}
  for all $v\in S^2$ and $n\in \N$ (recall that $d=\deg F$). Recall
  that the number of $n$-tiles incident to an $n$-vertex $n$ is given
  by $2\deg_{F^n}(v)$. It follows that the number of $n$-tiles
  contained in $U^n_{j_0}(X^n)$ is uniformly bounded (i.e.,
  independent of $n\in \N$ and the $n$-tile $X^n$). Thus the number of
  $n$-edges in $E^n_{j_0}(X^n)$, which equals the number of $n$-arcs
  in $A^n_{j_0}(X^n)$, is uniformly bounded by a number $N\in
  \N$. Thus $\gamma^{-1}(x)$ contains at most $N$ points, i.e., each
  equivalence class of $\sim$ contains at most $N$ points. The proof
  of the first statement of the lemma is finished. 

  \smallskip Assume now that $F$ has critical periodic cycles. Let
  $X^n$ be a white $n$-tile such that $U^n=U^n_{j_0}(X^n)$ is
  contained in the interior of the white $0$-tile $X^0_\wt$ (for some
  sufficiently large $n\in \N$). The existence of such an $n$-tile
  follows easily from the expansion property of $F$, see
  \cite[Lemma~7.9]{expThurMarkov}. Note that $F^n\colon X^n\to
  X^0_\wt$ is a homeomorphism. Let $U^{2n}:=
  (F^n|X^n)^{-1}(U^n)$. Note that $U^{2n}$ is compactly contained in
  $U^n$. Define $U^{(i+1)n}:= (F^n|X^n)^{-1}(U^{in})$ for all $i\geq
  1$. Then $U^{(i+1)n}$ is compactly contained in $U^{in}$. 
  Let the point $x_0\in S^2$ be defined by
  \begin{equation*}
    \{x_0\}= \bigcap_j U^{jn}_{k_0}.
  \end{equation*}
  Note that the number of $in$-tiles contained in $U^{in}$ is the same
  for all $i\in \N$, since $F^n\colon U^{(i+1)n} \to U^{in}$ is a
  homeomorphism (for all $i\in \N$). Hence the number of $n$-edges in
  $E^{in}:= U^{in} \cap \bigcup \E^{in}$, which equals the number of $n$-arcs in
  $A^{in}:= \gamma^{-1}(E^{in})$, is the same for all $i\in \N$, i.e., uniformly
  bounded.  This means there is a constant $N\in \N$, such that the
  number of $n$-arcs $a^n$ with $x_0 \in \gamma(a^n)$ is bounded by
  $N$ for all $n\in \N$. Thus the number of points in
  $\gamma^{-1}(x_0)$ (which is an equivalence class of $\sim$) is at
  most $N$. This finishes the proof of the second
  statement of the lemma.

\end{proof}

\begin{proof}[Proof of
  Lemma~\ref{lem:Fnocrit_simwsimb}~\eqref{item:sim_simwsimb}]
  It follows from
  Lemma~\ref{lem:Fnocrit_simwsimb}~\eqref{item:sim_cl_simwsimb} that
  it is enough to show that $\Sim{\wt} \vee \Sim{\bt}$ is
  \emph{closed}. Consider $s,t\in S^1$ satisfying $s\Sim{\wt}\vee
  \Sim{\bt} t$. Then
  \begin{equation*}
    s=s_1\Sim{\wt}s_2\Sim{\bt}\dots \Sim{\wt}s_{M-1} \Sim{\bt} s_M=t,
  \end{equation*}
  for some points $s_1,\dots, s_M\in S^1$. Recall from the proof of
  Lemma~\ref{lem:Fnocrit_simwsimb}~\eqref{item:sim_cl_simwsimb} that
  $\Sim{\wt}\vee \Sim{\bt}\;\leq \;\sim$. Since the size of each
  equivalence class with respect to $\sim$ is at most a constant $N$
  by Theorem~\ref{thm:sim_finite}; it follows that we can always
  choose $M=2N$ independently of $s,t$.

  \smallskip Consider now convergent sequences $(s^n),(t^n)\subset
  S^1$, i.e., $s^n\to s, t^n\to t$ (as $n\to \infty$), such that
  $s^n\Sim{\wt} \vee\Sim{\bt} t^n$ for all $n\in \N$. By the above
  there are $s^n_j\in S^1$ ($j=1, \dots, M$) such that
  \begin{equation*}
    s^n=s^n_1\;\Sim{\wt}\; s^n_2\;\Sim{\bt} \;\dots\; \Sim{\bt}\;
    s^n_M=t^n,
  \end{equation*}
  for all $n\in \N$. 
  By taking subsequences we can assume that $s^n_j\to s_j$ as $n\to
  \infty$, for all $j$. Since $\Sim{\wt}, \Sim{\bt}$ are closed it
  follows that
  \begin{equation*}
    s = s_1\;\Sim{\wt}\; \dots \;\Sim{\bt}s_N=t, 
  \end{equation*}
  meaning $s\Sim{\wt}\vee \Sim{\bt} t$. This means that $\Sim{\wt}\vee
  \Sim{\bt}$ is closed as desired. 
\end{proof}

\section{$F$ is a mating}
\label{sec:f-mating}

In this section we prove Theorem \ref{thm:mating1}. This means that we
show that $F$ is obtained as a \defn{mating} in the case when $F$ has
no periodic critical points. The construction will however be done for
the general case, in preparation to prove Theorem \ref{thm:mating2}.

\smallskip Recall from Section \ref{sec:mating-polynomials} the
construction of mating of two polynomials $P_{\wt},P_{\bt}$ (which are
monic and of the same degree $d$). Assume for now that every critical
point of $P_{\wt},P_{\bt}$ is \emph{strictly preperiodic}. This means
that the Fatou sets of $P_{\wt},P_{\bt}$ consist both of a single
(unbounded) component, i.e., their filled Julia sets equal their Julia
sets $\K_{\wt}=\J_{\wt}$, $\K_{\bt}=\J_{\bt}$.  %are \emph{dendrites}.
Let $\sigma_{\wt} \colon S^1\to \J_{\wt}$, $\sigma_{\bt} \colon S^1\to
\J_{\bt}$ be the Carath\'{e}odory semi-conjugacies of
$\J_{\wt},\J_{\bt}$.  Consider the equivalence relations
$\Approx{\wt},\Approx{\bt}$ on $S^1$ induced by $\sigma_{\wt},\sigma_{\bt}$,
\begin{align}
  \label{eq:def_approx_w}
  & z\Approx{\wt} w :\Leftrightarrow \sigma_{\wt}(z)=\sigma_{\wt}(w) 
  \\
  \notag
  & z\Approx{\bt} w :\Leftrightarrow \sigma_{\bt}(\bar{z})=
  \sigma_{\bt}(\overline{w}), 
\end{align}
for all $z,w\in S^1=\partial \D$. 
From \eqref{eq:cara_loop} and Lemma~\ref{lem:usc_from_map} it follows  
that 
\begin{align*}
& z^d/\Approx{\wt} \colon S^1/\Approx{\wt}
\;\;\to \; S^1/\Approx{\wt}
\\ 
& z^d/\Approx{\bt} \colon S^1/\Approx{\bt}
\;\;\to \; S^1/\Approx{\bt}
\end{align*}
are topologically conjugate to $P_{\wt}\colon\J_{\wt}\to \J_{\wt}$ and
$P_{\bt}\colon\J_{\bt}\to \J_{\bt}$. Let
\begin{equation*}
  \approx \;\;:=\;\; \Approx{\wt} \vee \Approx{\bt}.
\end{equation*}

% *** same thing as above with $\mu$ instead of $z^d$ ***
% To keep the setup from the previous sections we
% identify $S^1$ with $\R/\Z$. Recall that $\mu=\mu_d\colon \R/\Z\to
% \R/\Z$ is $\mu(s) = ds \,(\bmod 1)$ (which is conjugate to
% $z^d\colon \partial \D \to \partial \D$).  Consider the equivalence
% relations $\Approx{\wt},\Approx{\bt}$ obtained from
% $\sigma_{\wt},\sigma_{\bt}$,
% \begin{align*}
%   & s\Approx{\wt} t :\Leftrightarrow \sigma_{\wt}(s)=\sigma_{\wt}(t) 
%   \\
%   & s\Approx{\bt} t :\Leftrightarrow \sigma_{\bt}(-s)= \sigma_{\bt}(-t),
% \end{align*}
% for all $s,t\in S^1=\R/\Z$. 
% From \eqref{eq:cara_loop} and Lemma~\ref{lem:usc_from_map} it follows  
% that 
% \begin{align*}
% & \mu/\Approx{\wt} \colon S^1/\Approx{\wt}
% \;\to \; S^1/\Approx{\wt}
% \\ 
% & \mu/\Approx{\bt} \colon S^1/\Approx{\bt}
% \;\to \; S^1/\Approx{\bt}
% \end{align*}
% are topologically conjugate to $P_{\wt}\colon\J_{\wt}\to \J_{\wt}$ and
% $P_{\bt}\colon\J_{\bt}\to \J_{\bt}$.   
From the construction of the mating of $P_{\wt},P_{\bt}$ we obtain the
following.
\begin{lemma}
  \label{lem:mating_dendrites}
  Let $P_{\wt},P_{\bt}$ be monic polynomials of the same degree $d$, where every
  critical point is strictly preperiodic. Let $\approx$ be defined as
  above. 
  Then the topological mating of $P_{\wt},P_{\bt}$ is
  topologically conjugate to the map
  \begin{equation*}
    z^d/ \!\approx \;\colon S^1/\!\approx \;\;\to\; S^1/\!\approx. 
  \end{equation*}
\end{lemma}
We will show that there are polynomials $P_{\wt},P_{\bt}$ as above
such that
\begin{equation*}
  \Approx{\wt}\;\; = \;\;\Sim{\wt}, \quad \Approx{\bt}\;\;= \;\;\Sim{\bt}. 
\end{equation*}
Here $\Sim{\wt},\Sim{\bt}$ are the equivalence relations from
Definition~\ref{def:simw_simb}. Lemma~\ref{lem:Fnocrit_simwsimb} then
implies that $\approx \;=\;\sim$ (i.e., the equivalence relation
induced by the invariant Peano curve $\gamma$ \eqref{eq:eq_rel}). This
will prove Theorem~\ref{thm:mating1} using
Theorem~\ref{thm:sim_usc_clos_siminfty}. 
%Of particular importance will
%be to show that $\approx$ is closed.
% remove last sentence? 

\begin{proof}[Proof of Lemma~\ref{lem:mating_dendrites}]
  The mating of $z^d/\!\Approx{\wt}$ and $z^d/\!\Approx{\bt}$ is given
  by considering the equivalence relation on the disjoint union of
  $S^1/\!\Approx{\wt}$ and $S^1/\!\Approx{\bt}$ generated by
  identifying $[z]_\wt\in S^1/\!\Approx{\wt}$ with $[\bar{z}]_\bt\in
  S^1/\!\Approx{\bt}$. The quotient is denoted by $S^1/\!\Approx{\wt, \bt}$.
  %$S^1/\!\Approx{\wt}\; \mate S^1/\!\Approx{\bt}$.  
  The maps $z^d/\!\Approx{\wt}$,
  $z^d/\!\Approx{\bt}$ descend to this quotient, i.e., to a map
  \begin{equation*}
    z^d/\!\Approx{\wt,\bt} \;\colon S^1/\!\Approx{\wt, \bt} 
    \;\;\to\;
    S^1/\!\Approx{\wt,\bt}. 
  \end{equation*}
  % \begin{equation*}
  %   z^d/\!\Approx{\wt} \;\mate z^d/\!\Approx{\bt}\colon S^1/\!\Approx{\wt}
  %   \; \mate S^1/\!\Approx{\bt} \to S^1/\!\Approx{\wt}
  %   \; \mate S^1/\!\Approx{\bt}. 
  % \end{equation*}
  Since $z^d/\!\Approx{\wt}$, $z^d/\!\Approx{\bt}$ are topologically
  conjugate to $P_\wt, P_\bt$ it follows (from the definition of the
  conjugacy) that $P_\wt \mate P_\bt$ is topologically conjugate to
  $z^d/\!\Approx{\wt,\bt}$. Note that the map $S^1\to S^1/\!\Approx{\wt}
  \to S^1/\!\Approx{\wt,\bt}$ (i.e., the composition of the quotient
  maps) is surjective. The equivalence relation induced by this map is
  $\Approx{\wt}\vee \Approx{\bt}$. Furthermore the following diagram
  commutes
  \begin{equation*}
    \xymatrix{
      S^1 \ar[r]^{z^d} \ar[d] & S^1\ar[d]
      \\
      S1/\!\Approx{\wt} \ar[r]^{z^d/\Approx{\wt}} \ar[d] &
      S^1/\!\Approx{\wt}\ar[d]
      \\
      S^1/\!\Approx{\wt, \bt} \ar[r]^{z^d/\Approx{\wt,\bt}}& S^1/\! \Approx{\wt,\bt}. 
      }
  \end{equation*}
  The statement follow from Lemma~\ref{lem:usc_from_map}. 
\end{proof}

\medskip We now discuss the general case, where $F$ is allowed to have
periodic critical points, i.e., we outline the proof of Theorem
\ref{thm:mating2}. The Carath\'{e}odory semi-conjugacies
$\sigma_{\wt},\sigma_{\bt}$ (for monic, postcritically finite
polynomials $P_{\wt},P_{\bt}$ both of degree $d\geq2$) are defined as
before.

Define equivalence relations on $S^1$ by
\begin{align}
  \label{eq:def_approxFw}
  z\Approx{\F,\wt} w \;:\Leftrightarrow\; &\sigma_{\wt}(z)=\sigma_{\wt}(w) \text{ or}
  \\
  \notag
  &\sigma_{\wt}(z),\sigma_{\wt}(w) \text{ are both contained in $A$}
  \\
  \notag
  &\text{where $A$ is the closure of a bounded Fatou component of
    $P_\wt$}. 
  \\
  \label{eq:def_approxFb}
  z\Approx{\F,\bt} w \;:\Leftrightarrow\;
  &\sigma_{\bt}(\bar{z})=\sigma_{\bt}(\overline{w}) \text{ or}
  \\
  \notag &\sigma_{\bt}(\bar{z}),\sigma_{\bt}(\overline{w}) \text{ are
    both contained in $A$}
  \\
  \notag &\text{where $A$ is the closure of a bounded Fatou component
    of $P_\bt$}.
\end{align}
for all $z,w\in S^1=\partial \D$. Note that in general the above does
not define equivalence relations. Namely the closures of distinct
bounded Fatou components $A, A'$ may not be disjoint. We will however
consider only polynomials $P_\wt, P_\bt$ of a special type, namely
where such sets $A,A'$ are always disjoint. 

We define $\widehat{\approx}$ to be the closure of
$\Approx{\F,\wt}\vee \Approx{\F,\bt}$. Similarly as in the last lemma
it will be shown that the quotient map
\begin{equation*}
  z^d/\widehat{\approx} \;\colon S^1/\widehat{\approx} 
  \;\to \;
  S^1/\widehat{\approx}
\end{equation*}
is topologically conjugate to
\begin{equation*}
  P_{\wt}\;\widehat{\mate} \;P_{\bt} \colon \K_{\wt}\;\widehat{\mate} \;\K_{\bt}  \to
  \K_{\wt}\;\widehat{\mate}\; \K_{\bt}, 
\end{equation*}
as defined in \eqref{eq:PwPbsimhat}. 
We will show that there are polynomials $P_{\wt},P_{\bt}$ such that
\begin{equation*}
  \Approx{\F,\wt}\;=\;\Sim{\wt}, \quad \Approx{\F,\bt}\;=\;\Sim{\bt}.
\end{equation*}
Then $\widehat{\approx}\;=\;\sim$ by
Lemma~\ref{lem:Fnocrit_simwsimb}~\eqref{item:sim_cl_simwsimb}. 
Thus Theorem \ref{thm:mating2} will be proved (using Theorem~\ref{thm:S1simS2}).

\subsection{Julia- and Fatou-type equivalence classes}
\label{sec:julia-fatou-type}

In the following $S^1$ is again identified with $\R/\Z$. Recall that
the map $\mu=\mu_d\colon \R/\Z\to \R/\Z$ is given by $\mu(s)= ds \,
(\bmod 1)$ (which is conjugate to $z^d\colon \partial \D \to \partial
\D$).

The \emph{non-trivial equivalence classes} of $\Sim{1,\wt}$, i.e., the
ones 
that contain at least two points, are called the (white)
\defn{critical equivalence classes}. They are mapped by $\gamma^1$
(and thus by all $\gamma^n$ and $\gamma$) to critical points of
$F$. We 
divide the critical equivalence classes into ones of \defn{Fatou-type}
and \defn{Julia-type} as follows.
\begin{itemize}
\item If $[\alpha^1]_{1,\wt}$ is \emph{periodic}, i.e., if
  \begin{equation*}
    \mu^n([\alpha^1]_{1,\wt})\subset [\alpha^1]_{1,\wt},
  \end{equation*}
  for some $n\geq 1$, it is of \emph{periodic Fatou-type};
\item if $[\alpha^1]_{1,\wt}$ is the \emph{preimage of a periodic
    critical cycle}, i.e., if
  \begin{equation*}
    \mu^n([\alpha^1]_{1,\wt})\subset [\tilde{\alpha}^1]_{1,\wt},
  \end{equation*}
  for some $n\geq 1$,
  where $[\tilde{\alpha}^1]_{1,\wt}$ is of periodic Fatou-type; then
  $[\alpha^1]_{1,\wt}$ is of \defn{preperiodic Fatou-type};
\item otherwise $[\alpha^1]_{1,\wt}$ is of \defn{Julia-type}, i.e., the
  periodic cycle that $[\alpha^1]_{1,\wt}$ eventually lands in does not
  contain any point of a critical equivalence class. 
\end{itemize}
Every Fatou-type equivalence class is mapped by $\gamma^1$ to a point
that is eventually mapped to a critical periodic cycle of $F$. However
if 
$[\alpha^1]_{1,\wt}$ is of Julia-type, and
$c=\gamma^1([\alpha^1]_{1,\wt})$, then the periodic cycle that $c$
eventually lands in may or may not be critical. This is due to the fact
that the periodic critical point of $F$ may ``come from'' the black
polynomial. 

\smallskip
Consider now an equivalence class $[\alpha^n]_{n,\wt}$ ($\alpha^n\in
\A^n$) with respect to $\Sim{n,\wt}$. It is defined to be of
\defn{Fatou-type/Julia-type} if $\mu^{n-1}([\alpha^n]_{n,\wt})$ is
of Fatou-type/Julia-type (recall ($\LC^n$ \ref{item:propL6})). 
We note the following (recall that $[\alpha^n]_{n,\wt}\subset
[\alpha^n]_{n+1,\wt}$ from ($\LC^n$ \ref{item:Ln5}))  
\begin{equation}
  \label{eq:eqtype}
  [\alpha^n]_{n,\wt}, [\alpha^{n}]_{n+1,\wt} \text{ are of the some
    type},
\end{equation}
for all $\alpha^n\in \A^n$.

% Let
% $\Sim{n,J}$ be the restriction of $\Sim{n,\wt}$ to Julia-type
% equivalence classes,
% \begin{equation}
%   \label{eq:defsimJn}
%   s\Sim{n,J} t \;:\Leftrightarrow\; s\Sim{n,\wt}t \;\text{ and }\; [s]_{n,\wt}
%   \text{ is of Julia-type}.
% \end{equation}
% One should properly write $\Sim{n,J,w}$ for this equivalence relation,
% we suppress the index ``$w$'' for purely typographic reasons.   

% From (\ref{eq:eqtype}) it follows that
% \begin{equation*}
%   \Sim{n,J}\;\leq \;\Sim{2,J}\;\leq \dots \;.
% \end{equation*}
% Their meet and closure is 
% \begin{align*}
%   &\Sim{\infty,J} := \textstyle{\bigvee} \Sim{n,J} 
%   \text{ and }
%   \\
%   &\Sim{J} \text{ is the closure of } \Sim{\infty,J}.
% \end{align*}
% The topological model for the white Julia set if $S^2/\Sim{J}$, where
% $\Sim{J}$ is extended to $S^2_{\wt}$ as in Section
% \ref{sec:lamin-lc_{\wt}-lc_{\bt}}. 

\subsection{Sizes of equivalence classes}
\label{sec:sizes-equiv-class}

The main result of this subsection is the following. 
\begin{prop}
  \label{prop:Fatou_crit_cycles}
  The expanding Thurston map $F$ has critical periodic cycles if and
  only if there are Fatou-type equivalence classes of $\Sim{1,\wt}$ or
  $\Sim{1,\bt}$. 
\end{prop}
We need some preparation. The \defn{degree} of a critical equivalence
class $[\alpha^1]_{1,\wt}$ is its size,
\begin{equation}
  \label{eq:defd_eq}
  d([\alpha^1]_{1,\wt}):= \#[\alpha^1]_{1,\wt}. 
\end{equation}
The degree of other equivalence classes will be the degree of the
critical class it contains. 
\begin{align*}
  d([\alpha^n]_{n,\wt}) 
  := 
  \begin{cases}
    \#[\alpha^1]_{1,\wt}, 
    &\text{if } [\alpha^1]_{1,\wt}\subset [\alpha^n]_{n,\wt};
    \\
    1, 
    &\text{if $[\alpha^n]_{n,\wt}$ contains no critical class}.
  \end{cases}
\end{align*}
Note that by ($\LC^n$ \ref{item:Ln5}) there can be at most one
critical class contained in $[\alpha^n]_{n,\wt}$, thus the above is well
defined. 

\smallskip
Consider now $[\alpha^n]_{n,\wt}$, where $\alpha^n\in \A^n$. Let (recall
($\LC^n$ \ref{item:propL6}))
\begin{equation*}
  [\alpha^{n-1}]_{n-1,\wt}:= \mu([\alpha^n]_{n,\wt}), \;
  [\alpha^{n-2}]_{n-2,\wt}:= \mu^2([\alpha^n]_{n,\wt}), \dots \;.
  %[\alpha^1]_{1,\wt}:= \mu^{n-1}([\alpha^n]_{n,\wt}). 
\end{equation*}

\begin{lemma}[Size of equivalence classes]
  \label{lem:sizes_an}
  In the setting as above it holds
  \begin{equation*}
    \#[\alpha^n]_{n,\wt} = d([\alpha^n]_{n,\wt})\cdot
    d([\alpha^{n-1}]_{n-1,\wt}) \cdot 
    \ldots \cdot
    d([\alpha^1]_{1,\wt}). 
  \end{equation*}
\end{lemma}

\begin{proof}
  The statement is clear for $n=1$. We proceed by induction. Thus we
  assume the statement is true for $n$. 

  \smallskip
  \begin{case}[1]
    $[\alpha^{n+1}]_{n+1,\wt}$  contains no angle $\alpha^n\in \A^n$.  
  \end{case}
  
  From ($\LC^n$ \ref{item:Ln5}) and ($\LC^n$ \ref{item:propL2}) it
  follows that the leaf $L([\alpha^{n+1}]_{n+1,\wt})$ is contained in
  the iterior of the complement of the $n$-th lamination, i.e., in a
  white $n$-gap $G^n$. This means 
  $[\alpha^{n+1}]_{n+1,\wt}$ is contained in the
  \emph{interior} of the $n$-arcs which form $G^n\cap S^1$
  (see (\ref{item:propL5})). It follows that $\mu^n$ is bijective on
  $[\alpha^{n+1}]_{n+1,\wt}$. Thus
  \begin{equation*}
    \#[\alpha^{n+1}]_{n+1,\wt}=\#\mu^n([\alpha^{n+1}]_{n+1,\wt})=
    \#[\alpha^1]_{1,\wt}= d([\alpha^1]_{1,\wt}). 
  \end{equation*}
  On the other hand $[\alpha^{n+1}]_{n+1,\wt}$ contains no $\alpha^n\in
  \A^n$, 
  thus no $\alpha^1\in \A^1$. Therefore $d([\alpha^{n+1}]_{n+1,\wt})=1$. 

  Similarly $[\alpha^n]_{n,\wt}$ contains no $\alpha^{n-1}\in \A^{n-1}$,
  hence $d([\alpha^n]_{n,\wt})=1$. Repeating the argument yields
  \begin{equation*}
    \#[\alpha^{n+1}]_{n+1,\wt}= \underbrace{d([\alpha^{n+1}]_{n+1,\wt})\cdot \ldots
    \cdot d([\alpha^2]_{2,\wt})}_{=1} \cdot d([\alpha^1]_{1,\wt}).
  \end{equation*}

  \begin{case}[2]
    $[\alpha^{n+1}]_{n+1,\wt}$ contains some $n$-angles, i.e.,
    $[\alpha^{n+1}]_{n+1,\wt}\supset [\tilde{\alpha}^n]_{n,\wt}$ for
    some 
    $\tilde{\alpha}^n\in \A^n$. 
  \end{case}
  Let $m:= \#[\tilde{\alpha}^n]$. We want to estimate
  $\#[\alpha^{n+1}]_{n+1,\wt}$. To do this we will estimate the
  $(n+1)$-angles between two (succeeding with respect to
  $\Sim{n,\wt}$) $n$-angles $\tilde{\alpha}^n_1,
  \tilde{\alpha}^n_2\in[\tilde{\alpha}^n]_{n,\wt}$. For any such
  succeeding $n$-angles there is a white $n$-gap $G^n$ containing
  these succeeding $n$-angles.  Note that there are $m$ succeeding
  $n$-angles in $[\tilde{\alpha}^n]$, thus there are $m$ such $n$-gaps
  $G^n$.

  Consider now the succeeding $(n+1)$-angles between
  $\tilde{\alpha}^n_1, \tilde{\alpha}^n_2$ (i.e., the $(n+1)$-angles
  in $[\alpha^{n+1}]_{n+1,\wt}$ between those $n$-angles). By
  Theorem~\ref{thm:LnLn+1}, these are exactly the $(n+1)$-angles
  contained in $G^n$ that are mapped by $\mu^n$ to $1$-angles which
  succeed $\tilde{\alpha}^0:=\mu^n(\tilde{\alpha}^n)$, i.e., mapped by
  $\mu^n$ to
  $[\tilde{\alpha}^0]_{1,\wt}$ (note that the lower index ``$1$'' is
  \emph{not} a misprint). Let $d=\#[\tilde{\alpha}^0]_{1,\wt}$, then
  there are exactly $d-1$ $(n+1)$-angles in $[\alpha^{n+1}]_{n+1,\wt}$
  between (and distinct from) $\tilde{\alpha}^n_1,
  \tilde{\alpha}^n_2$.  
  % Let $\tilde{\alpha}^0:=\mu^n([\tilde{\alpha}^n]_{n,\wt})\in \A^0$,
  % see ($\LC^n$ \ref{item:propL6}). Let $d:=
  % \#[\tilde{\alpha}^0]_{1,\wt}$. . By Theorem  and ($\LC^n$
  % \ref{item:Ln5}) there are exactly $d-1$ points in $G^n\cap
  % S^1\setminus [\tilde{\alpha}^n]_{n,\wt}$ that are mapped by $\mu^n$ to
  % points in $[\tilde{\alpha}^0]_{1,\wt}$; thus in
  % $[\alpha^{n+1}]_{n+1,\wt}$. 
  The same  
  argument applies to each of the $m$ gaps intersecting
  $[\tilde{\alpha}^n]$. Therefore 
  \begin{equation*}
    \#[\alpha^{n+1}]_{n+1}= m + m(d-1)=md.
  \end{equation*}
  
  By inductive hypothesis it holds
  \begin{equation*}
    m=\#[\tilde{\alpha}^n]_{n,\wt} = d([\tilde{\alpha}^n]_{n,\wt})\cdot
    d([\tilde{\alpha}^{n-1}]_{n-1,\wt}) \cdot 
    \ldots \cdot
    d([\tilde{\alpha}^1])_{1,\wt}, 
  \end{equation*}
  where $[\tilde{\alpha}^{j}]_{j,\wt} :=
  \mu^{n-j}([\tilde{\alpha}^n]_{n,\wt})$ for $j=1,\dots, n$. Note that
  $[\tilde{\alpha}^j]_{j,\wt}\subset [\alpha^{j+1}]_{j+1,\wt}$, thus
  $d([\tilde{\alpha}^j]_{j,\wt})=d([\alpha^{j+1}]_{j+1,\wt})$ for
  $j\geq 1$. By the same argument
  $d=d([\tilde{\alpha}^0]_{1,\wt})=d([\alpha^1])_{1,\wt}$. The claim
  follows.
\end{proof}

\begin{lemma}
  \label{cor:sizes_Julia}
  The equivalence class $[\alpha^n]_{n,\wt}$ is of \emph{Julia-type} if
  and only if 
  \begin{equation*}
    \lim_{m\to \infty} \# [\alpha^n]_{m,\wt} < \infty.   
  \end{equation*}
  In fact then there is an $m_0$ (independent of $n$ and $\alpha^n$)
  such that
  \begin{align*}
    &[\alpha^n]_{m,\wt}=[\alpha^n]_{n+m_0,\wt},
%     \text{ equivalently }
%     \\
%     &L^m([\alpha^n]_{m,\wt})=L^{n+m_0}([\alpha^n]_{n+m_0}), 
  \end{align*}
  for all $\alpha^n\in \A^n$ and $m\geq n+m_0$. Furthermore
  \begin{equation*}
    \#[\alpha^n]_{n,\wt}\leq 2^{d-1},
  \end{equation*}
  for all Julia-type equivalence classes $[\alpha^n]_{n,\wt}$. 
\end{lemma}

\begin{proof}
  It is clear that for any Fatou-type equivalence class
  $[\alpha^n]_{n,\wt}$ it holds
  \begin{equation*}
    {\lim_{m\to\infty} \#[\alpha^n]_{m,\wt}= \infty} 
  \end{equation*}
  by Lemma~\ref{lem:sizes_an}.

  \smallskip
  Let $[\alpha^1]_{1,\wt}$ be of Julia-type ($\alpha^1\in \A^0$). There
  is an $m_0$ (independent of $\alpha^1$), 
  such that $\mu^j([\alpha^1]_{1,\wt})$ is not contained in any critical
  equivalence class (of $\Sim{1,\wt}$), i.e.,
  $d(\mu^j([\alpha^1]_{1,\wt}))=1$, for all $j\geq m_0$.   

  Let $[\alpha^n]_{m,\wt}$ be of Julia-type, $m\geq n+m_0$, $\alpha^n\in
  \A^n$. Then $d(\mu^j([\alpha^n]_{m,\wt}))=1$ for all $j\geq
  n+m_0-1$. This proves the first claim, using Lemma \ref{lem:sizes_an}.  

  \medskip
  To estimate the maximal size of a Julia-type equivalence class, let
  $m_j$ ($j=1,\dots, 
  N$) be the sizes of the Julia-type equivalence classes of
  $\Sim{1,\wt}$. From Lemma \ref{lem:sizes_an} it follows that (for any
  Julia-type equivalence class) $\#[\alpha^n]_{n,\wt}\leq \prod
  m_j$. Maximizing this product subject to $\sum_j(m_j-1)=d-1$
  (see ($\LC^n$ \ref{item:propLnum_L})) yields the second statement.  
\end{proof}

\begin{proof}
  [Proof of Proposition \ref{prop:Fatou_crit_cycles}]
%   Note first the following immediate consequence of Lemma
%   \ref{lem:sizes_an}. 
%   \begin{align*}
%     &\Sim{1,\wt} \text{ has no Fatou-type classes if and only if }
%     \\
%     &\text{there is a constant } M<\infty \text{ such that }
%     \#[\alpha^1]_{n,\wt} \leq M \text{ for all } \alpha^1\in \A^1.
%   \end{align*}
  Assume $F$ has no critical periodic cycles. Then there is a constant
  $M< \infty$ such that $\deg_{F^n}(c)\leq M$ for all $c\in S^2$ and
  $n$. Recall that $\deg_{F^n}(c)$ is the number of white/black
  $n$-tiles 
  attached at the $n$-vertex $c$. Let $\alpha\in S^1$ such that
  $\gamma^n(\alpha)=c$. Then $[\alpha]_{n,\wt}\leq M$,
  $[\alpha]_{n,\bt}\leq M$. Therefore $\Sim{1,\wt},\Sim{1,\bt}$ has no
  Fatou-type classes by Lemma \ref{cor:sizes_Julia}. 

  \medskip
  Assume now that $F$ has critical periodic points. Let us assume
  first that 
  $c=F(c)$ is a critical point. Then there are at least two
  white/black $1$-tiles containing $c$. Thus
  $(\gamma^1)^{-1}(c)=:[\alpha^1]_{1}$ contains at least two
  points. Let $\{\alpha^0\}:=\mu([\alpha^1]_1)$. Recall from
  (\ref{eq:Angn_comm_dia}) that $F\circ
  \gamma^1=\gamma^0\circ \mu$. Furthermore $\mu$ maps $\A^1$ to $\A^0$
  and $\gamma^0=\gamma^1$ on $\A^0$. Thus $\gamma^0\circ \mu
  = \gamma^1\circ \mu$ on $\A^1$, thus
  $c=F\circ \gamma^1(\alpha^1)= \gamma^1(\alpha^0)$. It
  follows that $\alpha^0\in [\alpha^1]_1$, or
  $[\alpha^0]_1=[\alpha^1]_1$.  
  Therefore $[\alpha^0]_{1,\wt}$ or $[\alpha^0]_{1,\bt}$ has to
  contain at least two points (since $\Sim{1}\; = \;\Sim{1,\wt}\vee
  \Sim{1,\bt}$). Note that this equivalence class is of periodic
  Fatou-type.  

  Now assume that $F^n(c)=c$ for some $n\geq 1$. The same argument as
  above yields that there is $[\alpha^n]_{n,\wt}$ or
  $[\alpha^n]_{n,\bt}$, without loss of generality
  $[\alpha^n]_{n,\wt}$, containing at least two points, such that
  $\mu^n([\alpha^n]_{n,\wt})\subset [\alpha^n]_{n,\wt}$. By Lemma
  \ref{lem:sizes_an} one of the classes
  $[\alpha^n]_{n,\wt},[\alpha^{n-1}]_{n-1,\wt}:=\mu([\alpha^n]_{n,\wt}),
  \dots, [\alpha^1]_{1,\wt} :=  \mu^{n-1}([\alpha^n]_{n,\wt})$ has to contain a
  critical class, which is periodic with respect to $\mu^n$. 
\end{proof}
  
\subsection{The equivalence relation $\Approx{\wt}$ }
\label{sec:equiv-relat-simj}

Recall from Section~\ref{sec:critical-portrait} that the equivalence
classes $[\alpha^1]_{1,\wt}$ as well as the equivalence classes
$[\alpha^1]_{1,\bt}$ form \emph{critical portraits} in the sense of
Poirier given in \cite{Poirier}. This
means they define unique monic, centered, postcritically finite
polynomials. Furthermore the equivalence relations induced by their
Carath\'{e}odory semi-conjugacies may be obtained from the critical
portraits. 

Thus (following Poirier) we define an equivalence relation
$\Approx{\wt}$ on $S^1$ which by Poirier equals the one defined in
\eqref{eq:def_approx_w} (see next section).

\smallskip
Recall from (\ref{item:propGnGn+1}) that each white $n$-gap $G^n$
is contained in (exactly) one white $(n-1)$-gap $G^{n-1}$.  
Here and in the following we will consider sequences $(G^n)_{n\in\N}$
of gaps $G^n\in \G^n_{\wt}$
such that 
\begin{align}
  \label{eq:defDn}
  &G^1\supset G^2\supset \dots . 
  \intertext{We write}  
  \notag
  &[(G^n)]:= \bigcap G^n\cap S^1,
\end{align}
%\begin{equation*}  
%\end{equation*}
where it is always understood that the sequence $(G^n)_{n\in\N}$ is as
in (\ref{eq:defDn}).  
Define for $s,t\in S^1$
\begin{equation}
  \label{eq:def_approx}
  s\Sim{G} t \; :\Leftrightarrow \; s,t\in [(G^n)], \text{ for some
    sequence $(G^n)_{n\in \N}$ as above}. 
\end{equation}
%where  $(D^n)_{n\in\N}$ is as in (\ref{eq:defDn}).
Then we define 
\begin{equation}
  \label{eq:defapprox}
  s\Approx{\wt} t :\Leftrightarrow \text{ there are } s_1,\dots, s_N\in S^1 \text{
    such that }
  s=s_1\Sim{G} s_2 \Sim{G} \dots \Sim{G} s_N=t.     
\end{equation}
Note that $\Sim{G}$ is not an equivalence relation, but $\Approx{\wt}$
is. Note that $\Sim{G}$ should be properly equipped with an index
``$\wt$'', which we suppress. The reader should be aware that there are
analogously defined relations in terms of black gaps. 
 
Let us record the following, we set $\A^{\infty}:=\bigcup \A^n$.  
\begin{lemma}[Properties of {$[(G^n)]$}]
  \label{lem:propsimD}
  The sets $[(G^n)]$ satisfy the following. 
  \begin{enumerate}
  \item
    \label{item:propimD1}
    $\#[(G^n)] \leq k$, recall that
    $k=\#\post(F)$.
  \item  
    \label{item:propimD2}

    {If $\alpha\in S^1\setminus \A^{\infty}$ then } $\alpha$ {
      is contained in a \emph{single} set } $[(G^n)]$;  
    \\
    {if $\alpha\in \A^{\infty}$ then $\alpha$ is contained in
      \emph{at most two} such sets.} 
    % \begin{align*}
    % &\text{If $\alpha\in S^1\setminus \A^{\infty}$ then } \alpha \text{
    %   is contained in a \emph{single} set } [(G^n)];  
    % \\
    % &\text{if $\alpha\in \A^{\infty}$ then $\alpha$ is contained in
    %   \emph{at most two} such sets.} 
    % \end{align*}

  \item 
    \label{item:propimD3}
    Let $\alpha^n,\tilde{\alpha}^n\in \A^n$ be \emph{not
      equivalent} with respect to $\Sim{n,\wt}$. Then 
    \begin{equation*}
      \alpha^n,\tilde{\alpha}^n \text{ are \emph{not} in the same set
      } [(G^n)].   
    \end{equation*}
  \item 
    \label{item:propimD4}
    Disjoint sets $[(G^n)]$ are
    \emph{non-crossing}. 
  \end{enumerate}
\end{lemma}

\begin{proof}
  The first property follows from (\ref{item:propL5}). The second
  from the fact that each $\alpha\notin \A^{\infty}$ is contained in a
  \emph{single} $n$-arc for each $n$; each $\alpha\in \A^{\infty}$ is
  contained in at
  most two $n$-arcs (which can be in the same or different sets
  $[(G^n)]$). The third follows from (\ref{item:propDcap}). The last
  property follows from ($\LC^n$ \ref{item:propL2}). 
\end{proof}

We now list properties of $\Approx{\wt}$, its equivalence classes are
denoted by $[\alpha]_{\approx}$. 

\begin{lemma}[Properties of $\Approx{\wt}$]
  \label{lem:prop_approx}
  The equivalence relation $\Approx{\wt}$ satisfies the following.
  \begin{enumerate}
  \item 
    \label{item:prop_approx1}
    If $[\alpha^n]_{n,\wt}$ is of \emph{Julia-type} ($\alpha^n\in \A^n$) then
    \begin{equation*}
      [\alpha^n]_{\approx}\cap \A^n =[\alpha^n]_{n,\wt}.
    \end{equation*}
        If $[\alpha^n]_{n,\wt}$ is of \emph{Fatou-type} ($\alpha^n\in \A^n$) then  
    \begin{equation*}
      [\alpha^n]_{\approx}\cap \A^\infty =\{\alpha^n\}.
    \end{equation*}

  \item
    \label{item:prop_approx2}
    Each equivalence class of $\Approx{\wt}$ is finite, in fact
    \begin{equation*}
      \#[\alpha]_{\approx}\leq (k-1)2^{d-1}. 
    \end{equation*}
  \item 
    \label{item:prop_approx3}
    The number of sets $[(G^n)]$ that may
    form a chain, meaning the number $N$ from (\ref{eq:defapprox}),
    is finite; more precisely
    \begin{equation*}
      N\leq 2^{d-1}. 
    \end{equation*}
  \end{enumerate}
\end{lemma}

\begin{proof}
  (\ref{item:prop_approx1}) 
  From Lemma \ref{lem:propsimD} (\ref{item:propimD2}) it follows that
  distinct sets $[(G^n)]$, $[(\widetilde{G}^n)]$ may only intersect in a
  point $\alpha^n\in 
  \A^{\infty}$. From Lemma \ref{lem:propsimD} (\ref{item:propimD3}) it
  then follows that if $\alpha^n,\tilde{\alpha}^n\in \A^n$ are not
  equivalent with respect to $\Sim{n,\wt}$, then they are not
  equivalent with respect to $\Approx{\wt}$, or
  \begin{equation*}
    [\alpha^n]_{\approx} \cap \A^{n}\subset [\alpha^n]_{n,\wt}.
  \end{equation*}

  Consider a Julia-type equivalence class
  $[\alpha^n]_{n,\wt}$ ($\alpha^n\in \A^n$). 
  %Then $[\alpha^n]_{n+m}\supset [\alpha^n]_{n,\wt}$ is of Julia type for
  %$m\geq 0$.
  Let $m_0$ be the constant from Lemma \ref{cor:sizes_Julia}, thus
  $[\alpha^n]_{m,\wt}=[\alpha^n]_{n+m_0}$ for all $m\geq n+m_0$. Fix a
  $m\geq n+m_0$. 
  Let $\beta^m, \tilde{\beta}^m\in [\alpha^n]_{m,\wt}$ be succeeding
  (with respect to $\Sim{m,\wt}$); there are at most $2^{d-1}$ such
  succeeding $m$-angles (see Lemma \ref{cor:sizes_Julia}). 
  
  There is a gap $G^m\in\G^m_{\wt}$ containing $\beta^m,\tilde{\beta}^m$. 
  Thus succeeding angles of
  $[\alpha^n]_{n+m_0,\wt}$ are equivalent with respect to $\Approx{\wt}$,
  meaning
  \begin{equation*}
    [\alpha^n]_{n,\wt}\subset [\alpha^n]_{n+m_0,\wt}\subset
    [\alpha^n]_{\approx},
  \end{equation*}
  thus $[\alpha^n]_{\approx}\cap \A^{n} = [\alpha^n]_{n,\wt}$
  follows as desired.

  \smallskip Now let $[\alpha^n]_{n,\wt}$ be of Fatou-type
  ($\alpha^n\in \A^n$). Then $\#[\alpha^n]_{m,\wt}\to \infty$ as $m\to
  \infty$ by Lemma \ref{cor:sizes_Julia}. If
  $[\alpha^n]_{m,\wt}\subsetneq [\alpha^n]_{m+1,\wt}$, then succeeding
  angles in $[\alpha^n]_{m,\wt}$ are not succeeding in
  $[\alpha^n]_{m+1,\wt}$ (see Case (2) in the proof of Lemma
  \ref{lem:sizes_an}). Thus a set $[(G^n)]$ may contain at most one
  point from $[\alpha^n]_{n,\wt}$. From
  Lemma~\ref{lem:propsimD}~(\ref{item:propimD2}) and
  Lemma~\ref{lem:propsimD}~(\ref{item:propimD3}) it follows that
  \begin{equation*}
     [\alpha^n]_{\approx} \cap \A^\infty = \{\alpha^n\}.
  \end{equation*}

  \medskip
  (\ref{item:prop_approx3}) From the above it follows that $N = \max
  \#[\alpha^n]_{n,\wt}$, where the maximum is taken over all Julia-type
  classes. The claim thus follows from %(\ref{item:prop_approx1}) and
  Lemma~\ref{cor:sizes_Julia}.

  \medskip
  (\ref{item:prop_approx2})
  Let $G^m\in \G^m_{\wt}$ be the gap containing succeeding $m$-angles
  $\beta^m,\tilde{\beta}^m\in [\alpha^n]_{m,\wt}$, 
  as described in the proof of (\ref{item:prop_approx1}). Then
  $[(G^m)]$ contains at most $k-2$ points distinct from
  $\beta^m,\tilde{\beta}^m$ by Lemma \ref{lem:propsimD}
  (\ref{item:propimD1}). Thus 
  \begin{equation*}
    \#[\alpha^n]_{\approx}\leq \#[\alpha^n]_{m,\wt} + 
    (k-2)\#[\alpha^n]_{m,\wt} \leq (k-1)2^{d-1},  
  \end{equation*} 
  by Lemma \ref{cor:sizes_Julia}.
\end{proof}

\begin{prop}
  \label{prop:approx_closed}
  The equivalence relation $\Approx{\wt}$ is \emph{closed}.
\end{prop}

\begin{proof}
  We first show the corresponding result for $\Sim{G}$, i.e., the
  following.
  \begin{claim}
    Let $(s_j),(t_j)\subset S^1$ be sequences such that
    \begin{align*}
      &s_j\to s, t_j\to t, \; \text{ and } s_j \,\Sim{G}\, t_j 
      \text{ for all }j, 
      \intertext{then}
      &s\Sim{G} t. 
    \end{align*}
  \end{claim}
  If $s_j=s$ is constant the claim follows, since the set
  $\{t_j\Sim{G} s\}$ is finite (see
  Lemma~\ref{lem:propsimD}~(\ref{item:propimD1})).

  Assume now that $(s_j)$ is not constant. Without loss of generality
  we can assume that $(s_j)$ is strictly increasing, the sets
  $[(G^n_j)]\ni s_j,t_j$ are disjoint (for distinct lower indices
  $j,j'\in\N$), and $s_j\neq t_j$ for all $j$. Since disjoint sets
  $[(G^n_j)]$ are non-crossing (Lemma \ref{lem:propsimD}
  (\ref{item:propimD4})) it follows that $(t_j)$ is strictly
  decreasing. Let $a^n(s)\subset S^1$ be an $n$-arc containing $s$. If
  $s$ is contained in two $n$-arcs, i.e., if $s\in \A^n$, we choose
  $a^n(s)$ as the $n$-arc having $s$ as the right endpoint. Similarly
  let $a^n(t)\subset S^1$ be an $n$-arc containing $t$. If $t\in
  \A^n$, let $a^n(t)$ be the $n$-arc with $t$ as the left endpoint.

  \smallskip
  For each $n$ the points $s_j,t_j$ are in the interiors of
  $a^n(s),a^n(t)$ respectively for sufficiently large $j$. Since
  $s_j\Sim{G}t_j$ it follows that $a^n(s),a^n(t)$ are contained in the
  same gap $G^n\in\G^n_{\wt}$. 
  Thus $s,t\in [(G^n)]$ proving the claim. 
  
  \medskip
  Consider now sequences $(s^n),(t^n)\subset S^1$, where $s^n\to s,
  t^n\to t$, such that $s^n\Approx{\wt} t^n$ for all $n$. Thus by Lemma
  \ref{lem:prop_approx} (\ref{item:prop_approx3}) there are $s^n_j\in
  S^1$ such that 
  \begin{equation*}
    s^n=s^n_1\;\Sim{G}\; s^n_2\;\Sim{G} \;\dots\; \Sim{G}\; s^n_N=t^n.
  \end{equation*}
  Here $N\in \N$ is \emph{independent} of $n$. 
  By taking subsequences we can assume that $s^n_j\to s_j$ as $n\to
  \infty$, for all $j$. From the previous claim it follows that
  $s_j\Sim{G} s_{j+1}$ (for $j=1,\dots N-1$). 
  Thus
  \begin{equation*}
    s = s_1\;\Sim{G}\; \dots \;\Sim{G}s_N=t, 
  \end{equation*}
  meaning $s\Approx{\wt} t$. 

\end{proof}

\subsection{The white polynomial}
\label{sec:white-polynomial}

A.\ Poirier \cite{Poirier}, extending work of Bielefeld-Fisher-Hubbard
\cite{MR1149891}, has shown that postcritically finite polynomials
admit a combinatorial classification in terms of external rays. The
result is paraphrased here, not in full generality, but only in the
relevant case at hand.

\begin{poiriers_theorem*}[\cite{Poirier}, \cite{MR1149891}]
  Let the sets $[\alpha^1_j]_{1,\wt} \subset \Q/\Z\subset \R/\Z=S^1$,
  $j=1,\dots, m$ form a \emph{critical portrait} as in
  Definition \ref{def:critical_portrait}. Then there is a unique
  monic, 
  centered, postcritically finite polynomial $P_{\wt}$ such that
  \begin{enumerate}%[(${Poi}$ 1)]
    \renewcommand{\theenumi}{Poi \arabic{enumi}}
  \item 
    \label{item:Poi1}
    the equivalence realtion $\Approx{\wt}$ (as defined in the last
    section) is the equivalence relation induced by the
    Carath\'{e}odory semi-conjugacy of $P_\wt$, meaning that
    \begin{equation*}
      s\Approx{\wt} t \Leftrightarrow \sigma_\wt(s)=\sigma_\wt(t); 
    \end{equation*}
    where $\sigma\colon S^1=\R/\Z\to \partial \K_\wt =\J_\wt$ is
    defined as in Section~\ref{sec:carath-semi-conjugacy-polyn}. 
    % the extension of the
    % Riemann map $\psi \colon \CDach\setminus \Dbar \to \CDach
    % \setminus \K$ (normalized 
    % by $\psi(\infty)=\infty$, $\psi'(\infty)> 0$), $\K$ is the filled
    % Julia set, and $\J$ 
    % the Julia set of $P_{\wt}$. 
    This is \cite[Theorem~1.13]{Poirier} and
    \cite[Proposition~7.7]{Poirier}. In particular each set
    $\sigma([\alpha]_\approx)$ is a single point.

%    This is Theorem I 3.9 in \cite{Poirier} and
%    Proposition II 3.6 in \cite{Poirier}. 

    \smallskip
  \item 
    \label{item:Poi2}
    If $[\alpha^n]_{n,\wt}$ is of Fatou-type, then all points in
    $\sigma([\alpha^n]_{n,\wt})$ are in the boundary of the \emph{same}
    bounded Fatou component of $P_{\wt}$. Distinct sets $[\alpha^n]_{n,\wt},
    [\tilde{\alpha}^n]_{n,\wt}$ are in the boundaries of \emph{distinct}
    bounded Fatou components. This is
    \cite[Proposition~8.4]{Poirier}. 
    %\cite[Proposition II 4.7]{Poirier}. 
    Furthermore for each bounded
    Fatou component $A$ there is a 
    $\alpha^n\in \A^n$ such that $\sigma(\alpha^n)\in \clos A$, where
    $[\alpha^n]_{n,\wt}$ is of Fatou-type.  

    \smallskip
    Since $P_{\wt}^{-1}(\sigma(\A^n))= \sigma(\A^{n+1})$ it follows that
    (for every Fatou-type class $[\alpha^n]_{n,\wt}$)
    \begin{equation*}
      \sigma([\alpha^n]_{\infty,\wt}) 
      \text{ is \emph{dense}}
    \end{equation*}
    in the boundary of the (bounded) Fatou component $A$ satisfying
    $\sigma([\alpha^n]_{n,\wt})\subset \partial A$.

    \smallskip
  \item 
    \label{item:Poi3}
    For each gap $G^n\in\G^n_{\wt}$ the set 
    \begin{align*}
      &\sigma(G^n) \text{ is a \emph{connected} subset of
        the Julia set of }P_{\wt}. 
      \intertext{The images of two disjoint gaps $G^n,\widetilde{G}^n\in
        G^n_{\wt}$ are 
        disjoint:} 
      &G^n \cap \widetilde{G}^n =\emptyset \quad
      \Rightarrow \quad
      \sigma(G^n\cap S^1)\cap \sigma(\widetilde{G}^n\cap
      S^1)=\emptyset;
    \end{align*}
    see \cite[Section~5]{Poirier}. 
    %see \cite[Chapter II.2]{Poirier}. 
%     From the first of these
%     properties it follows immediately that
%     \begin{equation*}
%       \sigma([(G^n)]) \text{ is a \emph{single} point.}
%     \end{equation*}
  \end{enumerate}
\end{poiriers_theorem*}

From now on the \emph{white polynomial} $P_{\wt}$ will be the one
obtained from Poirier's Theorem from the sets $[\alpha^1]_{1,\wt}$,
$\alpha^1,\in \A^1$.

\subsection{Outline of the proof of Poirier's Theorem}
\label{sec:outl-proof-poir}

Poirier's definition of the critical portrait of a postcritically finite
polynomial is slightly different from ours. This is due to the fact
that he describes general such polynomials, not just ones with
``separated Fatou set'' (see Proposition \ref{prop:Fatou_Pw}) as
considered here. For the convenience of the reader we give a very
brief outline of the proof of the main result from Poirier's Theorem,
namely the existence of the polynomial $P_{\wt}$. 

\smallskip
Consider a topological polynomial, i.e., a Thurston map $P\colon S^2
\to S^2$ such that $P^{-1}(\infty)= \infty$. It is well known that $P$
is ``Thurston equivalent'' to a polynomial if and only if it has no
``Levy cycle'' (Theorem 5.4 and Theorem 5.5 in \cite{MR1149891}). A
Levy cycle is a Jordan curve $\Gamma\subset S^2\setminus \post(P)$
such that 
\begin{itemize}
\item each component of $S^2\setminus \Gamma$ contains at least
  two postcritical points and
\item some component $\Gamma'$ of $P^{-j}(\Gamma)$ is isotopic rel.\
  $\post(P)$ to $\Gamma$ for some $j$; and the map
  \begin{equation*}
    P^j\colon \Gamma'\to \Gamma \text{ is of degree $1$.}
  \end{equation*}
\end{itemize}
To prove Poirier's Theorem (in our special case) one constructs a
(postcritically finite) topological polynomial from the critical
portrait. For each $\alpha^1\in \A^1$ there is an ``extended external
ray'' $R(\alpha^1)$ that is mapped by $P$ to $R(\mu(\alpha^1))$. 
The extended external
rays associated to the angles of one equivalence class
$[\alpha^1]_{1,\wt}$ intersect 
in a point. Assume there is a Levy cycle $\Gamma$ (i.e., $P$ is not
equivalent to a polynomial). We can choose $\Gamma$ in such a way that
$\Gamma$ intersects no preperiodic extended external ray (Lemma 8.7 in
\cite{MR1149891}). From the last property of the critical portrait
(Definition \ref{def:critical_portrait}) it follows that two
postcritical points are separated by some preperiodic extended
external rays. Thus $\Gamma$ contains at most one postcritical point
in its interior, giving a contradiction. 

\subsection{The Fatou set of $P_{\wt}$}
\label{sec:fatou-set-pw}

Here we show that the Fatou set of $P_{\wt}$ is ``separated''.
\begin{prop}
  \label{prop:Fatou_Pw}
  The Fatou set of $P_{\wt}$ has the following property.
  \begin{itemize}
  \item The closures of two distinct bounded components
    $A_1,A_2$ of the Fatou set of $P_{\wt}$ are disjoint,
    \begin{equation*}
      \clos A_1 \cap \clos A_2 =\emptyset.
    \end{equation*}
  \item No bounded Fatou component of $P_{\wt}$ contains a point
    $\sigma([\alpha^n]_{n,\wt})$, where
    $[\alpha^n]_{n,\wt}$ is of 
    Julia-type, in its boundary (recall from (\ref{item:Poi1}) and Lemma
    \ref{lem:prop_approx} (\ref{item:propimD1}) that $\sigma$ maps
    $[\alpha^n]_{n,\wt}$ to a single point).  
  \end{itemize}
\end{prop}

We need some preparation to prove this proposition. The key is an
explicit description of the set of angles in $S^1$ that are mapped by
$\sigma$ to 
the boundary of a given bounded Fatou component/critical point (or
more generally $n$-vertex).

\smallskip
Fix an $n$-angle $\alpha^n\in \A^n$. We will consider the equivalence
class $[\alpha^n]_{m,\wt}$ for some $m\geq n$. 
Let $G^m_1,\dots, G^m_{N_m}\in \G^m_{\wt}$ be the gaps intersecting
$[\alpha^n]_{m,\wt}$.  
Here $N_m= \# [\alpha^n]_{m,\wt}$. 
\begin{lemma}
  \label{lem:Dm_cpt_cont}
  In the setting as above 
  \begin{equation*}
    \bigcup_j G^{m+n_0}_j\cap S^1 \text{ is compactly contained in } 
    \bigcup_j G^m_j\cap S^1,
  \end{equation*}
  for all $m\geq n_0$, where $n_0$ is the constant from
  \upshape{(\ref{item:propDcap})}.  
\end{lemma}

\begin{proof}
  Note first that every $m$-angle in $[\alpha^n]_{m,\wt}$ is
  contained in two $m$-arcs. Thus $[\alpha^n]_{m, \wt}$ is contained
  in the interior of $\bigcup_j G^m_j\cap S^1$.

  From Theorem~\ref{thm:LnLn+1} and ($\LC^n$~\ref{item:Ln5})) it
  follows that $[\alpha^n]_{i,\wt}\subset \bigcup_j G^m_j$ for all
  $i\geq m$ by construction

  \smallskip
  Every boundary point of $\bigcup_j G^m_j\cap S^1$ is a
  point $\tilde{\alpha}^m\in \A^m$ not equivalent (with respect to
  $\Sim{m,\wt}$) to $\alpha^n$. The statement follows from
  (\ref{item:propDcap}).  
\end{proof}

\begin{lemma}
  \label{lem:sigma_Fatou}
  Let $A$ be a bounded component of the Fatou set of $P_{\wt}$, $\alpha^n\in
  \A^n$ such that $\sigma(\alpha^n)\in \partial A$. The gaps
  $G^m_j=G^m_j(\alpha^n)\in \G^m_{\wt}$ are the ones 
  intersecting $[\alpha^n]_{m,\wt}$ as before. Then
  \begin{equation*}
    \sigma^{-1}(\partial A)= \bigcap_m \bigcup_j G^m_j \cap S^1.
  \end{equation*}
\end{lemma}

\begin{proof}
  The right hand side of the above expression is compact and contains
  all points of $[\alpha^n]_{\infty,\wt}$. Since the set
  $\sigma([\alpha^n]_{\infty,\wt})$ is dense in $\partial A$
  (\ref{item:Poi2}), it follows that $\sigma(\bigcap_m \bigcup_j G^m_j
  \cap S^1)\supset \partial A$. Note that
  \begin{equation*}
    \bigcap_m \bigcup_j G^m_j = \bigcup [(G^n)], 
  \end{equation*}
  where the union on the right hand side is taken over all sequences
  of white gaps $G^1\supset G^2\supset \dots$ such that $G^n\cap
  [\alpha^n]_{\infty,\wt}\neq \emptyset$ for all $n$. For each such
  sequence the point $\sigma([(G^n)])$ is an accumulation point of
  $\sigma([\alpha^n]_{\infty,\wt})$, thus
  \begin{equation*}
    \sigma\big(\bigcap_m \bigcup_j G^m_j \cap S^1\big)=\partial A.    
  \end{equation*}

  \smallskip
  Consider a gap $\widetilde{G}^m\in \G^m_{\wt}$ 
  that is distinct from all gaps $G^m_j$. From Lemma
  \ref{lem:Dm_cpt_cont} it follows that $\widetilde{G}^m\cap
  S^1$ and $G^{m+n_0}_j  \cap S^1$ are disjoint. By
  (\ref{item:Poi3}) it follows that these sets are mapped to disjoint
  sets by $\sigma$, thus $\sigma(\widetilde{G}^m\cap S^1)\cap
  \partial A=\emptyset$. This proves the claim. 
\end{proof}

% Note that we can write the above as
% \begin{equation*}
%   \sigma^{-1}(\partial \F)= \bigcup [(G^m)],
% \end{equation*}
% where the union is taken over all sets $[(G^m)]$ such that $G^m\cap
% [\alpha^n]_{m,\wt}\neq \emptyset$ for all $m$. 
The same argument as
above applies to Julia-type equivalence classes (see
Lemma~\ref{cor:sizes_Julia} and Lemma~\ref{lem:prop_approx} 
(\ref{item:prop_approx1})). 

\begin{cor}
  \label{cor:JuliaType}
  Let $[\alpha^n]_{n,\wt}$ ($\alpha^n\in \A^n$) be of \emph{Julia-type},
  $\sigma(\alpha^n)=v$. Then
  \begin{equation*}
    \sigma^{-1}(v)= \bigcup [(G^m)],
  \end{equation*}
  where the (finite) union is taken over all sets $[(G^m)]$ such that
  $G^m\cap [\alpha^n]_{m,\wt}\neq \emptyset$ for all $m$. 
\end{cor}

Proposition \ref{prop:Fatou_Pw} now follows using
(\ref{item:propDcap}). 

\subsection{The equivalence relation $\Approx{\F,\wt}$}
\label{sec:equiv-relat-simp}

We consider the
equivalence relation $\Approx{\F,\wt}$ obtained from the Carath\'{e}odory
semi-conjugacy, together with the identification of Fatou components as in
(\ref{eq:def_approxFw}). Our main objective is to show the following. 
\begin{prop}
  \label{prop:SimPFSimw}
  We have
  \begin{equation*}
    \Approx{\F,\wt}\;=\;\Sim{\wt}.
  \end{equation*}
\end{prop}
Recall that $\Sim{\wt}$ is the equivalence relation from Definition
\ref{def:simw_simb}. 
To prove this proposition some preparation is needed
first. Let us first note the following, which is an immediate
consequence of Proposition \ref{prop:Fatou_Pw}.

\begin{lemma}
  \label{lem:ApproxF_equiv}
  The relation $\Approx{\F,\wt}$ is an \emph{equivalence relation}. 
\end{lemma}

We write $[s]_{\F,\wt}:= \{t\in S^1\mid s\Approx{\F,\wt} t\}$ for equivalence
classes of $\Approx{\F,\wt}$. A description of them follows
immediately from Section \ref{sec:fatou-set-pw}. 
 
\begin{lemma}
  \label{lem:Fatou_classes_ApproxF}
  Consider $[s]_{\F,\wt}$ ($s\in S^1$). Either 
  \begin{itemize}
  \item 
    $[s]_{\F,\wt}\cap \A^\infty=\emptyset$, then 
    \begin{equation*}
      [s]_{\F,\wt} = [(G^m)],
    \end{equation*}
    for one sequence $G^1\supset G^2\supset \dots$ of white
    gaps. Note that in this case
    \begin{equation*}
      [s]_{\F,\wt} \text{ is \emph{compactly contained} in } G^n 
      \text{ for all }n. 
    \end{equation*}
  \item Or there is $\alpha^n\in [s]_{\F,\wt}$, $\alpha^n\in \A^n$. Then
    \begin{equation*}
      [s]_{\F,\wt}=[\alpha^n]_{\F,\wt}  = \bigcup [(G^m)],  
    \end{equation*}
    where the union is taken over all sets $[(G^m)]$ (as in
    (\ref{eq:defDn})),  satisfying
    $G^m\cap [\alpha^n]_{m,\wt}\neq \emptyset$ for all $m$. Again
    \begin{equation*}
      [s]_{\F,\wt} \text{ is \emph{compactly contained} in } 
      \bigcup_j G^m_j,
    \end{equation*}
    for all $m$, where the union is taken over all white $m$-gaps
    intersecting $[\alpha^n]_{m,\wt}$.  
  \end{itemize}
\end{lemma}

\begin{lemma}
  \label{lem:approxFclosed}
  The equivalence relation $\Approx{\F,\wt}$ is \emph{closed}. 
\end{lemma}

\begin{proof}
  Consider a convergent sequence $s_n\to s$ in $S^1$. Let $t_n\in
  [s_n]_{\F,\wt}$ for all $n$, such that $t_n\to t$. We want to show that
  $s\Approx{\F,\wt}t$, i.e., $t\in [s]_{\F,\wt}$. This is clearly the case when
  $s_n=s$ is constant, since $[s]_{\F,\wt}$ is compact.

  Thus we can assume that $s_n\neq s$ for all $n$. Fix an $m$. Since
  $\A^m$ is a finite set it follows that $[s_n]_{\F,\wt}\cap
  \A^m=\emptyset$ for 
  sufficiently large $n$. Thus we can assume that $[s_n]_{\F,\wt} \cap
  \A^n = \emptyset$,   for all $n$. It follows that 
  \begin{equation*}
    [s_n]_{\F,\wt} \text{ is contained in a
      \emph{single} white $n$-gap $G^n$}. 
  \end{equation*}
  Assume first that $[s]_{\F,\wt}$ contains $\alpha^n\in \A^n$. Let
  $G^m_j\in \G^m_{\wt}$ be the gaps intersecting $[\alpha^n]_{m,\wt}$ as in
  Lemma \ref{lem:Fatou_classes_ApproxF}. Since $[s]_{\F,\wt}$ is compactly
  contained in $\bigcup_j  G^m_j$ it follows that $s_n$ is in the
  interior of $\bigcup_j G^m_j$ for sufficiently large $n$. Taking a
  subsequence if necessary as before, we can assume that
  \begin{equation*}
    s_m\text{ is in the \emph{interior} of } 
    \bigcup_j G^m_j,
  \end{equation*}
  for all $m$. Thus the white $m$-gap $G^m\supset [s_m]_{\F,\wt}$ equals
  one of the gaps $G^m_j$. It follows 
  that $t\in \bigcap_m \bigcup_j G^m_j= [s]_{\F,\wt}$ as desired.  

  \smallskip
  The case when $[s]_{\F,\wt}$ contains no angle in $\A^\infty$ is proved
  by exactly the same argument. 
\end{proof}

\begin{proof}[Proof of Proposition \ref{prop:SimPFSimw}]
  From the second part in Lemma \ref{lem:Fatou_classes_ApproxF} it
  follows that $[\alpha^n]_{n,\wt}\subset [\alpha^n]_{\F,\wt}$ for all
  $\alpha^n\in \A^n$. Thus
  \begin{equation*}
    \Sim{n,\wt} \;\leq \; \Approx{\F,\wt}
  \end{equation*}
  for all $n$. It follows that $\Sim{\infty,\wt}\; \leq\;
  \Approx{\F,\wt}$. Since $\Approx{\F,\wt}$ is closed by
  Lemma~\ref{lem:approxFclosed} it follows that
  \begin{equation*}
    \Sim{\wt} \;\leq\; \Approx{\F,\wt}.
  \end{equation*}

  To see the reverse inequality we first prove the following.
  \begin{claim}
    $s,t\in [(G^n)]$ implies $s\Sim{\wt} t$.
  \end{claim}
  Recall from (\ref{item:propL5}) that $G^n\cap S^1$ consists of $k$
  $n$-arcs $[\alpha^n_0,\beta^n_0], \dots, [\alpha^n_{k-1},
  \beta^n_{k-1}]$, where $\beta^n_j \Sim{n,\wt} \alpha^n_{j+1}$ (lower
  index taken $\bmod k$). Since $\lim_n \alpha^n_j = \lim_n \beta^n_j$
  for all $j$ the claim follows.    

  \smallskip
  From the claim it follows using Lemma~\ref{lem:Fatou_classes_ApproxF} that
  $\Approx{\F,\wt} \; \leq \; \Sim{w}$, finishing the proof.
  % \begin{equation*}
  %   \Sim{\infty,\wt} \;\leq \; \Approx{\F,\wt} \;\leq \; \Sim{\wt}.
  % \end{equation*}
  % Taking the closures (as in Lemma \ref{lem:usc_closure}) yields the
  % result by Lemma \ref{lem:approxFclosed}. 
\end{proof}

\subsection{The black polynomial $P_{\bt}$}
\label{sec:black-polynomial-pb}

 The
\emph{black polynomial} $P_{\bt}$ is the one obtained from Poirier's
Theorem from the black critical portrait, i.e., the sets
$[\alpha^1]_{1,\bt}$ (for all $\alpha^1\in \A^1$).  
% The \defn{black polynomial} $P_{\bt}$ is defined analogously in terms of
% the sets $[\alpha^1]_{1,\bt}$. 
More precisely $P_{\bt}$ is the (unique monic, centered,
postcritically 
finite) polynomial such that the
equivalence relation defined by (for all $z,w\in S^1=\partial D$)
\begin{equation*}
  s\Approx{\bt} t :\Leftrightarrow \sigma_{\bt}(\bar{z})= \sigma_{\bt}(\overline{w})
\end{equation*}
is equal to the equivalence relation relation $\Approx{\bt}$ defined in
terms of the black gaps as in Section~\ref{sec:equiv-relat-simj}. Here
$\sigma_{\bt}$ is a Carath\'{e}odory semi-conjugacy of the Julia set of
$P_{\bt}$. 
The
equivalence relation $\Approx{\F,\bt}$ on $S^1$ is then defined as in
(\ref{eq:def_approxFb}). As in Proposition~\ref{prop:SimPFSimw} it
follows that
\begin{equation*}
  \Approx{\F,\bt}\;= \;\Sim{\bt},
\end{equation*}
where $\Sim{\bt}$ was defined in Definition \ref{def:simw_simb}.

\subsection{Proof of Theorem \ref{thm:mating1}}
\label{sec:proof-theorem}

We assume now that $F$ has no periodic critical points.  This means 
that there are
no Fatou-type equivalence classes of $\Sim{1,\wt},\Sim{1,\bt}$
(Proposition \ref{prop:Fatou_crit_cycles}), hence no Fatou-type
classes of $\Sim{n,\wt}$, $\Sim{n,\bt}$. 
The white polynomial $P_{\wt}$ is defined as in Section
\ref{sec:white-polynomial}, the black polynomial $P_{\bt}$ as in Section
\ref{sec:black-polynomial-pb}. 

\smallskip
From (\ref{item:Poi2}) it
follows that the Fatou sets of $P_{\wt},P_{\bt}$ have no bounded components,
thus their Julia sets are dendrites. Let $\Approx{\wt},
\Approx{\F,\wt}$ be the equivalence relations (on $S^1$) from Section
\ref{sec:equiv-relat-simj} and Section
\ref{sec:equiv-relat-simp}. Then $\Approx{\bt}, \Approx{\F,\bt}$ are
defined analogously in terms of the black
equivalence relations $\Sim{n,\bt}$. Since $P_{\wt},P_{\bt}$ have no bounded
Fatou components it follows from Proposition \ref{prop:SimPFSimw} that
\begin{equation*}
  \Approx{\wt}\;=\;\Approx{\F,\wt}\;= \;\Sim{\wt} 
  \text{ and } 
  \Approx{\bt}\;=\;\Approx{\F,\bt}\;=\; \Sim{\bt}.
\end{equation*}
Recall that $\sim$ is the equivalence relation (on $S^1$) induced by
the invariant Peano curve (\ref{eq:eq_rel}). From
Lemma~\ref{lem:Fnocrit_simwsimb}~\eqref{item:sim_simwsimb} it follows
that $\sim\;=\; \Sim{\wt} \vee \Sim{\bt}\;= \;\Approx{\wt} \vee \Approx{\bt}\;=\;
\approx$. 
Theorem~\ref{thm:mating1} now follows using Theorem~\ref{thm:S1simS2}
and Lemma~\ref{lem:mating_dendrites}.   
 
\section{Proof of Theorem \ref{thm:mating2}}
\label{sec:proof-theorem-1}
We finish the proof of Theorem \ref{thm:mating2} here.
The white/black polynomials $P_{\wt},P_{\bt}$ are defined as in Section
\ref{sec:white-polynomial} and \ref{sec:black-polynomial-pb}. 

Recall
from (\ref{eq:def_approxFw}), (\ref{eq:def_approxFb}) the definition
of the associated equivalence relations $\Approx{\F,\wt},
\Approx{\F,\bt}$. In Proposition \ref{prop:SimPFSimw} it was shown that
$\Approx{\F,\wt}\;=\;\Sim{\wt}$ as well as $\Approx{\F,\bt}\;=\;
\Sim{\bt}$. Let $\sim$ be the closure of $\Approx{\F,\wt}\vee
\Approx{\F,\bt}\;= \;\Sim{\wt}\vee \Sim{\bt}$. In
Lemma~\ref{lem:Fnocrit_simwsimb} it was shown that $\sim$ is the equivalence
relation induced by the invariant Peano curve $\gamma$. 
Recall the definition of $P_{\wt}\widehat{\mate}\, P_{\bt}$ from
(\ref{eq:PwPbsimhat}). In Section~\ref{sec:proof-lemma-top_conj} we
will show the following lemma. 

\begin{lemma}
  \label{lem:PwPb_zdsim}
  The map 
  \begin{equation*}
    P_{\wt}\widehat{\mate}\, P_{\bt}\colon \K_{\wt}\widehat{\mate}\, \K_{\bt} \to 
    \K_{\wt} \widehat{\mate}\, \K_{\bt}
  \end{equation*}
  is well defined and topologically conjugate to 
  \begin{equation*}
    z^d/\!\sim \colon S^1/\!\sim \,\to S^1/\!\sim. 
  \end{equation*}
\end{lemma}

Theorem \ref{thm:mating2} follows, using Theorem \ref{thm:S1simS2}. 

\subsection{Closures}
\label{sec:closures}
%We first need to show that the map $P_{\wt}\widehat{\mate} \,P_{\bt}$ is well
%defined. 
Here we collect some elementary lemmas that will be needed.

\smallskip
Let $S,S'$ be compact metric spaces, $h\colon S\to S'$ be a
continuous surjection, and $\approx$ be an equivalence relation on
$S'$. The equivalence relation $\sim$ on $S$ defined by
\begin{equation*}
  s\sim t :\Leftrightarrow h(s)\approx h(t), \text{ for all } s,t\in S, 
\end{equation*}
is called the \defn{pullback} of $\approx$ by $h$.
\begin{lemma}
  \label{lem:hclosedS}
  In the setting as above, $\sim$ is closed if and only
  if $\approx$ is closed.  
\end{lemma}

The proof is straightforward and left as an exercise. 

We now assume that an equivalence relation $\sim$ is defined on $S$,
we want to define a corresponding equivalence relation $\approx$ on
$S'$. 

\begin{lemma}
  \label{lem:sim_push}
  Let $h\colon S\to S'$ be a surjection, $\sim$ an equivalence
  relation on $S$ that is bigger than the one induced by $h$
  ($h(s)=h(t) \Rightarrow s\sim t$ for all $s,t\in S$). Define
  $\approx$ on $S'$ as follows
  \begin{equation*}
    s'\approx t' :\Leftrightarrow  \text{ there exists } s,t\in S
    \text{ such that } s'=h(s), t'=h(t), 
  \end{equation*}
  for all $s',t'\in S'$. Then $\approx$ is an equivalence relation (on
  $S'$) such that $\sim$ is the pullback of $\approx$ by $h$. 
\end{lemma}

\begin{proof}
  It is straightforward to check that $\approx$ is an equivalence
  relation, which is left as an exercise. 

  If $s\sim t$, then $s'=h(s)\approx\; t'=h(t)$ for all $s,t\in S$.

  Now let $s'\approx t'$ for some $s',t'\in S'$. Consider $\tilde{s},
  \tilde{t}\in S$ with $s'=h(\tilde{s})$, $t'=h(\tilde{t})$. We want
  to show that $\tilde{s}\sim \tilde{t}$.  There are $s,t\in S$ with
  $s\sim t$ and $h(s)=s'$, $h(t)=t'$. Since $\sim$ is bigger than the
  equivalence relation induced by $h$ it follows that $\tilde{s} \sim
  s \sim t \sim \tilde{t}$.
\end{proof}

\begin{lemma}
  \label{cor:hclusure}
  Let $h\colon S\to S'$ be a continuous surjection, where $S,S'$ are
  compact metric spaces. Let $\sim, \approx$ be equivalence relations
  on $S,S'$, such that $\sim$ is the pullback of $\approx$ by $h$. 
  Let $\simhat$ be the closure of
  $\sim$, $\widehat{\approx}$ the closure of $\approx$. Then $\simhat$
  is the pullback of $\widehat{\approx}$ by $h$. 
\end{lemma}

\begin{proof}
  Recall from the proof of Lemma~\ref{lem:usc_closure} that the
  closure of an equivalence relation is given by the intersection of
  all bigger closed equivalence relations. 

  The pullback of $\widehat{\approx}$ is closed by
  Lemma~\ref{lem:hclosedS}, as well as bigger than $\sim$, hence
  bigger than $\widehat{\sim}$. 

  Now consider the equivalence relation $\approx'$ on $S'$ induced by
  $\widehat{\sim}$ and $h$ as in Lemma~\ref{lem:sim_push}. It is
  bigger than $\approx$. Furthermore $\approx'$ pulls
  back to $\widehat{\sim}$ by $h$ and is closed by
  Lemma~\ref{lem:hclosedS}. It follows that $\widehat{\sim}$ is bigger
  than the pullback of $\widehat{\approx}$ by $h$. 
\end{proof}

Consider now a continuous surjection $\mu\colon S\to S$ on a compact
metric space $S$. An equivalence relation $\sim$ on $S$ is called
\defn{invariant} with respect to $\mu$ if
\begin{equation*}
  s\sim t \Rightarrow \mu(s) \sim
  \mu(t) \text{ for all } s,t\in S. 
\end{equation*}
\begin{lemma}
  \label{lem:closure_dynamics}
  Let $\mu\colon S\to S$ be continuous, surjective; $\sim$ an
  equivalence relation on $S$ invariant
  with respect to $\mu$. 
  Then the closure $\simhat$ of $\sim$ is
  invariant with respect to $\mu$. 
\end{lemma}

\begin{proof}
  Consider the equivalence relation $\approx$ given by
  \begin{equation*}
    s\approx t :\Leftrightarrow \mu(s) \simhat \mu(t),
  \end{equation*}
  for all $s,t\in S$. From Lemma \ref{lem:hclosedS} it follows that
  $\approx$ is closed. Note that 
  \begin{equation*}
    s\sim t \Rightarrow \mu(s)\sim \mu(t) \Rightarrow \mu(s)\simhat
    \mu(t)
    \Rightarrow s\approx t,
  \end{equation*}
  meaning that $\approx\;\geq\; \sim$. Note that the meet of any two
  closed equivalence relations is closed (see (\ref{eq:defmeet})),
  thus the meet
  \begin{align*}
    &\approx \wedge \;\simhat \text{ is closed and bigger than $\sim$; which implies\ }
    \\
    &\simhat \; = \; \approx \wedge\, \simhat.
  \end{align*}
  Thus for all $s,t\in S$
  \begin{equation*}
    s\;\simhat\; t \Rightarrow s\,\approx \,t \Rightarrow \mu(s)\, \simhat\,\mu(t),
  \end{equation*}
  finishing the proof. 
\end{proof}

In the next lemma we ``take the closure of a commutative diagram and
show that everything goes well''.   Let $\sim,\approx $ be equivalence
relations on compact metric 
spaces $S,S'$. The maps $\mu\colon S\to S$, 
$\varphi\colon S'\to S'$ as well as $h\colon S\to S'$
are continuous surjections such that the following
diagram commutes  
  \begin{equation*}
    \xymatrix{
      (S,\sim) \ar[r]^{\mu} \ar[d]_{h}
      &
      (S,\sim) \ar[d]^{h}
      \\
      (S',\approx) \ar[r]_\varphi & (S',\approx).
    }
  \end{equation*}
  By this is meant that $\varphi\circ h = h \circ \mu$. The
  equivalence relation 
  $\sim$ is invariant with respect to $\mu$, and
  $\approx$ is invariant with respect to $\varphi$. Furthermore $\sim$
  is the pullback of $\approx$ by $h$. 
 
\begin{lemma}
  \label{lem:closure_top_conju}
  In the setting as above, let $\simhat$ be the closure of $\sim$;
  $\widehat{\approx}$ be the closure of $\approx$. Then 
  \begin{align*}
    &\mu/ \simhat\colon S/\simhat \to S/\simhat
    \intertext{is topologically conjugate to}
    &\varphi/\widehat{\approx}\colon S'/ \widehat{\approx} \to S'/
  \widehat{\approx}.
  \end{align*}
\end{lemma}

\begin{proof}
  We note first that the maps $\mu/\simhat$,
  $\varphi/\widehat{\approx}$ are well defined by Lemma
  \ref{lem:closure_dynamics}. From Lemma \ref{cor:hclusure} it
  follows that $\simhat$ is the pullback of $\widehat{\approx}$ by
  $h$. From (CE \ref{item:usc4}) it follows that
  $S'/\widehat{\approx}$ is a compact Hausdorff space. Applying Lemma
  \ref{lem:usc_from_map} to the map $S\xrightarrow{h} S' \to
  S'/\widehat{\approx}$ yields that $S/\simhat$ is homeomorphic to
  $S'/\widehat{\approx}$, where the homeomorphism is given by
  $\tilde{h}([s]_{\simhat}) = [h(s)]_{\widehat{\approx}}$. Write
  $\widetilde{\varphi}= \varphi/\widehat{\approx}$,
  $\tilde{\mu}=\mu/\simhat$, then
  \begin{equation*}
    \widetilde{\varphi}\circ \tilde{h}([s]_{\simhat}) 
    = \widetilde{\varphi}([h(s)]_{\widehat{\approx}})
    = [\varphi\circ h(s)]_{\widehat{\approx}}
    = [h\circ \mu(s)]_{\widehat{\approx}}
    = \tilde{h}([\mu(s)]_{\simhat})
    = \tilde{h} \circ \tilde{\mu}([s])_{\simhat}.
  \end{equation*}
  This
  finishes the proof.  
\end{proof}

\subsection{Proof of Lemma  \ref{lem:PwPb_zdsim}}
\label{sec:proof-lemma-top_conj}
We first show that $P_{\wt}\widehat{\mate}\, P_{\bt}$ is well defined.

Consider $\K_{\wt}\sqcup \K_{\bt}$, the disjoint union of $\K_{\wt},\K_{\bt}$. The
equivalence relation $\sim$ on $\K_{\wt}\sqcup \K_{\bt}$ is the one generated
by
\begin{align}
  \label{eq:defsimKwKb}
  &\sigma_{\wt}(z) \sim \sigma_{\bt}(\bar{z}), 
  \quad \text{for all $z\in S^1=\partial \D$ and }
  \\
  \notag
  &x\sim y, \quad \text{if } x,y \in \clos A_{\wt} \text{ or } x,y \in \clos
  A_{\bt},
\end{align}
for all $x,y \in \K_{\wt}$, or $x,y\in \K_{\bt}$. Here $A_{\wt},A_{\bt}$ are
bounded components of the Fatou sets of $P_{\wt},P_{\bt}$ and
$\sigma_{\wt}(z)\in \K_{\wt}\subset \K_{\wt}\sqcup \K_{\bt}$, $\sigma_{\bt}(\bar{z})\in
\K_{\bt}\subset \K_{\wt}\sqcup \K_{\bt}$. The map 
\begin{equation*}
  P_{\wt}\sqcup P_{\bt} \colon \K_{\wt}\sqcup \K_{\bt} \to \K_{\wt}\sqcup \K_{\bt},
\end{equation*}
is well defined. Furthermore the equivalence relation defined in
(\ref{eq:defsimKwKb}) is invariant with respect to $P_{\wt}\sqcup
P_{\bt}$. Let $\simhat$ be the closure of $\sim$, and
$\K_{\wt}\widehat{\mate}\, \K_{\bt}=\K_{\wt}\sqcup K_{\bt}/\simhat$. From Lemma
\ref{lem:closure_dynamics} it follows that $P_{\wt}\sqcup P_{\bt}$ descends to
this quotient, meaning that
\begin{equation*}
  P_{\wt}\widehat{\mate}\, P_{\bt}\colon \K_{\wt}\widehat{\mate}\, \K_{\bt} \to
  \K_{\wt}\widehat{\mate}\, \K_{\bt},
\end{equation*}
is well defined.

\smallskip
We
first show a one-sided version of Lemma \ref{lem:PwPb_zdsim}.
Identify the closure of each bounded Fatou component in $\K_{\wt}$ to form 
the quotient $\K_{\wt}/\F$. Recall that the Fatou set of $P_{\wt}$ is
separated. Since $P_{\wt}$ maps each bounded Fatou component to a bounded
Fatou component the quotient map
\begin{equation*}
  P_{\wt}/\F \colon \K_{\wt}/\F \to \K_{\wt}/\F 
\end{equation*}
is well defined. 
\begin{lemma}
  \label{lem:PwF_simFw}
  The map $P_{\wt}/\F$ as above is topologically conjugate to
  \begin{equation*}
    z^d\colon S^1/\Approx{\F,\wt} \to S^1/\Approx{\F,\wt}.
  \end{equation*}
\end{lemma}

\begin{proof}
  Consider the equivalence relation on $\K_{\wt}$, defined by
  ($x,y\in \K_{\wt}$)
  \begin{equation*}
    x\Approx{\F} y :\Leftrightarrow x,y \in \clos\F,
  \end{equation*}
  where $\F$ is a bounded Fatou component of $P_{\wt}$. 
  \begin{claim}
    $\Approx{\F}$ is closed.
  \end{claim}
  Consider two convergent sequences $x_n\to x_0$, $y_n\to y_0$ in
  $\K_{\wt}$, satisfying $x_n\Approx{\F} y_n$ (for all $n\geq 1$). We need
  to show that $x_0\Approx{\F}y_0$. This is clear when the sequence
  $(x_n)$ is contained in a single equivalence class of $\Approx{\F}$. 

  Assume now that each $x_n$ is contained in a distinct
  equivalence class, which we can assume to be non-trivial. This means
  that $x_n,y_n$ are contained in the closure of the same bounded
  Fatou component $A_n$. From the subhyperbolicity of $P_{\wt}$ it
  follows that $\diam A_n\to 0$, thus $x_0=\lim x_n=\lim y_n=y_0$
  proving the claim.

  \smallskip
  From the claim it follows, that $\K_{\wt}/\F$ is a compact Hausdorff
  space (see (CE \ref{item:usc4})). Clearly $\Approx{\F,\wt}$ is the
  equivalence relation (on $S^1$) induced by the map $S^1
  \xrightarrow{\sigma_{\wt}} \K_{\wt}\to \K_{\wt}/\F$. Thus we obtain from Lemma
  \ref{lem:usc_from_map} that $S^1/\Approx{\F,\wt}$ is homeomorphic to
  $\K_{\wt}/\F$. The topological conjugacy is clear, since $P_{\wt}$ maps each
  point $\sigma_{\wt}(z)\in \K_{\wt}$ to $\sigma_{\wt}(z^d)\in
  \K_{\wt}$ (for all $z\in\partial \D$).  
\end{proof}

\begin{lemma}
  \label{lem:simw_inv}
  The equivalence relation $\Sim{\wt}\vee \Sim{\bt}$ is invariant with
  respect to $z^d\colon S^1\to S^1$.
\end{lemma}

\begin{proof}
  From ($\LC^n$ \ref{item:propL6}) it follows that $\Sim{\infty,\wt} = \bigvee
  \Sim{n,\wt}$ is invariant with respect to $\mu$. Thus $\Sim{\wt}$ (being
  the closure of $\Sim{\infty,\wt}$) is invariant with respect to
  $z^d\colon S^1\to S^1$
  by Lemma~\ref{lem:closure_dynamics}. Similarly $\Sim{\bt}$ is
  invariant for this map. It is immediate that the join of two invariant
  equivalence relations is invariant.  
\end{proof}

\begin{proof}[Proof of Lemma \ref{lem:PwPb_zdsim}]
  %The Carath\'{e}odory semi-conjugacy $\sigma_{\wt}\colon S^1 \to \K_{\wt}$ is
  %not surjective.
  Let $\K_{\wt}/\F$ be as in the last lemma, $\K_{\bt}/\F$ the quotient
  obtained by identifying bounded Fatou components of $\K_{\bt}$. 
  Consider the equivalence relation $\simeq$ on
  $\K_{\wt}/\F \sqcup \K_{\bt}/\F$ generated by
  \begin{align*}
    [\sigma_{\wt}(z)] &\simeq [\sigma_{\bt}(\bar{z})], \text{ where}
    \\
    [\sigma_{\wt}(z)]&\in \K_{\wt}/\F \subset \K_{\wt}/\F \sqcup \K_{\bt}/\F
    \\
    [\sigma_{\bt}(\bar{z})]&\in \K_{\bt}/\F \subset \K_{\wt}/\F \sqcup \K_{\bt}/\F,
  \end{align*}
  for all $z\in S^1=\partial\D$. 
  Clearly $\simeq$ is invariant with respect to
  $P_{\wt}/\F\sqcup P_{\bt}/\F$. 
  Consider the map
  \begin{equation*}
    S^1 \sqcup S^1 \xrightarrow{[\sigma_{\wt}(z)], [\sigma_{\bt}(\bar{w})]} 
    \K_{\wt}/\F \sqcup \K_{\bt}/\F.
  \end{equation*}
  The pullback of $\simeq$ is $\Sim{\wt}\vee \Sim{\bt}$ (on each $S^1$),
  see Proposition \ref{prop:SimPFSimw}; it is invariant with respect
  to $\mu$ (Lemma \ref{lem:simw_inv}). Thus we have the following
  commutative diagram
  \begin{equation*}
    \xymatrix{
      (S^1\sqcup S^1,\Sim{\wt}\vee \Sim{\bt}) \ar[r]^{z^d} \ar[d]
      &
      (S^1\sqcup S^1,\Sim{\wt}\vee \Sim{\bt})  \ar[d]
      \\
      (\K_{\wt}/\F \sqcup \K_{\bt}/\F, \simeq) \ar[r]%_{P_{\wt}/\F\sqcup P_{\bt}/\F} 
      & (\K_{\wt}/\F \sqcup \K_{\bt}/\F,\simeq). 
    }
  \end{equation*}
  The quotient of $\K_{\wt}/\F \sqcup \K_{\bt}/\F$ with
  respect to the closure $\widehat{\simeq}$ is $\K_{\wt}\widehat{\mate}\,
  \K_{\bt}$; the quotient of the map $P_{\wt}/\F\sqcup P_{\bt}/\F$ is
  $P_{\wt}\widehat{\mate}\,P_{\bt}$ (see (\ref{eq:PwPbsimhat})). 

  The closure of $\Sim{\wt}\vee \Sim{\bt}$ is $\sim$, i.e., the
  equivalence relation induced by $\gamma$ (see
  Lemma~\ref{lem:Fnocrit_simwsimb}. Clearly $S^1\sqcup S^1/\!\sim\; =
  S^1/\sim$. Note that $\K_{\wt}/\F,\K_{\wt}/\F$ are metrizable
  (\cite[Proposition~2.2]{MR872468}).  The claim follows from Lemma
  \ref{lem:closure_top_conju}.
\end{proof}

The author believes that in general $\Sim{\wt} \vee \Sim{\bt} \,\neq\, \sim$,
in particular  $\Sim{\wt} \vee \Sim{\bt}$ will not be closed in
general. We do not present the examples that seem to indicate this
here. 

\section{$\gamma$ maps Lebesgue measure to measure of maximal   
  entropy}  
\label{sec:gamma-maps-lebesgue}

In this section we show that $\gamma$ maps Lebesgue measure on $S^1$ to
the measure of maximal entropy on $S^2$, i.e., prove Theorem
\ref{thm:gammaL1L2}. 

%The proof could be substantially simplified if we
%assumed that there are no periodic critical points. 

The \defn{measure of maximal entropy} $\nu$ for $F\colon S^2\to S^2$
may be constructed as the 
weak limit 
of $1/d^n \sum_{x\in F^{-n}(x_0)} \delta_x$, where $x_0\in S^2$ is
arbitrary ($F^{-n}(x_0)$ are the preimages of $x_0$ under $F^n$). 

We denote Lebesgue measure on the circle $S^1$ by $\abs{A}$, it is
assumed here to be normalized ($\abs{S^1}=1$). This is the measure of
maximal entropy of the map $\mu$ %$ := z^d\colon S^1 \to S^1$;
meaning it is the weak limit of
$\frac{1}{d^n}\sum_{w\in \mu^{-n}(z_0)}\delta_{\wt}$ for any $z_0\in S^1$.

\medskip
Let $x_0\in S^2\setminus \post$ be a point with the smallest number of
preimages by $\gamma$; $\gamma^{-1}(x_0)=\{t_1,\dots, t_N\}$ (see
Theorem \ref{thm:sim_finite}). Consider
preimages, 
$\{s_1,\dots, s_{dN}\}:= \mu^{-1}(\{t_1,\dots, t_N\})$, and
$\{w_1,\dots, w_d\}:= F^{-1}(x_0)$. By the commutativity of the
diagram from Theorem \ref{thm:main}, it follows that
$\gamma(\{s_1,\dots,s_{dN}\})=\{w_1,\dots, 
w_d\}$ and $\gamma^{-1}(\{w_1,\dots, w_d\})=\{s_1,\dots,s_{dN}\}$. By
the minimality of $x_0$ it follows that $\#\gamma^{-1}(w_j) \geq
N$. Thus $\#\gamma^{-1}(w_j)=N$ for all $j$. The same argument yields
that for all $w\in F^{-n}(x_0)$ there are $N$ points in
$\mu^{-n}(\{t_1,\dots,t_N\})$ that are mapped by $\gamma$ to
$w$.      

\smallskip
Thus the (probability) measure
\begin{align*}
  \frac{1}{Nd^n}\sum_{s\in \mu^{-n}\{t_1,\dots,t_N\}} \delta_s
  \quad
  \text{(on $S^1$)} 
  \intertext{ is mapped by $\gamma$ to }
  \\
  \frac{1}{d^n}\sum_{w\in F^{-n}(x_0)} \delta_{w} 
  \quad 
  \text{(on $S^2$)}.
\end{align*}
Clearly the first measure converges weakly to Lebesgue measure on
$S^1$, and the second measure converges weakly to the measure of
maximal entropy $\nu$ (of $F\colon S^2\to S^2$). This proves the theorem.

\section{Fractal tilings}
\label{sec:fractal-tilings}

From the invariant Peano curve $\gamma\colon S^1\to S^2$ one obtains
\emph{fractal tilings}. Indeed divide the circle $\R/\Z$ in $d$
intervals $[j/d,(j+1)/d]$ ($j=0,\dots d-1$). Since $\mu$ maps each
such interval onto $\R/\Z$ it follows from Theorem \ref{thm:main} that
$F$ maps each set $\gamma([j/d,(j+1)/d])$ to the whole sphere. 
The tiling lifts to the \emph{orbifold covering} (which is either the
Euclidean or the hyperbolic plane).
The
thus obtained tiles are illustrated for the example from Section
\ref{sec:laminations} in Figure \ref{fig:tiling_Lattes1}. 

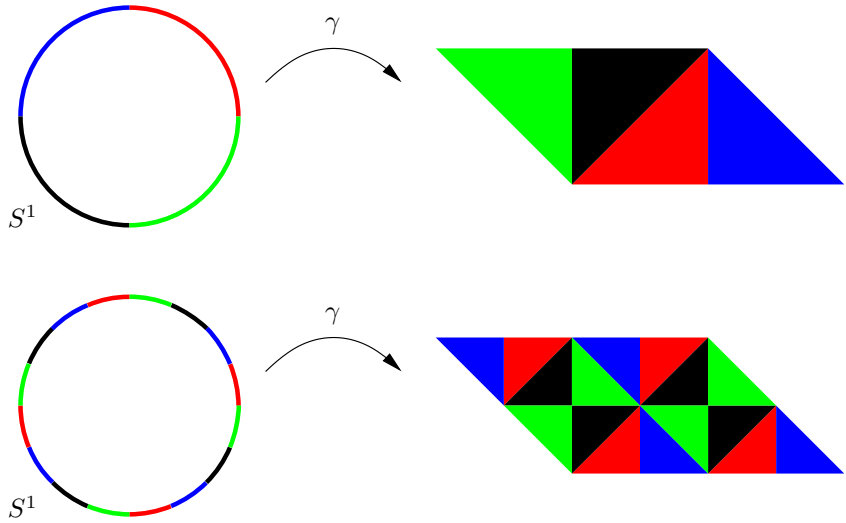
\begin{figure}
  \centering
  \includegraphics[width=11cm]{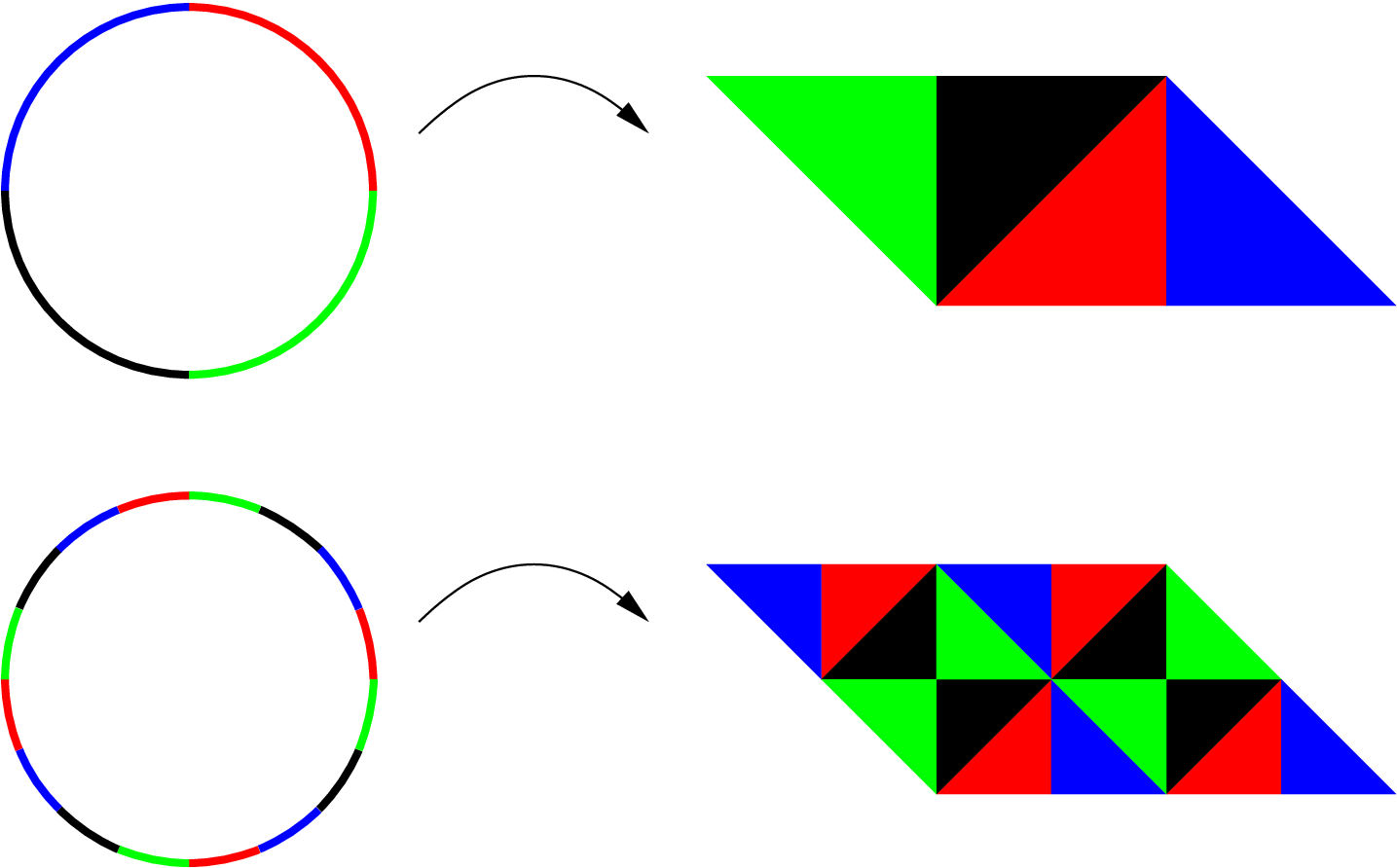}
  \begin{picture}(10,10)
    \put(-200,185){$\gamma$}
    \put(-200,75){$\gamma$}
    \put(-320,110){$S^1$}
    \put(-320,0){$S^1$}
  \end{picture}
  \caption{Tiling given by example from Section \ref{sec:laminations}} 
  \label{fig:tiling_Lattes1}
\end{figure}

This example is atypical however, since usually the tiles are very
fractal. We show the fractal tilings obtained from the Peano curve
$\gamma$ for two more examples.

The first is a Latt\`{e}s map whose orbifold has signature
$(2,3,6)$. It is the map $R_4=1-(3z+1)^3/(9z-1)^2$ from
\cite[Section 6.1]{snowemb}. 
The first approximation $\gamma^1$ of the
Peano curve is illustrated (in the orbifold covering) in Figure
\ref{fig:R4g1}. Tiles given by the resulting Peano curve are
illustrated in Figure \ref{fig:tiles_R4}. The two critical portraits
(i.e., the equivalence classes of $\Sim{1,\wt}, \Sim{1,\bt}$) that
describe $R_4$ according to Theorem \ref{thm:FcriticalPortraits} 
are:
\begin{align*}
  &\text{white portrait: } \left\{\frac{1}{9}, \frac{4}{9},
    \frac{7}{9}\right\}, 
  \\
  &\text{black portrait: } \left\{\frac{1}{3}, \frac{2}{3}\right\},
  \left\{\frac{1}{6}, \frac{5}{6}\right\}.
\end{align*}

\begin{figure}
  \centering
  \includegraphics[width=4cm]{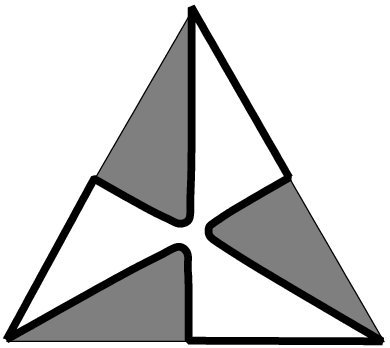}
  \caption{First approximation $\gamma^1$ for $R_4$.}
  \label{fig:R4g1}
\end{figure}

\begin{figure}
  \centering
  \subfigure[First order Tiles.]
  {
    \includegraphics[width=5.5cm]{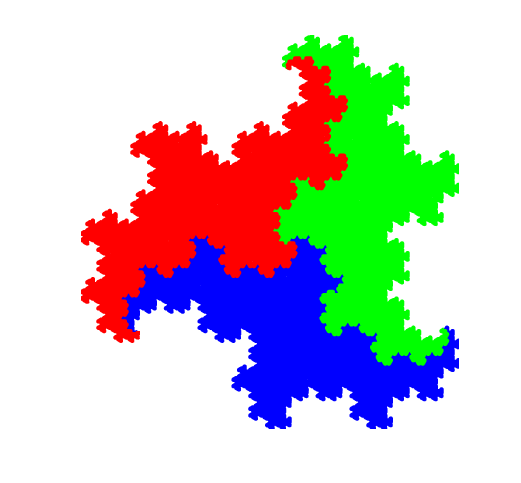}
    \label{fig:tiles1_R4}
  }
  \subfigure[Fourth order tiles.]     
  {
    \includegraphics[width=5.5cm]{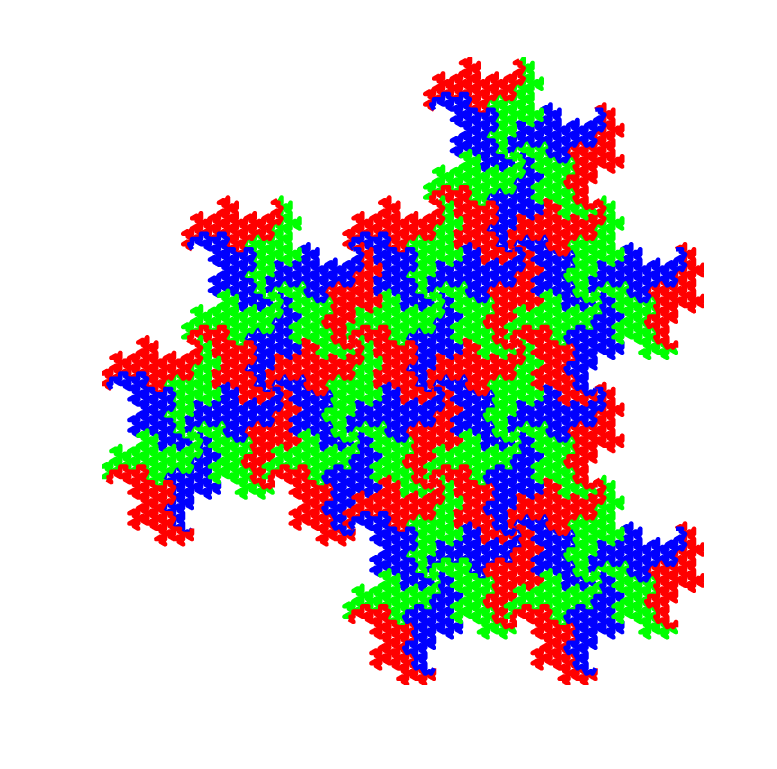}
    \label{fig:tiles2_R4}
  }     
  \caption{Tilings induced by the Peano curve of $R_4$.}
  \label{fig:tiles_R4}
\end{figure}

\smallskip
The second example is the map $R_2= 1-2
(z-1)(z+3)^3/((z+1)(z-3)^3)$ (see \cite[Section 6.1]{snowemb}). It is
a Latt\`{e}s map whose orbifold has signature $(3,3,3)$. The first
approximation $\gamma^1$ is shown (in the orbifold covering) in Figure
\ref{fig:R2g1}. The tiles that are obtained from the resulting
invariant Peano curve are shown in Figure \ref{fig:tiles_R2}. The two
critical portraits describing $R_2$ are: 
\begin{align*}
  &\text{white portrait: } \left\{\frac{7}{60}, \frac{22}{60},
    \frac{37}{60}\right\}, \left\{\frac{43}{60}, \frac{58}{60}\right\},  
  \\
  &\text{black portrait: } \left\{\frac{15}{60}, \frac{30}{60},
    \frac{45}{60}\right\},
  \left\{\frac{13}{60}, \frac{58}{60}\right\}. 
\end{align*}

\begin{figure}
  \centering
  \includegraphics[width=5cm]{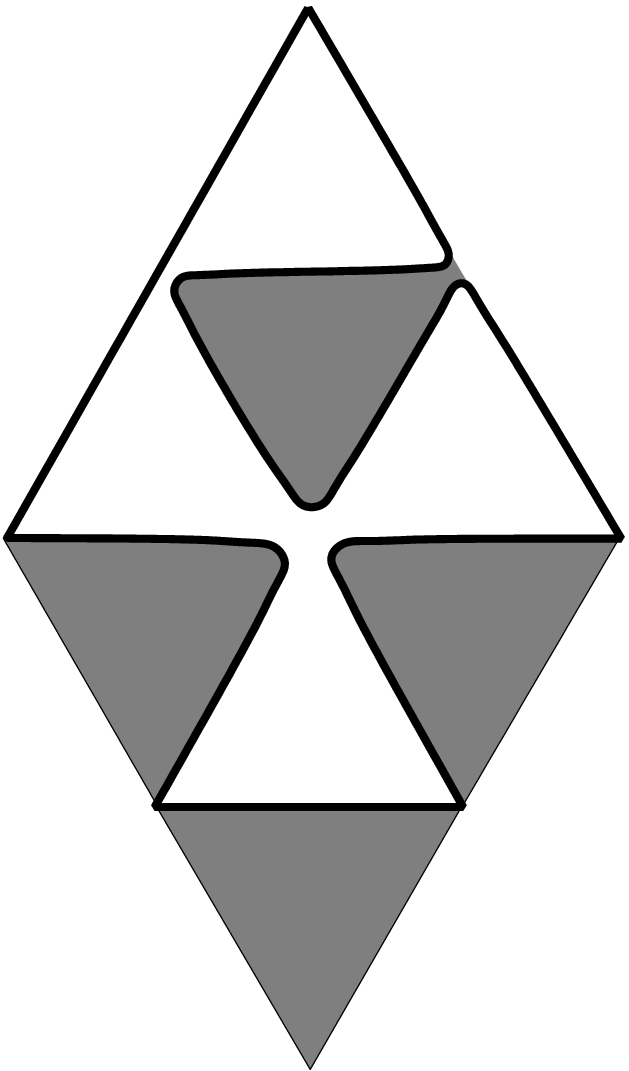}
  \caption{First approximation $\gamma^1$ for $R_2$.}
  \label{fig:R2g1}
\end{figure}

\begin{figure}
  \centering
  \subfigure[First order tiles.]
  {
    \includegraphics[width=5.5cm]{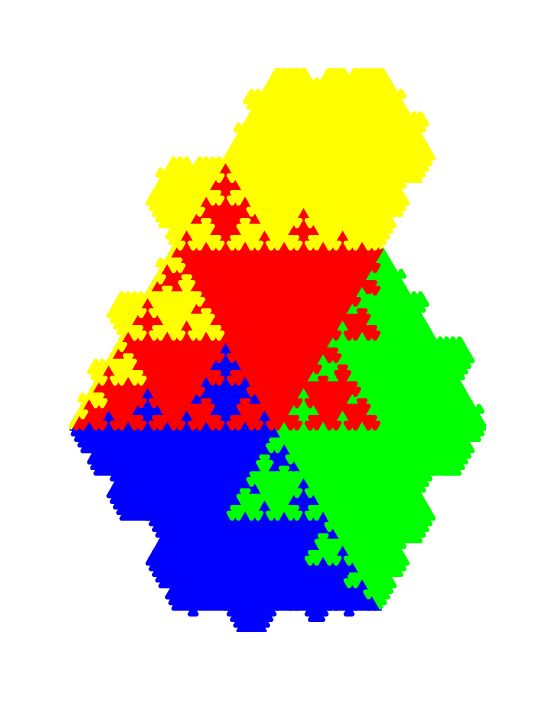}
    \label{fig:tiles1_R2}
  }
  \subfigure[Third order tiles.]
  {
    \includegraphics[width=5.5cm]{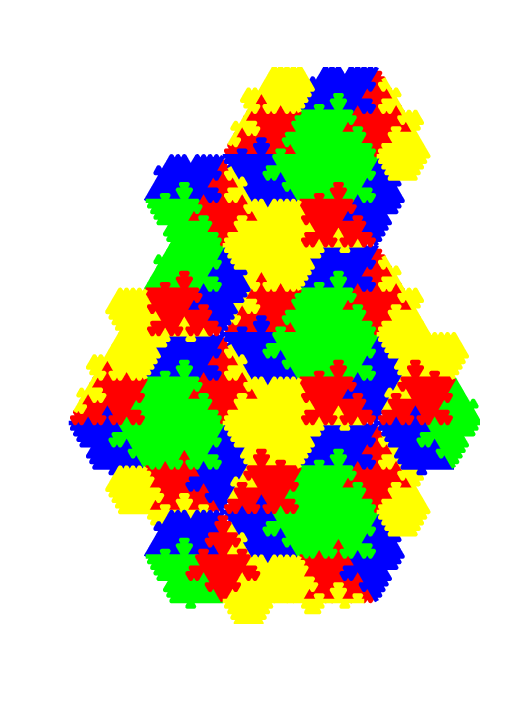}
    \label{fig:tiles2_R2}
  }
  \caption{Tilings induced by the Peano curve of $R_2$.}
  \label{fig:tiles_R2}
\end{figure}

\smallskip
All examples considered above had \emph{parabolic orbifold}. Consider
a \emph{rational} expanding Thurston map (meaning it has no
\emph{Thurston obstruction}) with \emph{hyperbolic orbifold}. The
tiling obtained from the invariant Peano curve lifts to the
\emph{orbifold cover}, i.e., the hyperbolic plane. Thus one obtains
fractal tilings of the hyperbolic plane with interesting self-similar
properties. 

There are other ways to obtain fractal tilings from the invariant
Peano curve $\gamma$. Instead of dividing the circle into $d^n$
intervals of the same length, we can take the images of the $n$-arcs
by $\gamma$. Thus we get tilings of the hyperbolic/Euclidean plane
with $k$($=\#\post$) different tiles. Each tile divides into tiles of
the $(n+1)$-th order. 

There is yet another way to obtain tilings from the invariant Peano
curve $\gamma$ in a natural way. Namely define tiles as the images of
(either white or black) $n$-gaps by $\gamma$.

\section{Open Questions}
\label{sec:open-questions}

The construction presented here to decompose, or unmate, an expanding
Thurston map into polynomials is not the most general one. In
\cite{unmating} an example of an expanding Thurston map (which is
rational) was given that arises indeed as the (topological) mating of
two poynomials, yet this cannot be shown with the methods presented
here.

\begin{open}
  Let $f\colon S^2\to S^2$ be a Thurston map. Give a necessary and
  sufficient condition that $f$ arises as a mating. 
\end{open}

Ideally this condition should give all shared matings, i.e., all
distinct possibilities how $f$ arises as a mating. Furthermore from
one should be able to read off the polynomials into which $f$ unmates
in a combinatorial manner. For \emph{hyperbolic} rational Thurston
maps such a necessary and sufficient condition to arise as a mating is
known, namely the existence of an \emph{equator}, see
\cite[Theorem~4.2]{unmating}. 

Several people have asked whether there is a bound on the number of
points that are identified in a mating. If a mating (of strictly
preperiodic polynomials) is obtained as constructed here this is
answered by Theorem~\ref{thm:sim_finite}. 

\begin{open}
  Is it possible to decide whether an expanding Thurston map $F$ is 
  equivalent to a rational map from the critical portraits (see
  Section \ref{sec:critical-portrait})? By Thurston's topological
  characterization \cite{DouHubThurs} this amounts to the question
  whether it is 
  possible to read off \defn{Thurston obstructions} from the
  \emph{critical  portraits}. 

  In principle this is possible. Recall
  that each $1$-tile/$1$-edge has a natural corresponding
  $1$-gap/$1$-arc. Thus every multicurve in $S^2\setminus \post$ can
  be naturally represented in the ``critical portrait sphere''
  $\widetilde{S}^2$ (i.e.,
  the sphere whose two hemispheres are $S^2_{\wt}$ and $S^2_{\bt}$ as in
  Section \ref{sec:laminations-2}). A multicurve $\Gamma$ in this picture is
  just a multicurve in $\widetilde{S}^2\setminus \A^0$. Since each
  $1$-gap is mapped by $\mu$ to $S^2_{\wt}$ or $S^2_{\bt}$, it is possible to take
  the preimage of the multicurve. It is invariant if each component of
  the preimage is isotopic rel.\ $\A^0$ to one component of
  $\Gamma$. The Thurston matrix is then taken as usual.
  
  However it is not clear whether the description above offers any
  advantage in finding Thurston obstructions.  
\end{open}

\begin{open}
  Consider a postcritically finite rational map $\tilde{f}$ whose \emph{Julia
    set is a 
  Sierpi\'{n}ski carpet}. Identifying the closure of each Fatou
  component yields an expanding Thurston map $f$ (see Section
  \ref{sec:expand-thurst-maps}). Assume $f$ has an invariant Peano
  curve $\gamma$ (meaning we do not have to take an iterate in Theorem
  \ref{thm:main}). Is it possible to construct from $\gamma$ a
  semiconjugacy 
  $\tilde{\gamma}\colon S^1 \to \widetilde{\J}$ ($\widetilde{\J}$ is
  the Julia set of $\tilde{f}$) such that
  $\tilde{f}(\tilde{\gamma}(z))=\tilde{\gamma}(z^d)$ for all $z\in
  S^1$ (where $d=\deg
  \tilde{f}$)? This is false in general (see 
  \cite[Section 4]{MR1961296}), but possibly true under some additional
  assumptions.  
\end{open}
    
% Literaturverzeichnis
\bibliographystyle{alpha}
\bibliography{main}

\end{document}